\documentclass[10pt,reqno]{amsart}
\usepackage[margin=2.3cm]{geometry}
\usepackage[english]{babel}
\usepackage[utf8]{inputenc}

\usepackage{graphicx}	
\graphicspath{{figures/}}			

\usepackage[dvipsnames]{xcolor}			
\usepackage{amsfonts, amsmath, amsthm, amssymb}
\usepackage{mathtools, mathrsfs}
\usepackage[shortlabels]{enumitem}
\usepackage{bigints}
\usepackage[mathscr]{euscript}
\usepackage[final]{pdfpages}
\usepackage{framed}
\usepackage[11pt]{moresize}
\usepackage[font=footnotesize,labelfont=bf]{caption}
\usepackage{subcaption}
\usepackage{tikz}
\usepackage[noadjust]{cite}
\usepackage{wrapfig}
\usepackage{bigints}
\usepackage{afterpage}
\usepackage{todonotes}
\usepackage{setspace}
\usepackage{appendix}
\usepackage{etoolbox}
\usepackage{xcolor}
\newcommand{\e}{\varepsilon}
\newcommand{\f}{\frac}
\newcommand{\im}{\implies}

\newcommand{\la}{\lambda}
\newcommand{\al}{\alpha}
\newcommand{\be}{\beta}

\newcommand{\p}{\partial}
\newcommand{\lra}{\longrightarrow}

\newcommand{\w}{\omega}

\newcommand{\Om}{\Omega}

\newcommand{\bu}{\mathbf{u}}

\newcommand{\bz}{\mathbf{z}}

\newcommand{\C}{\mathbb{C}}

\newcommand{\K}{\mathbb{K}}
\newcommand{\bbP}{\mathbb{P}}
\newcommand{\Q}{\mathbb{Q}}
\newcommand{\R}{\mathbb{R}}

\newcommand{\N}{\mathbb{N}}

\newcommand{\cA}{\mathcal{A}}
\newcommand{\cB}{\mathcal{B}}

\newcommand{\cD}{\mathcal{D}}

\newcommand{\cF}{\mathcal{F}}
\newcommand{\cG}{\mathcal{G}}
\newcommand{\cH}{\mathcal{H}}
\newcommand{\cI}{\mathcal{I}}
\newcommand{\cJ}{\mathcal{J}}

\newcommand{\cL}{\mathcal{L}}

\newcommand{\cN}{\mathcal{N}}
\newcommand{\cO}{\mathcal{O}}
\newcommand{\cP}{\mathcal{P}}

\newcommand{\cU}{\mathcal{U}}
\newcommand{\cV}{\mathcal{V}}

\newcommand{\sI}{\mathscr{I}}

\newcommand{\sM}{\mathscr{M}}
\newcommand{\scP}{\mathscr{P}}

\newcommand{\sP}{\mathsf{P}}
\newcommand{\sU}{\mathsf{U}}
\newcommand{\sD}{\mathsf{D}}

\newcommand{\bbR}{\mathbb{R}}
\newcommand{\bbC}{\mathbb{C}}

\newcommand{\fH}{\mathfrak{H}}
\newcommand{\fD}{\mathfrak{D}}

\newcommand{\tr}{\operatorname{tr}}
\newcommand{\T}{\operatorname{T}}

\newcommand{\sgn}{\operatorname{sign}}
\newcommand{\dist}{\operatorname{dist}}
\newcommand{\ran}{\operatorname{ran}}

\newcommand{\spec}{\operatorname{Spec}}
\newcommand{\esspec}{\operatorname{Spec_{\rm ess}}}

\newcommand{\dom}{\operatorname{dom}}
\newcommand{\spn}{\operatorname{span}}

\newcommand{\hatt}{\widehat}
\newcommand{\no}{\nonumber}

\DeclareMathOperator{\Mas}{Mas}

\newcommand{\de}[2]{\frac{d #1}{d #2 }}

\newtheorem{lemma}{Lemma}[section]

\usepackage[colorlinks=true,backref=page,citecolor=Orange,linkcolor={blue}]{hyperref}
\usepackage[capitalise,nameinlink]{cleveref}

\newtheorem{prop}[lemma]{Proposition}
\newtheorem{theorem}[lemma]{Theorem}

\newtheorem{hypo}[lemma]{Hypothesis}

\usepackage{scalerel} 

\theoremstyle{definition}
\newtheorem{rem}[lemma]{Remark}

\newlength\myheight  
\AtEndEnvironment{define}{\null\hfill\ensuremath{\scaleto{\diamond}{\myheight}}}

\crefname{rem}{Remark}{Remarks}
\crefname{conj}{Conjecture}{Conjectures}
\crefname{hypo}{Hypothesis}{Hypotheses}
\crefname{enumhypoi}{Hypothesis}{Hypotheses}
\crefname{theorem}{Theorem}{Theorems}
\crefname{cor}{Corollary}{Corollaries}
\crefname{prop}{Proposition}{Propositions}
\crefname{define}{Definition}{Definitions}

\AtBeginEnvironment{appendices}{\crefalias{section}{appendix}}

\definecolor{mitch}{rgb}{0.98,0.9,1}
\definecolor{skyblue}{rgb}{0.9,0.95,1}
\definecolor{purplemunsell}{rgb}{0.62, 0.0, 0.77}

\begin{document}
	\allowdisplaybreaks
	\title[]{Hadamard-type formulas for real eigenvalues of canonically symplectic operators}
	\author[M. Curran]{Mitchell Curran}
	\address{School of Mathematics and Statistics,
		The University of Sydney, Camperdown, NSW 2006, Australia}
	\email{mcur3925@sydney.edu.au}
	\author[S. Sukhtaiev]{Selim Sukhtaiev}
	\address{Department of Mathematics and Statistics,
		Auburn University, Auburn, AL 36849, USA}
	\email{szs0266@auburn.edu}
	\date{\today}
	\keywords{Eigenvalues, Maslov index, crossing forms, star graphs}
	\thanks{S.S. was
		supported in part by NSF grants DMS--2510344, DMS--2418900, Simons Foundation grant
		MP--TSM--00002897, grant no. $2024154$ from the U.S.-Israel Binational Science Foundation Jerusalem, Israel, and by the Office of the Vice President for
		Research \& Economic Development (OVPRED) at Auburn University through
		the Research Support Program grant.}

	\begin{abstract}
		We give first-order asymptotic expansions for the resolvent and Hadamard-type formulas for the eigenvalue curves of one-parameter families of canonically symplectic operators. We allow for parameter dependence in the boundary conditions, bounded perturbations and trace operators associated with each off-diagonal operator, and give formulas for derivatives of eigenvalue curves emanating from the discrete eigenvalue of the unperturbed operator in terms of Maslov crossing forms. We derive the Hadamard-type formulas using two different methods: via a symplectic resolvent difference formula and asymptotic expansions of the resolvent, and using Lyapunov-Schmidt reduction and the implicit function theorem. The latter approach facilitates derivative formulas when the eigenvalue curves are viewed as functions of the spectral parameter. We apply our abstract results to derive a spectral index theorem for the linearised operator associated with a standing wave in the nonlinear Schr\"odinger  equation on a compact star graph. 
	\end{abstract}
		\maketitle

	\parskip=0em
	\tableofcontents
	\setlength{\parskip}{3mm}
	\numberwithin{equation}{section}
	
	\section{Introduction}

Recently many works have used the Maslov index to study nonlinear PDEs. In this paper, we are motivated by two recent works \cite{CCLM23,LS20first}. The first \cite{LS20first} discusses \emph{Hadamard-type} formulas for the rates of change of eigenvalues of families of self-adjoint realizations of elliptic operators on bounded domains and abstract symmetric operators in Hilbert spaces. The second \cite{CCLM23} derives analogous formulas for real eigenvalues of a \emph{canonically symplectic} operator \cite{KKS04} arising in the stability analysis of standing waves in the nonlinear Schr\"odinger (NLS) equation on a compact interval. Here we extend the results of the latter by using the framework of the former: we study the variation of real eigenvalues of canonically symplectic operators where the off-diagonal blocks are families of self-adjoint extensions of densely-defined symmetric operators. Our results bridge the celebrated classical Rayleigh-Hadamard-Rellich theory for eigenvalue variation \cite{Rayleigh45,hadamard1908,Rellich69} with the Arnol'd-Maslov-Keller index theory from symplectic geometry \cite{Maslov65,A67,Keller58} for operators with the canonical symplectic structure. Since such operators are not self-adjoint, the monotonicity properties underpinning that setting are generally absent, leading to far richer spectral behaviour. Furthermore, as shown in \cite{CCLM23}, these local formulas can provide crucial information in determining global counts of eigenvalues. Before discussing our abstract setting and main results of the paper, we provide some examples, accompanied by numerical experiments, to motivate our analysis.

\subsection{Motivating examples}  \, 
\hspace{-2.5mm}{   Furthering the work done in \cite{CCLM23}, in what follows we analyse the real spectrum associated with standing waves in the NLS equation \eqref{NLS} on a compact star graph. The study of NLS on metric graphs has attracted much attention in recent years, see, for example, the reviews \cite{Noja14_open_problems,KNP22}. From a physical standpoint, such a model describes wave propagation in quasi one-dimensional structures; for example, pulse propagation in networks of optical waveguides \cite{BS11_Tshaped,HTM11_Hshaped,BIJ15_conductor}, and the dynamics of Bose–Einstein condensates (BECs) confined to thin traps \cite{BCMRRSV00,BCRSV01,BB01_boseeinstein} (see also \cite[\S7.6.6]{berk}). In this interpretation, the graph edges represent propagation channels obtained as the thin-domain limit of a multidimensional branched structure (see \cite{Kuchment02_graph_models,SBM17_uzbek_shrinking}, as well as \cite[\S7.5]{berk} and the references therein), while the vertex conditions encode transmission and scattering at junctions. While the cubic case ($p=1$ in \eqref{NLS}) arises naturally in many applications, e.g. pulse propagation in Kerr media and BEC dynamics, higher powers ($p>1$) model higher-order nonlinear responses, such as those encountered in non-Kerr optical media. Unlike infinite graphs, the compact metric graphs considered here may be viewed as idealized models of finite branched networks in which the wave remains confined to the structure. In the mathematical literature, a number of works have studied the existence, stability, and bifurcation of standing waves in the NLS equation on star graphs \cite{ACFN11_fast,ACFN12_stationary,ACFN14_variational,Khairzan19,KP18_shifted,KP18_half_solitons,RusinSusanto25} and on metric graphs of other topologies \cite{NPS15_tadpole,KPM21_flower,GW16_nonlinear_quantum,LLS18,Goodman19,MP16_dumbbell,NojaPel20}}. 

Let $\cG$ be a star graph with $m$ edges, where each edge is assigned a positive length and a direction. Edge $i$ has length $\ell_i\in(0,\infty)$ and is represented by the interval $[0,\ell_i]$, where $0$ corresponds to the central vertex and $\ell_i$ the free vertex. The Sobolev spaces of functions on $\cG$ are denoted by
\begin{align}
	&L^2(\cG)\coloneqq\bigoplus_{i=1}^m L^2(0,\ell_i), \qquad  \hatt{H}^k(\cG) \coloneqq \bigoplus_{i=1}^m H^k(0,\ell_i),\ k\in\N\no,
\end{align}  
where $H^k(0,\ell_i)$ is the standard $L^2$ based Sobolev space of order $k\in \N$ on the interval $(0,\ell_i)$. The boundary $\partial\cG$ is defined by
\begin{equation}
	\partial\cG\coloneqq \cup_{1\leq i \leq m} \{a_i,b_i\},
\end{equation}
where $a_i, b_i$ denote the endpoints of the $i$th edge. It is convenient to treat the $2m$ dimensional vectors representing the values of the functions at the endpoints of each edge as functions on the boundary $\partial\cG$, in particular, $L^2(\partial\cG)\cong \bbC^{2m}.$

On $\cG$ we consider the NLS equation with power nonlinearity:
\begin{equation}\label{NLS}
	i\Psi_t = \Psi_{xx} + |\Psi|^{2p} \Psi, \qquad x\in\cG, \quad t\in\R, \quad p>0,
\end{equation}
where $\Psi(x,t) = \left (\Psi_1(x,t), \dots, \Psi_m(x,t)\right )^\top\in\C^m$, $|\Psi|^{2p} \Psi \coloneqq\left (|\Psi_1|^{2p} \Psi_1, \dots, |\Psi_m|^{2p} \Psi_m\right )^\top$. A standing wave solution is given by
\begin{equation}\label{standingwave}
	\Phi(x,t) = e^{i\beta t}\phi(x), \qquad \beta \in \R, 
\end{equation}
where the wave profile $\phi(x) \coloneqq (\phi_1(x), \dots, \phi_m(x))^\top\in\R^m$ solves the standing wave equation (defined edge-wise)
\begin{equation}\label{SWE}
	\phi'' + \be \phi + \phi^{2p+1} = 0, \qquad \phi^{2p+1}\coloneqq(\phi_1^{2p+1}, \dots, \phi_m^{2p+1})^{\top}.
\end{equation} 
We let $\Phi$ satisfy Dirichlet conditions at the free vertices and $\delta$\emph{-type} conditions at the central vertex,
\begin{align}\label{vertex_conds_1}
	\begin{cases}
		&\phi_1(\ell_1) = \dots = \phi_m(\ell_m) =0, \\
		& \phi_1(0) = \dots = \phi_m(0), \\
		&\sum_{k=1}^{m} \phi_k'(0)= \al  \phi_1(0), \qquad \al\in\R,
	\end{cases}
\end{align}
and note that when $\al=0$ one has the \emph{Neumann-Kirchhoff} condition at the central vertex.

Assuming a solution to \eqref{SWE} -- \eqref{vertex_conds_1} exists\footnote{See \cref{rem:existence}.}, linearising \eqref{NLS} about such a solution using a complex-valued perturbation satisfying the same vertex conditions leads to the eigenvalue problem
\begin{equation}\label{1.6}
	(\cN + V)\mathbf{u} = \la  \mathbf{u}, \qquad \mathbf{u}=(u, v)^{\top}\in \dom \cN,
\end{equation}
where $\cN$ and $\cV$, acting in $L^2(\cG) \oplus L^2(\cG)$, are defined as follows
\begin{subequations}\label{Nops}
	\begin{align}
		\cN &= \begin{pmatrix}
			0 & -\cA \\
			\cA & 0
		\end{pmatrix}, \quad \begin{cases}
			\cA = -\p_{xx}, \quad \dom \cA  = \left \{ u \in \hatt{H}^2(\cG): \tr u \in \cL  \right \},  \label{Nop1}\\
			\dom \cN = \dom \cA\oplus \dom\cA = \{ \mathbf{u}\in \hatt{H}^2(\cG)\oplus \hatt{H}^2(\cG): \T \mathbf{u} \in \cL\oplus \cL \},
		\end{cases} \\
		V &= \begin{pmatrix}
			0 & -V_- \\ V_+ & 0
		\end{pmatrix}, \qquad
		\begin{cases}
			V_-(x) = -\phi(x)^{2p} -\be, \\
			V_+(x) = -(2p+1) \phi(x)^{2p} -\be.
		\end{cases}    \label{Nop2}
	\end{align}
\end{subequations}
{   In the NLS stability literature, see e.g. \cite[\S4.2]{pelinovsky}, \cite[\S7.1.3]{KapProm}, the off-diagonal operators are classically denoted by $L_\pm$; in our notation, $L_- = \cA + V_-$ and $L_+ = \cA + V_+$.} In \eqref{Nop1} $\tr$ is a \emph{trace} operator on $\hatt{H}^2(\cG)$ that maps a function to the value of the function and its first derivative on the boundary $\p\cG$, 
\begin{equation}
	\tr : \hatt{H}^2(\cG) \to L^2(\partial\cG) \times L^2(\partial\cG), \quad \tr u = ( \Gamma_{0} u, \Gamma_{1} u)^\top,
\end{equation}
where, for $u = (u_{1}, \dots, u_{m})^\top \in  \hatt{H}^2(\cG)$,
\begin{align}
	\begin{split}\label{Gamma_01_NLS}
		\Gamma_0 &: \hatt{H}^2(\cG) \to L^2(\partial\cG), \qquad \Gamma_0 u = (u_{1}(0), \dots, u_{m}(0), u_{1}(\ell_1),\dots, u_{m}(\ell_m) )^\top, \\
		\Gamma_1 &: \hatt{H}^2(\cG) \to L^2(\partial\cG), \qquad \Gamma_1 u = (u_{1}'(0), \dots, u_{m}'(0), -u_{1}'(\ell_1),\dots, -u_{m}'(\ell_m) )^\top.
	\end{split}
\end{align} 
The operator $\T$ in \eqref{Nop1} is then the following trace operator on $\big(\hatt{H}^2(\cG)\big)^2 = \hatt{H}^2(\cG)\oplus \hatt{H}^2(\cG)$,
\begin{equation*}
	\T : \big(\hatt{H}^2(\cG)\big)^2 \to \big(L^2(\partial\cG) \big)^4, \quad \T \mathbf{u} \coloneqq \tr u \oplus \tr v = ( \Gamma_{0} u, \Gamma_{1} u, \Gamma_{0} v, \Gamma_{1} v)^\top, \quad \mathbf{u} = (u,v)^\top
	\in \big(\hatt{H}^2(\cG)\big)^2.
\end{equation*}
Finally, $\cL$ is the Lagrangian subspace of the symplectic boundary space $\bbC^{2m}\times \bbC^{2m}$ which encodes the vertex conditions \eqref{vertex_conds_1}.

The operator $\cN$ is said to have the canonical symplectic structure \cite{KKS04}. Its spectrum, consisting solely of eigenvalues due to the compactness of the domain, thus has four-fold symmetry in the complex plane. 
{   To show that the wave is spectrally unstable, it suffices to establish the existence of eigenvalues with nonzero real part (as at least one of them would have positive real part due to the four-fold symmetry of the spectrum).}
While classical results \cite{GSS87,Grill88,GSS90} are relevant in the present context, see also \cite{Pel05,pelinovsky,KKS04,KKS05,KP12art}, here we opt for a symplectic approach, which is related to the study of \emph{Hadamard-type} variational formulas for parameter-dependent eigenvalues in an effort to further recent studies of the latter \cite{CCLM23,LS20first}. Namely, we consider the problem on a family of shrinking subdomains, parameterised by $t\in(0,1]$, and study the spectrum of the problem on the full domain when $t=1$ by analysing the $t$-dependent spectrum when $\la=0$ and $t\in(0,1]$. This approach has been used to study the spectra of ordinary and partial differential operators in many works, for example, \cite{Smale65,DJ11,JLM13,CJLS14,LSHad17}.

To that end, let  $\cG_t$, $t\in(0,1]$, be the star graph with edges $[0,t\ell_i]$. Restricting the eigenvalue problem for the operator \eqref{1.6}, \eqref{Nop1}, \eqref{Nop2} to $\cG_t$ and then rescaling back to $\cG$ yields the following $t$-dependent family of eigenvalue problems on $\cG$,
\begin{equation}\label{Nt_EVP_NLS}
	(\cN_t + V_t) \mathbf{u} = t^2 \la \mathbf{u}, \qquad \mathbf{u}=(u,v)^\top \in \dom \cN_t,
\end{equation}
where
\begin{subequations}\label{Nt_stargraphs}
	\begin{align}
		\cN_t &= \begin{pmatrix}
			0 & -\cA_t \\
			\cA_t & 0
		\end{pmatrix}, \quad \begin{cases}
			\cA_t = -\p_{xx}, \quad \dom \cA_t  = \left \{ u \in \hatt{H}^2(\cG): \tr_t u \in \cL  \right \},  \label{Ntop1}\\
			\dom \cN_t = \dom \cA_t\oplus \dom\cA_t = \{ \mathbf{u}\in \big(\hatt{H}^2(\cG)\big)^2: \T_t \mathbf{u} \in \cL\oplus \cL \},
		\end{cases} \\
		V_t &= \begin{pmatrix}
			0 & -V_{-,t} \\ V_{+,t} & 0
		\end{pmatrix}, \qquad
		\begin{cases}
			V_{-,t}(x) = -t^2\phi(tx)^{2p} -t^2\be, \\
			V_{+,t}(x) = -t^2(2p+1) \phi(tx)^{2p} -t^2\be,
		\end{cases}    \label{Ntop2}
	\end{align}
\end{subequations}
and
\begin{subequations} \label{Tt_trt_NLS}
	\begin{align}
		\tr_t : &\hatt{H}^2(\cG) \to L^2(\partial\cG) \times L^2(\partial\cG), \quad &&\tr_t u \coloneqq \left ( \Gamma_{0} u, \f{1}{t}\Gamma_{1} u\right )^\top,   \\
		\T_t : &\big(\hatt{H}^2(\cG)\big)^2 \to \big(L^2(\partial\cG) \big)^4, \quad &&\T_t \mathbf{u} \coloneqq \tr_t u_1 \oplus \tr_t u_2 = \left ( \Gamma_{0} u_1, \f{1}{t}\Gamma_{1} u_1, \Gamma_{0} u_2, \f{1}{t}\Gamma_{1} u_2\right )^\top. \label{Trace_t_NLS} 
	\end{align}
\end{subequations}

{  
Our result relating to this example is stated under two nondegeneracy assumptions on the linearised operators. First note that $\phi \in \ker(\cA +V_-)$, since by \eqref{SWE} we have
\begin{equation}
	(\cA+V_-) \phi = \phi''+\phi^{2p+1}+\be \phi =0,
\end{equation}
while $\phi\in\dom(\cA)$ since $\tr\phi = (\Gamma_0 \phi, \Gamma_1\phi)^{\top} \in\cL$ (cf. \eqref{vertex_conds_1}). We impose the following simplicity and invertibility conditions.
\begin{hypo}\label{hypo:NLS_simplicity}
	We assume that $\ker(\cA+V_-) = \spn \{\phi\}$ and $\ker(\cA+V_+) = \{0\}$.
\end{hypo}
When $\phi$ is nontrivial on all $m$ edges and $m\geq3$, simplicity of the zero eigenvalue of $\cA+V_-$ is equivalent to $\phi_1(0)\neq 0$. Indeed, any $u\in\ker(\cA+V_-)$ satisfies $u_i=a_i\phi_i$ on edge $i$, since $u_i$ and $\phi_i$ solve the same second-order equation and both satisfy Dirichlet conditions at the free vertices. Hence, if $\phi_1(0)\neq0$, then continuity at the central vertex forces $a_1=\cdots=a_m$, so $u\in\spn\{\phi\}$. Conversely, if $\phi_1(0)=0$, then continuity is automatic, while the $\delta$-type condition reduces to the single linear constraint $\sum_{i=1}^m a_i\phi_i'(0)=0$, where $\phi_i'(0)\neq0$ on each edge. Thus, in this case $\dim\ker(\cA+V_-)=m-1\geq 2$.}

Under \cref{hypo:NLS_simplicity}, zero is geometrically simple as an eigenvalue of $\cN+V$ with domain \eqref{Nop1}, with eigenfunction $\pmb{\phi} \coloneqq (0,\phi)^{\top}$. The following theorem gives a lower bound for the real spectral index via the associated \emph{eigenvalue curves} near $\la=0$, extending \cite{CCLM23} from compact intervals with Dirichlet boundary conditions to star graphs with $\delta$-type conditions at the central vertex and Dirichlet conditions at the free vertices.

\begin{theorem}\label{thm:spectral_index_thm}
	Let $\cN_t+V_t$ be the operator defined by \eqref{Nt_stargraphs} for $t\in(0,1]$, associated with the standing wave solution \eqref{standingwave}--\eqref{vertex_conds_1} to \eqref{NLS}. (When $t=1$, we drop the subscript, i.e. $\cN \coloneqq \cN_1, V\coloneqq V_1$.) Under \cref{hypo:NLS_simplicity}, there exists a $C^2$ curve $t(\la)$, defined for $0<|\la|\ll 1$, satisfying
	\begin{equation}\label{spectral_set_t_graph_of_lambda}
		t(\la)^2\la \in \spec\left (\cN_{t(\la)} + V_{t(\la)} \right ),
	\end{equation} 
	and such that $t(0)=1$ and $t'(0)=0$. Furthermore, if $t''(0) \neq 0$, then the number of positive real eigenvalues of $\cN +V$ satisfies the lower bound
	\begin{equation}\label{lower_bound_pos_real_eigvals}
		n_+(\cN+V) \geq | p_c - q_c - \mathfrak{c}|,
	\end{equation}
	where $p_c$ and $q_c$ are counts of \emph{conjugate points} defined as
	\begin{align}\label{pcqc}
		\hspace{3em} p_c \coloneqq \sum_{t_0\in(0,1)} \dim \ker(\cA_{t_0} + 
		V_{+,{t_0}}), \qquad q_c \coloneqq \sum_{t_0\in(0,1)}  \dim \ker(\cA_{t_0} + V_{-,{t_0}}),
	\end{align}
	and $\mathfrak{c}$ is determined by the concavity of $t(\la)$ at $\la=0$ as follows,
	\begin{equation}\label{c_NLS}
		\mathfrak{c} = \begin{cases}
			+1  &t''(0) < 0, \\
			0  &t''(0) > 0. 
		\end{cases}
	\end{equation}
\end{theorem}
\begin{rem}
The nondegeneracy condition $t''(0)\neq0$ is generic with respect to the parameters $\be,p,\phi_1(0), \phi_i'(0)$, $i=1, \dots,m-1$ in \eqref{SWE}--\eqref{vertex_conds_1}. (If $t''(0)=0$ then higher derivatives are needed to determine $\mathfrak{c}$ in \eqref{c_NLS}.) Regarding \eqref{pcqc}, the counts $p_c$ and $q_c$ of conjugate points may alternatively be characterised in terms of nontrivial intersections of Lagrangian planes, i.e. {  
\begin{equation}
	p_c = \sum_{t_0\in(0,1)} \dim(\K_{0,t_0} \cap (\{0\} \times \cL)), \quad q_c = \sum_{t_0\in(0,1)} \dim(\K_{0,t_0} \cap (\cL \times \{0\})),
\end{equation}}
where $\K_{\la,t}$ is the $t$- and $\la$-dependent \emph{Cauchy data plane} (for a precise definition see \cref{sec:Hadamard_crossingforms}).
\end{rem}

\begin{rem}\label{rem:t_greater_than_one}
We note that the $t$-dependent objects in \eqref{Nt_EVP_NLS}--\eqref{Trace_t_NLS} are well-defined for $t>1$ (which corresponds to stretching the graph $\cG$), hence we can extend the allowable $t$ values in \cref{thm:spectral_index_thm} to $(0,1+\e]$ for $\e>0$ small. Thus the eigenvalue curves are indeed well-defined for $1< t\leq 1+\e$ when $t''(0)>0$.
\end{rem}

\begin{rem}\label{corollary:VK}
A corollary of \cref{thm:spectral_index_thm} is the following \emph{Vakhitov-Kolokolov}-type (VK) criterion \cite{VK73,pelinovsky}, which furnishes a convenient numerical tool -- indeed one needs only local data at $(\la,t)=(0,1)$ -- to establish the existence of a positive real eigenvalue. {\it	Suppose $p_c=1$ and $q_c=0$. Let $t(\la)$ be the eigenvalue curve, given in \cref{thm:spectral_index_thm}, which passes through $(\la,t)=(0,1)$ and is defined for $0<|\la|\ll 1$. Suppose that $t''(0)\neq 0$. Then $\cN + V$ has a positive real eigenvalue if $t''(0) >0$, while $\spec(\cN + V) \subset i\R$ if $t''(0)<0$.} The first part of this assertion, that is, existence of unstable eigenvalues, follows from \cref{thm:spectral_index_thm}, while the second part can be shown as in \cite{CCLM23,pelinovsky}.
\end{rem}

\begin{figure}
\centering
\hspace*{\fill}
\subcaptionbox{\label{}} 
{\includegraphics[width=0.35\textwidth]{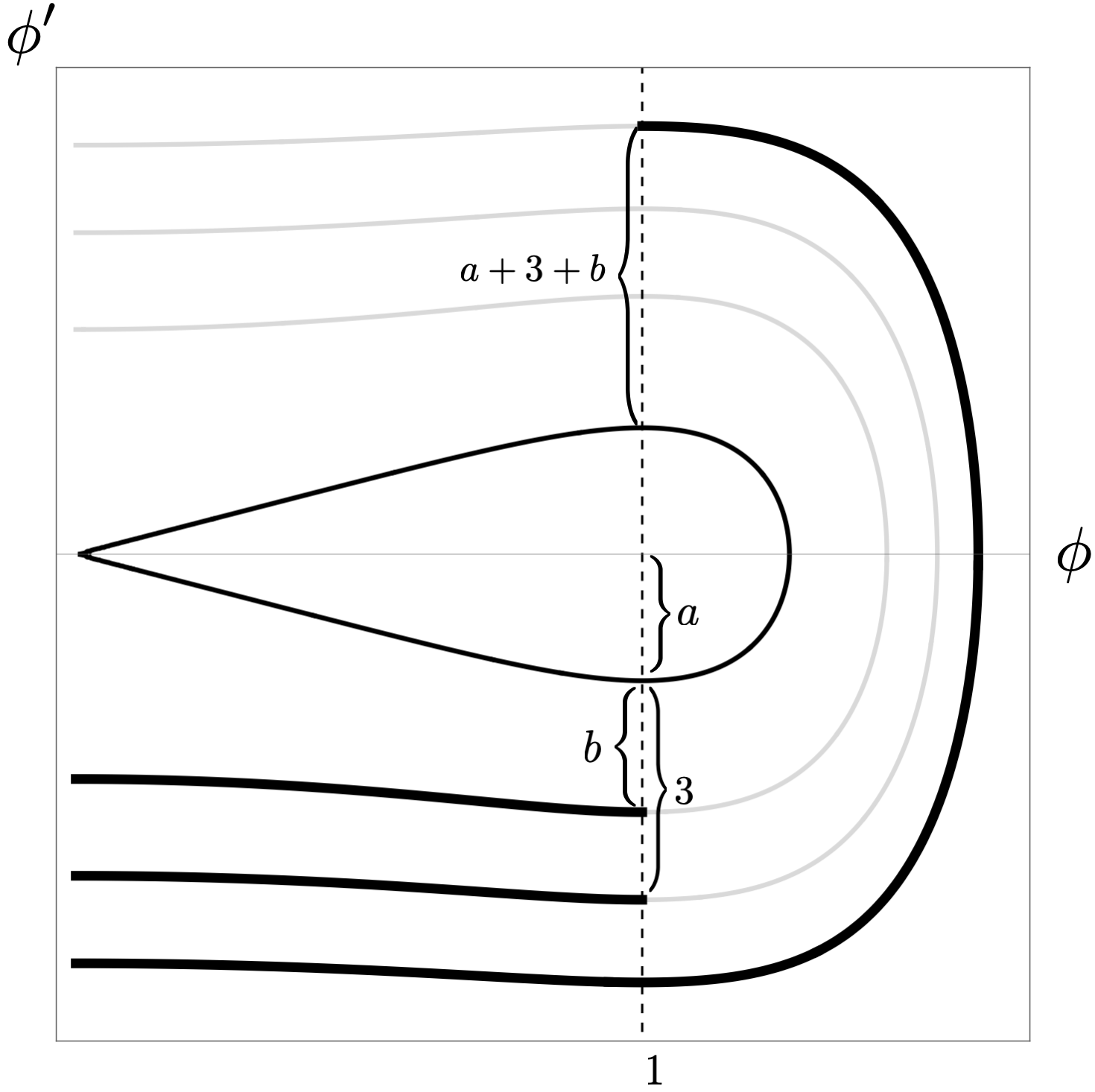}}\hfill 
\subcaptionbox{\label{}}
{\includegraphics[width=0.42\textwidth]{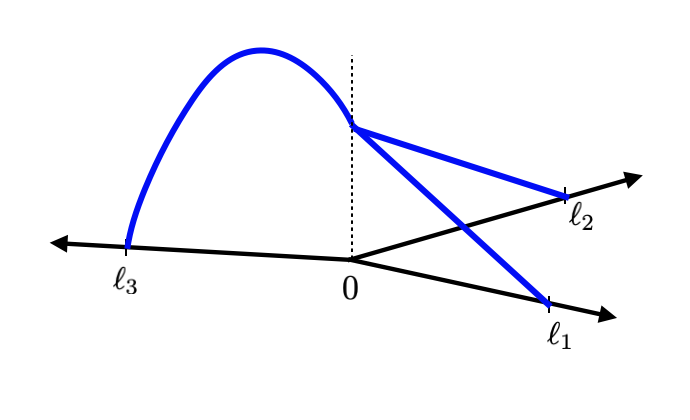}}
\hspace*{\fill} \\ \vspace{2mm}
\hspace*{\fill}
\subcaptionbox{$b=5$ \label{}} 
{\includegraphics[width=0.3\textwidth]{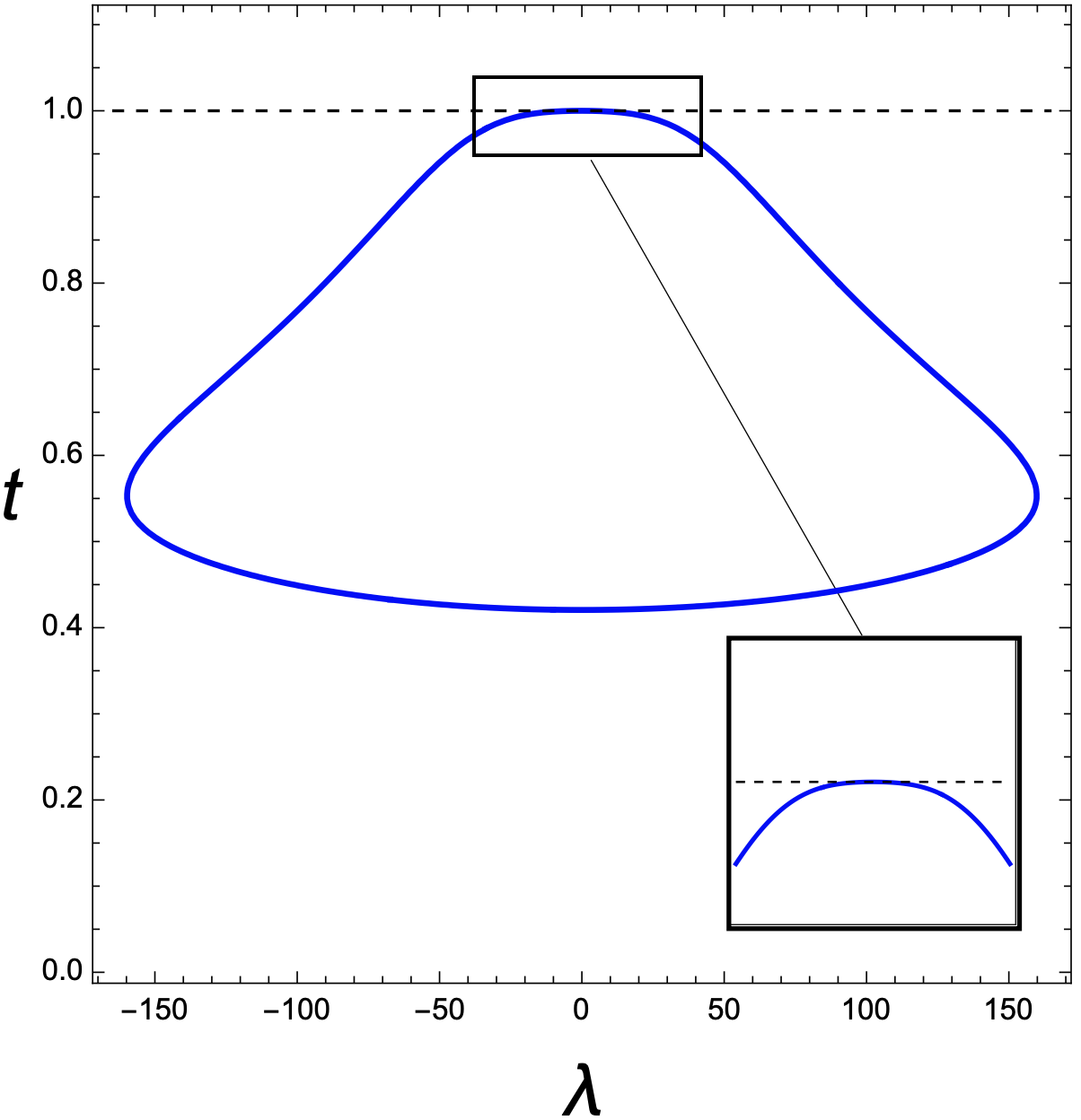}}\hfill 
\subcaptionbox{$b=3$ \label{}}
{\includegraphics[width=0.3\textwidth]{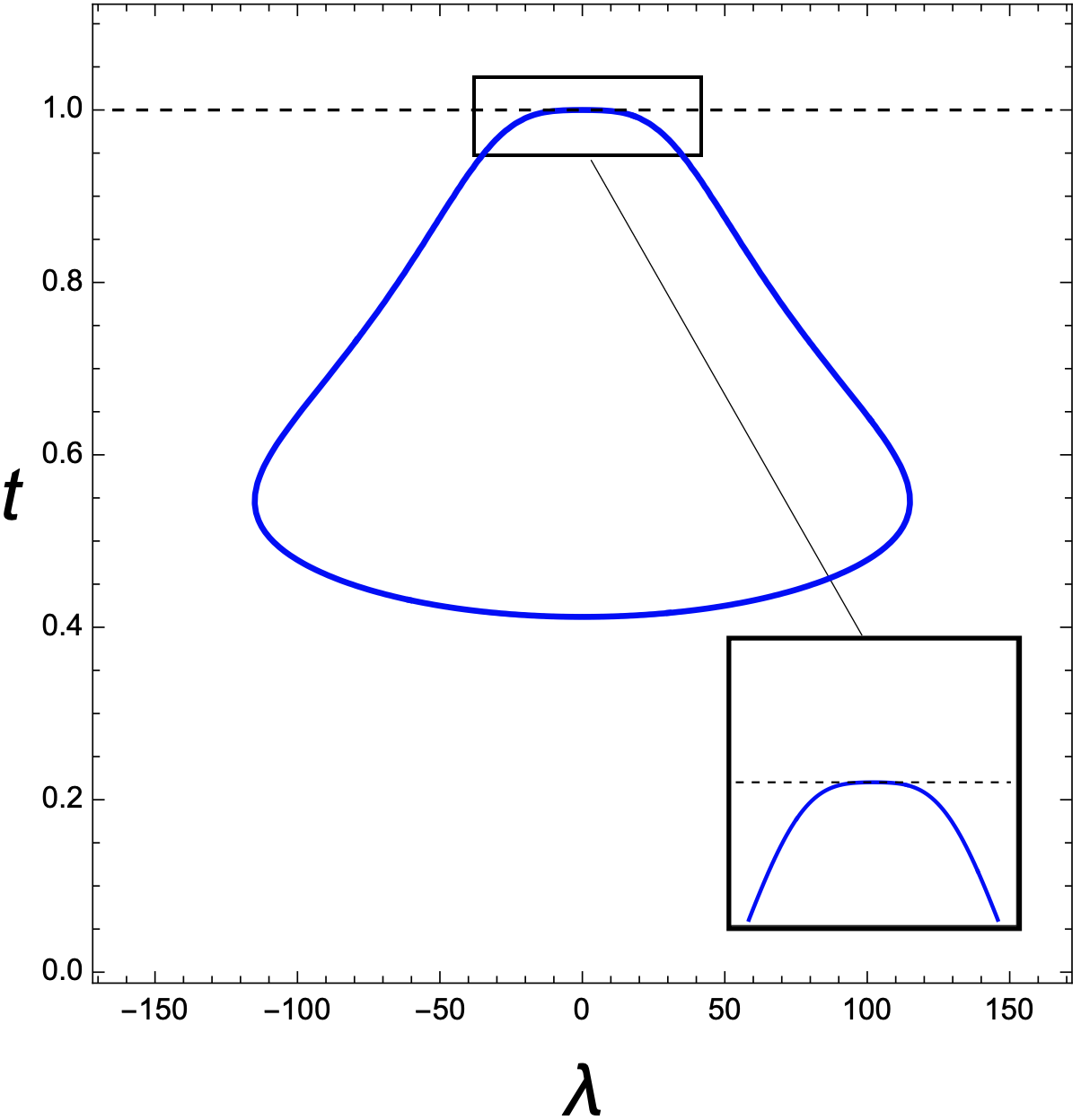}} \hfill
\subcaptionbox{$b=1$ \label{}}
{\includegraphics[width=0.3\textwidth]{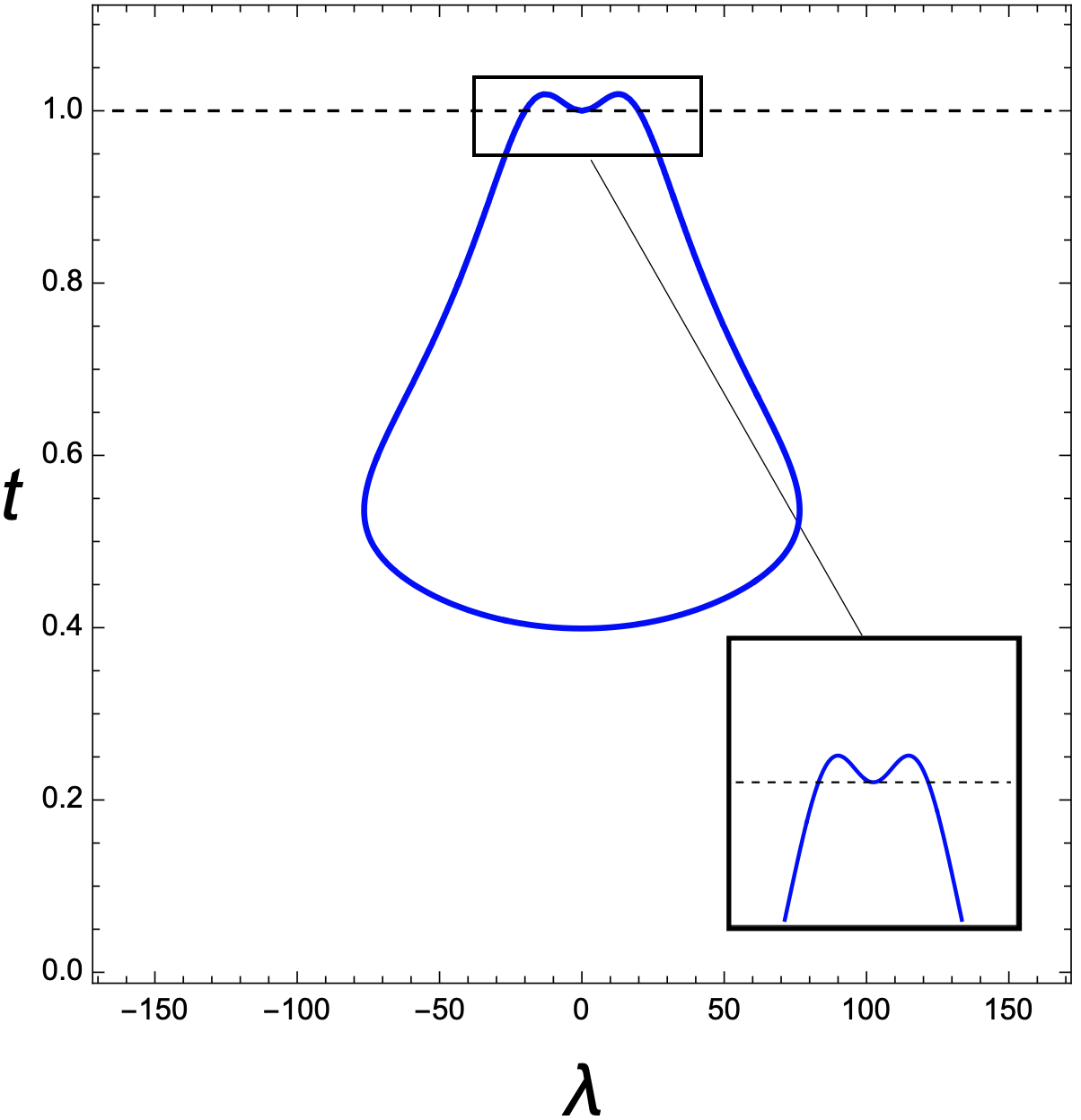}} 
\hspace*{\fill} 
\caption{(a) Phase curves (in bold) in the phase plane for \eqref{SWE} describing a standing wave solution $\phi = (\phi_1, \phi_2, \phi_3)^\top$ on a 3-star, where $\phi_1(0)=\phi_2(0)=\phi_3(0) = 1$ and $\phi_1'(0) = -a-b, \phi_2'(0) = -a-3,\phi_3'(0) = -\phi_1'(0) -\phi_2'(0) = 2a+3+b$, with $a\approx 0.8660$; (b) a schematic of the standing wave. In (c), (d) and (e) we give the eigenvalue curves for the standing wave solution described in (a) with $b$ as indicated. The standing waves all have one bump and two tails, and are non-negative on $\cG$.}
\label{fig:specral_pictures}
\end{figure}

In \cref{fig:specral_pictures} we give some numerical examples that highlight \cref{corollary:VK}. The figures showcase the real eigenvalue curves in the $\la t$-plane, representing the pairs $(\la,t)$ such that $t^2\la$ is an eigenvalue of $\cN_t+V_t$, for three standing wave solutions to the power NLS equation \eqref{NLS} with $p=3$, $\be=-1$, on the star graph with $m=3$ edges. The waves satisfy Neumann Kirchhoff conditions ($\al=0$) at the central vertex, in particular:
\begin{align}
\begin{cases}
	\phi_1(0)=\phi_2(0)=\phi_3(0) = 1, \\
	\phi_1'(0) = -a-b, \,\,\phi_2'(0) =-a-3,\,\,\phi_3'(0) = -\phi_1'(0) -\phi_2'(0) = 2a + 3 + b,
\end{cases}
\end{align}
where $a\approx0.8660$ and $b= 5,3,1$ in \cref{fig:specral_pictures} (c), (d) and (e) respectively. These initial conditions are integrated forward on each edge under \eqref{SWE} until the first instance at which the functions are zero. Thus the waves are non-negative on $\cG$, and all have one \emph{bump} (non-monotonic profile) and two \emph{tails} (monotonic profile) \cite{KNP22}. Integrating along the orbits in the phase plane associated with \eqref{SWE} shows that in all cases we have $\ell_2 \approx 0.262628$, while
\begin{align*}
&b=5 \im \ell_1 \approx 0.171588, \,\,\,\ell_3 \approx 0.399533, \\
&b=3 \im \ell_1 = \ell_2, \,\,\,\ell_3 \approx 0.467768, \\
&b=1 \im \ell_1 \approx 0.575249, \,\,\,\ell_3 \approx 0.573694,
\end{align*}
in (c), (d) and (e) respectively. The phase curves corresponding to the solutions on each edge, as well as a schematic of the standing wave, are given in (a) and (b), respectively, in \cref{fig:specral_pictures}.

Using a homotopy argument and \cite[Theorem 5.2.8]{berk}, one can show that $q_c = n_-(\cA+V_-) = 0$. Noting that the number of intersections of the eigenvalue curves with the line $\la=0$ for $t<1$ is the quantity $p_c+q_c$, \cref{fig:specral_pictures} (c), (d) and (e) therefore numerically verify that the three different standing waves described above all satisfy the condition of \cref{corollary:VK}. If the eigenvalue curves are continuous and there are no points of horizontal tangency for $t<1$ away from $\la=0$, it follows that the existence of a positive real eigenvalue on the full domain (given by intersections with $t=1$) may be predicted from the concavity of the eigenvalue curve through the point $(\la,t)=(0,1)$: if $t''(\la)|_{\la=0}>0$ then $\cN+V$ has a positive real eigenvalue, while if $t''(\la)|_{\la=0}<0$ then $\spec(\cN+V)\subset i\R$\footnote{\, It can be proven that there are no complex eigenvalues under the assumptions of \cref{corollary:VK}, see e.g. \cite{CCLM23,pelinovsky}.}. \Cref{fig:specral_pictures} therefore reflects a change in stability of the underlying wave as $b$ decreases (and the edge lengths $\ell_1$ and $\ell_3$ increase): the concavity of the eigenvalue curve changes from negative to positive, indicating that a pair of imaginary eigenvalues have bifurcated onto the real axis.

Following the analysis in \cite{LS20first}, in this paper we also concern ourselves with variations in the boundary conditions and their effect on the spectrum of canonically symplectic operators. To this end, in \cref{fig:t_dep_BCs} we numerically compute the eigenvalue curves for three different operators $\cN_t +V_t$ given by \eqref{Nt_stargraphs}, where the domain is a compact interval, i.e. $\dom A = H^2_0(0,\ell)$ and $\dom\cA_t = \left \{u\in H^2 : \tr_t u\in \cL_t \right \}$, and:
\begin{enumerate}
	\item $V_{-,t} = V_{+,t}=0$, $\ell=\pi$ and  
	$\cL_t$ is the Lagrangian plane corresponding to Dirichlet conditions at $x=0$ for all $t\in[0,2]$, and $t$-dependent boundary conditions at $x=\ell$ described by $\spn\left \{  \begin{bmatrix}
		\sin\left (\frac{\pi}{2} t\right ) \vspace{1mm} \\ - \cos\left (  \frac{\pi}{2} t \right )
	\end{bmatrix}  \right\} \subset \R^2$; thus $\cL_t$ describes Dirichlet conditions at $x=0,\pi$ when $t=0,2$, and the Dirichlet and Neumann conditions at $x=0$ and $x=\pi$, respectively, when $t=1$;
	\item $V_{-,t}$ and $V_{+,t}$ are as in \eqref{Ntop2} with $p=3$ and $\beta=-2$, where $\phi$ is a $T$-periodic non-negative solution to \eqref{SWE} satisfying Dirichlet boundary conditions and $\ell = T/2\approx 3.2842$, and $\cL_t$ is the Lagrangian plane described in item (1) (with right endpoint $x\approx 3.2842$);
	\item $V_{-,t}, V_{+,t}$ as in item (2) but with $p=1$, $\ell = T/2 \approx 1.0987$ and $\cL_t$ the Lagrangian plane with $t$-dependent boundary conditions at $x=0$ \emph{and} $x=\ell$ described by $\spn\left \{  \begin{bmatrix}
		\sin\left (\frac{\pi}{2} t\right ) \vspace{1mm} \\ - \cos\left (  \frac{\pi}{2} t \right )
	\end{bmatrix}  \right\} \subset \R^2$; thus $\cL_t$ describes Dirichlet conditions at $x=0,\ell$ when $t=0,2$, and Neumann conditions at $x=0,\ell$ when $t=1$.
\end{enumerate}
 In contrast to the self-adjoint case, the eigenvalue curves are not monotone, and it is therefore much more difficult to predict how many positive real eigenvalues exist when, for example, $t=1$, given the number of positive real eigenvalues when $t=0$. Moreover, such predictions become even more difficult given the following behaviour observed in all numerical examples: as the Lagrangian plane $\cL_t$ approaches the Dirichlet plane as $t\rightarrow 0^+$ or as $t\rightarrow 2^-$, a pair (possibly multiple pairs) of real eigenvalues diverge to $\pm\infty$. In any case, in this paper it is our goal to understand such eigenvalue curves locally by deriving expressions for their derivatives, allowing for $t$-dependence in the boundary conditions in our abstract setting. {   Related graphical and homotopy approaches to eigenvalue problems, in rather different settings, appear in \cite{Kollar12,KM14}.}

\begin{figure}
\centering
\hspace*{\fill}
\subcaptionbox{\label{}} 
{\includegraphics[width=0.3\textwidth]{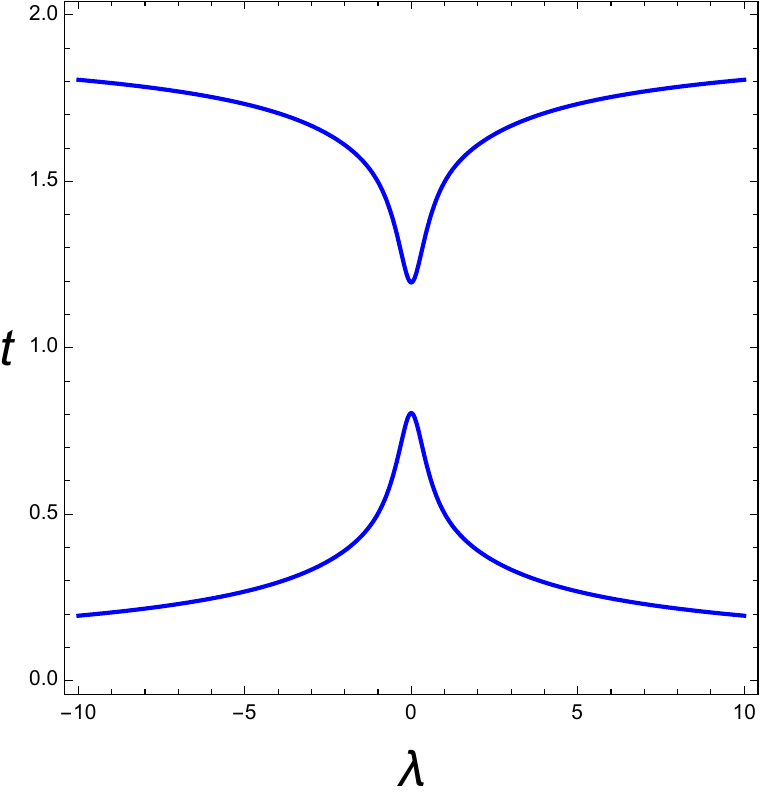}}
\subcaptionbox{\label{}}
{\includegraphics[width=0.3\textwidth]{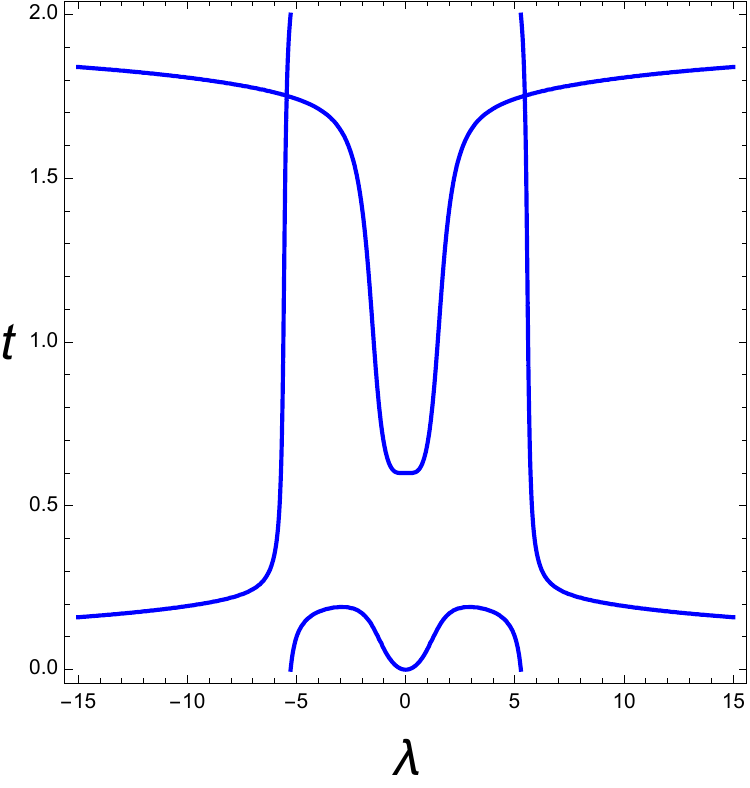}} \hfill
\subcaptionbox{\label{}}
{\includegraphics[width=0.3\textwidth]{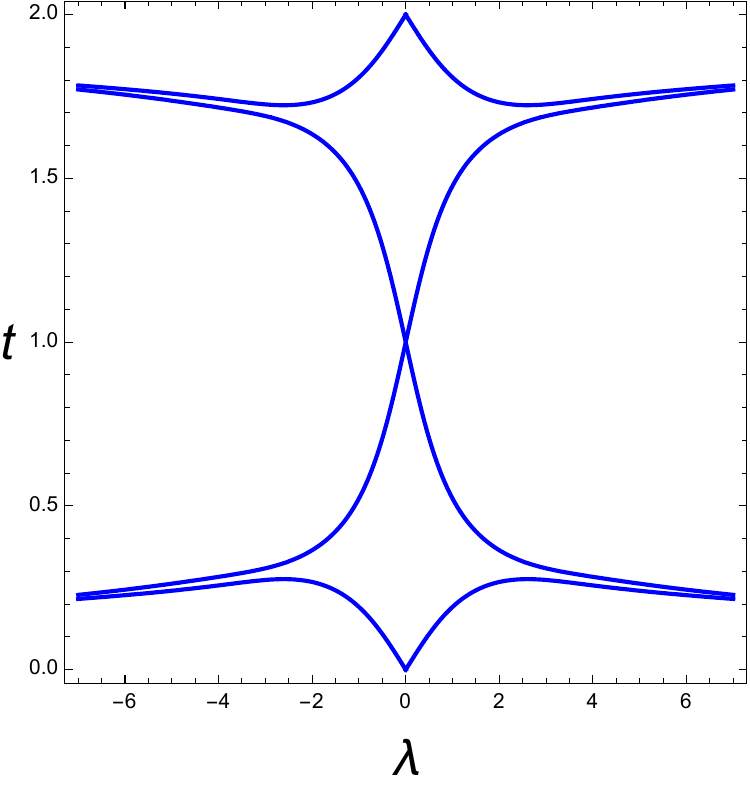}} 
\hspace*{\fill}  \vspace{2mm}
\caption{Eigenvalue curves for $\cN_t+V_t$ where (a), (b) and (c) correspond to items (1), (2) and (3) described in the text.}
\label{fig:t_dep_BCs}
\end{figure}

\subsection{Discussion of abstract setting and main results}

Motivated by the above examples and following the analyses in \cite{LS20first,CCLM23}, in this paper we derive Hadamard formulas for the real eigenvalue curves of canonically symplectic operators, i.e. off-diagonal block operators where the blocks are self-adjoint extensions of densely defined symmetric operators with finite deficiency indices. Here, as in \cite{LS20first}, we allow for $t$-dependence in the bounded perturbation of the operator, in the trace operator, and in the Lagrangian planes describing the boundary conditions. We adopt two different approaches: on the one hand, following \cite{LS20first} we derive a symplectic resolvent difference formula which hinges on a modified Green's identity for the operators of interest. We use these formulas to obtain a first-order asymptotic expansion of the resolvent, which, in turn, we use to derive a Hadamard formula for the eigenvalue curve $\la(t)$. On the other hand, following \cite{CCLM23} we use Lyapunov-Schmidt reduction and the implicit function theorem to arrive at the same Hadamard-type formula for  $\la(t)$, and to derive an additional formula for $t(\la)$. As outlined in the first example in the previous section, such a formula is necessary in order to determine global counts of eigenvalues. On a conceptual level, the main difference from the analysis in \cite{CCLM23} is that we now allow for $t$-dependence in the Lagrangian plane describing the boundary conditions (where such Lagrangian plane was the (fixed) Dirichlet plane in \cite{CCLM23}).

Let us succinctly describe the abstract setting and highlight the corresponding results. First, we fix a densely defined symmetric operator $A$ acting in a Hilbert space $\cH$ (e.g. $A=-\partial_{xx}$, $\cH=L^2(\cG)$ as above) and consider a $2\times 2$ block operator matrix acting in $\cH^2 \coloneqq \cH \oplus \cH$,
\begin{equation}
N : \dom(N) = \dom(A) \oplus \dom(A)\subseteq \cH^2\to\cH^2, \qquad N = \begin{pmatrix}
	0 & -A \\ A & 0
\end{pmatrix}.
\end{equation}
Given a trace map $\tr:\dom(A^*)\rightarrow \mathfrak H\times\mathfrak H$ corresponding to $A$ (e.g. \eqref{Gamma_01_NLS}), we first derive, see \cref{prop:T_trace}, a version of Green's identity for the non-self-adjoint operator $N$, 
\begin{equation}
\left \langle (N^*-\lambda)\tau \mathbf{u}, \mathbf{v}\right \rangle_{\cH^2} - \left \langle \mathbf{u}, (N^*-\lambda)\tau\mathbf{v}\right \rangle_{\cH^2} = \Om(\T\mathbf{u}, \T\mathbf{v}),\ \lambda\in\bbR.
\end{equation}
Here, $\Omega=\omega\oplus(-\omega)$, where $\omega$ is the standard symplectic form associated with $A$, see \eqref{omega}, $\T \coloneqq \tr\oplus\tr$, and $\tau:\cH^2\rightarrow\cH^2$ is the involution operator defined by $\tau(u,v):=(v,u)$, see \eqref{tau}. Using this version of Green's identity we first obtain a new resolvent difference formula for a pair of extensions of $N$, 
\begin{align}
R_1(\zeta) - R_2(\zeta) = \tau \left ( \T R_1(\bar{\zeta})\right )^* \Big ((P_1 J P_2)\oplus \left (-Q_1 J Q_2\right )\Big )\T R_2(\zeta), \label{resolvent_diff2_intro}
\end{align}
where for a pair $\cA_i, \cB_i$ of self-adjoint extensions of $A$ and a pair of orthogonal projections $P_i, Q_i\in \cB(\mathfrak H^2, \mathfrak H^2)$ related by  $\tr(\dom(\cA_i))=\ran P_i$, $\tr(\dom(\cB_i))=\ran Q_i$, we denote
\begin{equation}
R_i(\zeta):=(\cN_i-\zeta)^{-1},\ \ \zeta\not\in\spec(\cN_i),\ \   \cN_i := \begin{pmatrix}
	0 & -\cB_i \\ \cA_i & 0
\end{pmatrix},\  i=1,2.
\end{equation}

Our  second main result leverages \eqref{resolvent_diff2_intro} to provide a first-order asymptotic expansion for a one-parameter family of resolvents $t\mapsto R_t(\zeta):=(\cN_t+V_t-\zeta)^{-1}$, where $\cN_t$ is now determined by one-parameter families of self-adjoint extensions $\cA_t, \cB_t$ of $A$, trace maps $\tr_t$, orthogonal projections $P_t, Q_t$ all related via $\tr_t(\dom(\cA_t))=\ran P_t $, $\tr_t(\dom(\cB_t))=\ran Q_t$, and the bounded perturbation
\begin{equation}\label{V_potential}
\qquad V_t = \begin{pmatrix}
	0 & - V_{-,t} \\ V_{+,t} & 0
\end{pmatrix},
\end{equation}
where $V_{-,t}, V_{+,t}$ are bounded self-adjoint operators in $\cH$ (cf., e.g., \eqref{Ntop2}). Under natural assumptions on the one-parameter families listed above, see \cref{thm:resolvent_regularity} for exact statements, we obtain the following formula
\begin{align}
\begin{split}\label{Rt_asymp_exp_intro}
	R_t(\zeta) &\underset{t\to t_0}{=} R_{t_0}(\zeta) + \Big ( - R_{t_0}(\zeta) \dot V_{t_0}  R_{t_0}(\zeta) + \tau\left ( \T_{t_0} R_{t_0}(\bar\zeta)   \right )^*  \big (\dot P_{t_0} J \oplus  -\dot Q_{t_0} J\big ) \T_{t_0} R_{t_0}(\zeta)  \\
	& \qquad \qquad \qquad \qquad +\tau\left ( \T_{t_0} R_{t_0}(\bar\zeta)   \right )^*  \cJ \dot\T_{t_0} R_{t_0}(\zeta)\Big ) (t-t_0) +o(t-t_0) \quad \text {in}\,\, \cB(\cH^2).
\end{split}
\end{align}
This, in turn, yields the Hadamard-type formula for the eigenvalue curve $t\mapsto \lambda(t)$ bifurcating from a simple eigenvalue $\la_0\in\spec(\cN_{t_0}+V_{t_0})\cap\bbR$,
\begin{equation}\label{Had_fmla_via_res_intro}
\la'(t_0) = \f{\langle \tau \dot{V}_{t_0} \mathbf{u}_{t_0}, 	\mathbf{u}_{t_0}\rangle_{\cH^2} + \Om \left (( \dot P_{t_0} \oplus \dot Q_{t_0} )\T_{t_0}\mathbf{u}_{t_0}, \T_{t_0}\mathbf{u}_{t_0}\right ) + \Om \left (\T_{t_0}\mathbf{u}_{t_0}, \dot{\T}_{t_0}\mathbf{u}_{t_0}\right )}{ \langle \tau \mathbf{u}_{t_0}, \mathbf{u}_{t_0}\rangle_{\cH^2}},
\end{equation}
where $\langle \tau \mathbf{u}_{t_0}, \mathbf{u}_{t_0}\rangle_{\cH^2}\neq 0$ follows from $\la_0\in\bbR$ and the structure of the canonically symplectic operators. 

Our third main result establishes the previous formula via a different method --  Lyapunov-Schmidt reduction -- and, in addition, gives an expression for $t'(\la)$.  In particular, if $\la_0\in\bbR$ is an isolated eigenvalue of $\cN_{t_0}+V_{t_0}$ of geometric multiplicity $n$, we show in  \cref{prop:M} that the set of points $(\la,t)$ near $(\la_0,t_0)$ such that $\ker(\cN_t+V_t - \la) \neq \{0\}$ is given by the zero set of a $n\times n$ matrix $M(\la,t)$. We compute Maslov crossing forms in \cref{lemma:cross_forms} for the path of Lagrangian planes $\Upsilon_{\la,t} \coloneqq \K_{\la,t}\oplus \cF_{t}$ on the intersection $\left (\K_{\la,t}\oplus \cF_{t}\right ) \cap \fD$, where $\K_{\la,t} \coloneqq \T_t(\ker((N^*+V_t^*-\la)\tau))$, $\cF_{t} \coloneqq \ran P_t \times \ran Q_t$ and $\fD$ is the diagonal subspace in $\fH^4\times \fH^4$ (cf. \eqref{cKcFcD}),
\begin{subequations}
\begin{align}
	\mathfrak{m}_{t_0}(\mathbf{q},\mathbf{q}) &= \langle \tau \dot{V}_{t_0} \mathbf{u}_{t_0}, \mathbf{u}_{t_0}\rangle_{\cH^2} + \Om \left (( \dot P_{t_0} \oplus \dot Q_{t_0} )\T_{t_0}\mathbf{u}_{t_0},  \T_{t_0}\mathbf{u}_{t_0}\right ) + \Om \left (\T_{t_0}\mathbf{u}_{t_0}, \dot{\T}_{t_0}\mathbf{u}_{t_0}\right ), \label{t_crossing_form_intro}\\
	\mathfrak{m}_{\la_0}(\mathbf{q},\mathbf{q}) &= -\langle \tau \mathbf{u}_{t_0}, \mathbf{u}_{t_0}\rangle_{\cH^2}. \label{la_crossing_form_intro} 
\end{align}	
\end{subequations}
We show that in the geometrically simple case (when $n=1$), $t$ and $\la$ derivatives of $M$ are given precisely by \eqref{t_crossing_form_intro} and \eqref{la_crossing_form_intro} respectively, and moreover in \cref{thm:Hadamard} that the derivatives of eigenvalue curves are given by
\begin{align}\label{hadamard_fmls_summary}
\la'(t_0) = - \f{\mathfrak{m}_{t_0}}{\mathfrak{m}_{\la_0}} \,\,\,\text{if} \,\,\, \mathfrak{m}_{\la_0}\neq 0 \qquad \text{and} \qquad  t'(\la_0) = - \f{\mathfrak{m}_{\la_0}}{\mathfrak{m}_{t_0}} \,\,\,\text{if} \,\,\, \mathfrak{m}_{t_0}\neq 0.
\end{align}
We stress that the first equation in \eqref{hadamard_fmls_summary} together with \eqref{t_crossing_form_intro}, \eqref{la_crossing_form_intro} reaffirm the Hadamard-type formula \eqref{Had_fmla_via_res_intro}. 

The paper is organised as follows. In \cref{sec:2_setup}, we describe our abstract set-up and main assumptions, give the Green's identity which underpins our analysis, and prove the new symplectic resolvent difference formula for canonically symplectic operators. In \cref{sec:resolvent_formulas_expansions} we use our resolvent difference formula to give a first-order expansion of the resolvent, and use this to prove a Hadamard formula for eigenvalue curves $\la(t)$. In \cref{sec:Hadamard_crossingforms} we compute expressions for the Maslov crossing forms, and use Lyapunov--Schmidt reduction and the implicit function theorem to compute Hadamard formulas for the eigenvalue curves (namely for both $\la(t)$ and $t(\la)$). In \cref{sec:applications} we use our expressions for the crossing forms to prove \cref{thm:spectral_index_thm} for the example outlined in the introduction.

	\section{Set-up, symplectic resolvent difference formula and modified Green's identity}\label{sec:2_setup}
	Let $\mathcal{H}, \mathfrak{H}$ be separable Hilbert spaces, and let $A$ be a closed, densely defined, symmetric linear operator acting in $\cH$ with equal and finite defect indices,
	\begin{align}\label{defect_indices}
		\begin{split}
		\dim\ker(A^*+ i) = \dim\ker(A^*- i) < \infty,
		\end{split}
	\end{align}
	We denote the domain of the adjoint $A^*$ by $\cH_+\coloneqq\dom(A^*) \subset \cH$, which is a Hilbert space when equipped with the graph scalar product of $A^*$,
	\begin{equation}
		\langle u, v \rangle_{\cH_+} \coloneqq \langle u, v\rangle_{\cH} + \langle A^* u, A^* v\rangle_{\cH}.
	\end{equation}
	Throughout the rest of the paper we work under the following assumption. 
	\begin{hypo}\label{hypo:2.1_fundamental_hypo}
		We assume that $A$ is a densely defined, closed, symmetric linear operator acting in $\cH$ with equal and finite deficiency indices. Moreover, let   
		\begin{equation*}
			\tr\coloneqq (\Gamma_0,\Gamma_1)^\top :  \cH_+ \lra \mathfrak{H}\times \mathfrak{H},
		\end{equation*}
		be a bounded and surjective linear operator, called the \emph{trace operator}, which satisfies the following abstract Green's identity for all $u, v\in\dom(A^*)$,
		\begin{equation}\label{greens1}
			\langle A^*u, v\rangle_\cH - \langle u, A^*v \rangle_\cH = \langle \Gamma_1 u, \Gamma_0v\rangle_\fH -  \langle \Gamma_0 u, \Gamma_1 v\rangle_\fH.
		\end{equation}
	\end{hypo}
	Thus, under \cref{hypo:2.1_fundamental_hypo}, the triple $(\fH,\Gamma_0, \Gamma_1)$ is a \emph{boundary triplet} for the adjoint operator $A^*$. Consequently, by \cite[Proposition 14.5]{Behrndt_Hassi_deSnoo}, we have 
	\begin{equation}\label{fH_deficiency_indices}
		\dim \fH = \dim\ker(A^*\pm i).
	\end{equation}
	The right-hand side of the identity \eqref{greens1} gives rise to a symplectic form $\w: \fH^2 \times \fH^2 \to \C$ defined by
	\begin{align}\label{omega}
		\begin{split}
		\w\left ((f_1, f_2)^\top, (g_1,g_2)^\top\right ) &\coloneqq \langle f_2, g_1\rangle_\fH - \langle f_1, g_2\rangle_\fH, \\
			&= \left\langle  J (f_1,f_2)^\top, (g_1, g_2)^\top   \right \rangle_{\fH\times \fH}, \qquad J = \begin{pmatrix}
				0 & I_\fH \\ -I_\fH & 0
			\end{pmatrix}.
			\end{split}
	\end{align}
	Thus \eqref{greens1} may be written as
	\begin{equation}\label{omegaGreen}
		\langle A^*u, v\rangle_\cH - \langle u, A^*v \rangle_\cH = \w(\tr u, \tr v).
	\end{equation}
	We denote the annihilator of a subspace $\cF \subset \fH \times \fH$ by
	\begin{equation}
		\cF^\circ \coloneqq \{ (f_1, f_2)^\top \in \fH \times \fH : \w\left ((f_1, f_2)^\top,(g_1, g_2)^\top\right )=0 \,\,\text{for all}\,\, (g_1, g_2)^\top\in\cF\},
	\end{equation}
	and recall that $\cF$ is called \emph{Lagrangian} if $\cF = \cF^\circ$. We denote by $\Lambda(\fH\times \fH)$ the metric space of all Lagrangian subspaces of $\fH\times \fH$ equipped with the metric
	\begin{equation}
		d(\cF_1, \cF_2) \coloneqq \| Q_1 - Q_2\|_{\cB(\fH\times\fH)}, \quad \cF_1, \cF_2 \in \Lambda(\fH\times \fH),
	\end{equation}
	where $Q_j$ is the orthogonal projection onto $\cF_j$ acting in $\fH\times\fH, j=1,2$. Since the deficiency indices of $A$ are finite, there is a one-to-one correspondence between self-adjoint extensions of $A$ and the Lagrangian planes in $\Lambda(\mathfrak H\times \mathfrak H)$, see, e.g. \cite[\S14.2]{Schmudgen}. We will use this fact throughout, cf. \cref{lemma:difference1}, \cref{hypo:QPGH_families}.

	We are ready to introduce the operator of interest. Our study takes place on the Hilbert spaces
	\[
	\cH^2 \coloneqq {\cH \oplus \cH}, \qquad \cH_+^2\coloneqq \cH_+\oplus \cH_+, \qquad \fH^4\coloneqq (\fH\times \fH) \times (\fH\times\fH).
	\]
	Associated with the minimal symmetric operator $A$ is the following closed, densely-defined minimal operator $N$ acting in $\cH^2$,
	\begin{equation}\label{N}
		N:\dom(N) = \dom(A)\oplus \dom(A)\subseteq \cH^2 \to \cH^2, \qquad N = \begin{pmatrix}
			0 & -A \\ A & 0
		\end{pmatrix}.
	\end{equation}
	The associated adjoint operator $N^*$ is given by 
	\begin{equation}\label{N*}
			N^*: \dom(N^*) = \cH_+^2 \subseteq \cH^2\to\cH^2, \qquad N^* = \begin{pmatrix}
				0 & A^* \\ -A^* & 0
			\end{pmatrix}.
	\end{equation} 
	As an aside, we note that $\cH_+^2$ is a Hilbert space which we equip with the graph scalar product of $N^*$, 
	\begin{align}
		\begin{split}\label{inn_prod_H+}
			\langle \mathbf{u}, \mathbf{v} \rangle_{\cH_+^2} &\coloneqq \langle \mathbf{u}, \mathbf{v} \rangle_{\cH^2} + \langle N^* \mathbf{u}, N^* \mathbf{v}\rangle_{\cH^2}, \\
			&= \langle u_1, v_1\rangle_\cH + \langle u_2, v_2\rangle_\cH + \langle A^* u_2, A^* v_2 \rangle_\cH + \langle A^* u_1, A^* v_1 \rangle_\cH 
			\end{split}
	\end{align}
	where $\mathbf{u}=(u_1, u_2)^\top, \mathbf{v} = (v_1, v_2)^{\top} \in \cH_+^2$. 
	
	Let $\cA, \cB$ be self-adjoint extensions of the minimal symmetric operator $A$, i.e. $A\subset \cA, \cB \subset A^*$. The extensions of $N$ whose real spectra we wish to study will be denoted 
	\begin{equation}\label{cNt}
		\cN : \dom(\cN) = \dom(\cA) \oplus \dom(\cB)\subseteq \cH^2\to\cH^2, 
		\qquad \cN =  \begin{pmatrix}
			0 & -\cB \\ \cA & 0
		\end{pmatrix}.
	\end{equation}
	On the product space $\fH^4\times \fH^4$ we define the following symplectic form,
	\begin{align}
		\begin{split}\label{Om}
			\Om : \fH^4&\times\fH^4\lra \C, \\
			\Om( (f_1,&f_2,f_3, f_4)^\top,(g_1,g_2,g_3,g_4)^\top) \coloneqq   
			\langle f_2 , g_1 \rangle_\fH - \langle f_1, g_2\rangle_\fH 
			- \langle f_4, g_3\rangle_\fH + \langle f_3, g_4\rangle_\fH, \\
			&= \left \langle \mathcal{J} (f_1,f_2,f_3, f_4)^\top,(g_1,g_2,g_3,g_4)^\top   \right \rangle_{\fH^4},\quad 
			\mathcal{J} =\begin{pmatrix}
				J & 0 \\
				0 & -J 
			\end{pmatrix} = J\oplus (-J),
		\end{split}
	\end{align}
	where $J$ is defined in \eqref{omega}. An important operator in our analysis will be the bounded involution
	\begin{align}\label{tau}
		&\tau: \cH^2 \to \cH^2, \quad \tau = \begin{pmatrix}
			0 & I_{\cH} \\ I_{\cH} & 0
		\end{pmatrix}, \quad \tau \begin{pmatrix}
			u_1 \\ u_2
		\end{pmatrix}= \begin{pmatrix}
			u_2 \\ u_1
		\end{pmatrix}.
	\end{align} 
	The operator
	\begin{equation}
		\tau N:\dom(N) = \dom(A)\oplus \dom(A)\subseteq \cH^2 \to \cH^2, \qquad \tau N = \begin{pmatrix}
			A & 0  \\   0 & -A
		\end{pmatrix},
	\end{equation}
	is then symmetric since $\tau N = A \oplus (-A) \subset A^*\oplus (-A^*) = (\tau N)^*$, while 
	\begin{equation}\label{tauNV}
		\tau \cN: \dom(\cA)\oplus \dom(\cB) \subseteq \cH^2 \to \cH^2, \quad \tau \cN = \begin{pmatrix}
			\cA & 0 \\ 0 & -\cB
		\end{pmatrix},
	\end{equation}
	is obviously self-adjoint. We note that taking the adjoint is equivalent to conjugating by $\tau$,
	\begin{equation}\label{N*tau}
		\tau \cN^* \tau = \cN \iff \tau \cN \tau = \cN^*.
	\end{equation} 
	 Under \cref{hypo:2.1_fundamental_hypo}, one can define a trace operator on the space $\cH_+ \oplus \cH_+$ as follows. 
	\begin{prop}\label{prop:T_trace}
		Under \cref{hypo:2.1_fundamental_hypo}, the linear operator
		\begin{align}\label{T_define}
			\begin{split}
			&\T\coloneqq \tr \oplus \tr = [\Gamma_{0}, \Gamma_{1}]^\top \oplus [\Gamma_{0}, \Gamma_{1}]^\top : \cH_+^2 \lra \fH^4, \\
			&\T \mathbf{u} = \left [\tr u_1, \tr u_2 \right ]^\top = \left [\Gamma_{0}u_1, \Gamma_{1}u_1,\Gamma_{0}u_2, \Gamma_{1} u_2 \right ]^\top, \quad \mathbf{u} = (u_1, u_2)\in \cH_+^2,
			\end{split}
		\end{align}
		acting from $\cH_+^2$ into $\fH^4$ is bounded and surjective, and satisfies
		\begin{equation}\label{kerT_domN}
			\ker(\T) = \dom(N) = \dom(A) \oplus \dom(A).
		\end{equation}
		It is a \emph{trace} operator, in the sense that it satisfies the following abstract Green's identity associated with the maximal adjoint operator $N^*$ for functions $\mathbf{u}, \mathbf{v}\in \dom\left ( (\tau N)^*\right ) = \cH_+^2$,
		\begin{equation}\label{modified_greensOm}
			\left \langle (\tau N)^* \mathbf{u}, \mathbf{v}\right \rangle_{\cH^2} - \left \langle \mathbf{u}, (\tau N)^*\mathbf{v}\right \rangle_{\cH^2} = \Om(\T\mathbf{u}, \T\mathbf{v}).
		\end{equation}
		Moreover, for $\la\in\R$ we have
		\begin{equation}\label{modified_greens_real_la}
			\left \langle (N^* - \la) \tau \mathbf{u}, \mathbf{v}\right \rangle_{\cH^2} - \left \langle \mathbf{u}, (N^* - \la) \tau\mathbf{v}\right \rangle_{\cH^2} = \Om (\T \mathbf{u}, \T \mathbf{v}).
		\end{equation}

	\end{prop}
	
	\begin{proof}
		The Green's identity \eqref{modified_greensOm} follows from \eqref{omegaGreen}. Consequently \eqref{kerT_domN} follows from the fact that $\T$ is the trace map for symmetric operator $\tau N$ in the Hilbert space $\cH\oplus\cH$. The identity \eqref{modified_greens_real_la} follows from \eqref{modified_greensOm} and the fact that $\tau^*=\tau$, $\lambda\in\bbR$.
		\end{proof}

	In the classical setting of self-adjoint extensions $\cA$ of a symmetric operator $A$, Green's identity \eqref{greens1} for the adjoint operator $A^*$ facilitates the study of the eigenvalues of the self-adjoint restrictions of $A^*$ in terms of Lagrangian planes. In the present setting, one immediate impediment to such an approach with the canonical symplectic operator $N$ is that $N$ is not symmetric, i.e. a restriction of its adjoint. Nonetheless, ``untwisting" the operator via multiplication by $\tau$ reveals an operator that \emph{is} symmetric, and for that operator one has the modified Green's identity \eqref{modified_greensOm}. The restriction that $\la\in\R$ and the fact that $\tau$ is symmetric allows one to add in the extra terms involving $\la$ in \eqref{modified_greens_real_la}, and it is \emph{this} identity that furnishes a symplectic interpretation of the eigenvalues of the extensions \eqref{cNt} of the minimal operator $N$.

	Next, we write down a symplectic resolvent difference formula for the operator $\cN$. 
	\begin{theorem}\label{lemma:difference1}
		Assume \cref{hypo:2.1_fundamental_hypo}, and let  $\cA_1, \cA_2, \cB_1, \cB_2$ be self-adjoint extensions of $A$ with 	\begin{equation}
			\tr (\dom \cA_i) = \ran P_i, \qquad  \tr (\dom \cB_i)=\ran Q_i, \qquad i=1,2.
		\end{equation}Define $\cN_i$ by	
	\begin{equation}\label{cNi}
		\cN_i : \dom(\cA_i)\oplus \dom(\cB_i) \subseteq \cH^2\to\cH^2, \qquad \cN_i = \begin{pmatrix}
			0 & -\cB_i \\ \cA_i & 0
		\end{pmatrix}, \qquad i=1,2.
	\end{equation}
Let $\zeta \notin \spec(\cN_1) \cup \spec(\cN_2)$ and denote $R_i(\zeta) \coloneqq (\cN_i - \zeta)^{-1}, i=1,2$. Then $\bar\zeta \notin \spec(\cN_1) \cup \spec(\cN_2)$ and one has
		\begin{align}\label{resolvent_diff1}
			R_1(\zeta) - R_2(\zeta) = \tau\left ( \T R_1(\bar{\zeta})\right )^* \cJ \T R_2(\zeta),
		\end{align}
	and
		\begin{align}
			R_1(\zeta) - R_2(\zeta) = \tau \left ( \T R_1(\bar{\zeta})\right )^* \Big ((P_1 J P_2)\oplus \left (-Q_1 J Q_2\right )\Big )\T R_2(\zeta). \label{resolvent_diff2}
		\end{align}
	\end{theorem}
	\begin{proof}
		First, we note that $\bar\zeta \notin \spec(\cN_1) \cup \spec(\cN_2)$ follows from the four-fold symmetry of the spectrum of $\cN_i$. Letting $\mathbf{u},\mathbf{v}\in\cH^2$ be arbitrary, we have, using that $\tau \cN_1$ is self-adjoint,
		\begin{align*}
			&\left \langle \tau R_1(\zeta)\mathbf{u} - \tau R_2(\zeta)\mathbf{u}, \mathbf{v} \right \rangle_{\cH^2} = \left \langle R_1(\zeta)\mathbf{u} - R_2(\zeta)\mathbf{u}, \tau (\cN_1 - \bar{\zeta}) R_1(\bar{\zeta})\mathbf{v} \right \rangle_{\cH^2}, \\
			&\qquad = \left \langle \tau (\cN_1  - {\zeta}) R_1(\zeta)\mathbf{u}, R_1(\bar{\zeta})\mathbf{v} \right \rangle_{\cH^2} 
			- \left \langle R_2(\zeta)\mathbf{u}, \tau \cN_1t R_1(\bar{\zeta})\mathbf{v} \right \rangle_{\cH^2}
			+ \left \langle  \zeta \tau R_2(\zeta)\mathbf{u}, R_1(\bar{\zeta})\mathbf{v} \right \rangle_{\cH^2}, \\
			&\qquad = \left \langle \tau \mathbf{u}, R_1(\bar{\zeta})\mathbf{v} \right \rangle_{\cH^2}
			- \left \langle R_2(\zeta)\mathbf{u}, (\tau N)^* R_1(\bar{\zeta})\mathbf{v} \right \rangle_{\cH^2}
			+ \left \langle  \zeta \tau R_2(\zeta)\mathbf{u}, R_1(\bar{\zeta})\mathbf{v} \right \rangle_{\cH^2}, \\
			&\qquad = \left \langle \tau\mathbf{u}, R_1(\bar{\zeta})\mathbf{v} \right \rangle_{\cH^2}  
			- \left \langle (\tau  N)^* R_2(\zeta)\mathbf{u}, R_1(\bar{\zeta})\mathbf{v} \right \rangle_{\cH^2} 
			+ \left \langle \zeta \tau R_2(\zeta)\mathbf{u}, R_1(\bar{\zeta})\mathbf{v} \right \rangle_{\cH^2}
			+ \\
			& \qquad \qquad \qquad \qquad \Om(\T R_2(\zeta)\mathbf{u}, \T R_1(\bar{\zeta})\mathbf{v}), \\
			&\qquad = \left \langle \tau\mathbf{u}, R_1(\bar{\zeta})\mathbf{v} \right \rangle_{\cH^2}  
			- \left \langle \tau \left ( \cN_2 - {\zeta}\right ) R_2(\zeta)\mathbf{u}, R_1(\bar{\zeta})\mathbf{v} \right \rangle_{\cH^2} 
			+\left \langle \cJ \T R_2(\zeta)\mathbf{u}, \T R_1(\bar{\zeta})\mathbf{v} \right \rangle_{\fH^4},\\
			&\qquad =  \left \langle \left (\T R_1(\bar{\zeta})\right )^* \cJ \T R_2(\zeta)\mathbf{u}, \mathbf{v} \right \rangle_{\cH^2},
		\end{align*}
		where we used the modified Green's identity \eqref{modified_greensOm} in the fourth line, and that 
		\begin{equation}\label{}
			\tau\cN_1 R_1({\zeta}) = N^*\tau R_1({\zeta})
		\end{equation}
		because $\tau\cN_1 \subset (\tau N)^*$ and $\ran(R_1({\zeta})) \subset \dom \tau\cN_1$.  This yields \eqref{resolvent_diff1}. Next, since $R_1(\zeta) \mathbf{u} \in \dom(\cA_1)\times \dom(\cB_1)$, we have
		\begin{equation}\label{fact3}
			\T R_1(\zeta) \mathbf{u} = \left ( P_1 \oplus Q_1\right ) \T R_1(\zeta) \mathbf{u},
		\end{equation}
		where we have denoted
		\begin{equation}
			P_1 \oplus Q_1 = \begin{bmatrix}
				P_1 & 0 \\ 0 & Q_1
			\end{bmatrix}.
		\end{equation}
		Thus
		\begin{align*}
			&\left \langle \left (\T R_1(\bar{\zeta})\right )^* \cJ \T R_2(\zeta)\mathbf{u}, \mathbf{v} \right \rangle_{\cH^2} =  \left \langle \left (\left ( P_1 \oplus Q_1 \right )\T R_1(\bar{\zeta})\right )^* \cJ \left ( P_2 \oplus Q_2 \right ) \T R_2(\zeta)\mathbf{u}, \mathbf{v} \right \rangle_{\cH^2}, \\
			& \qquad \qquad \qquad = \left \langle \left (\T R_1(\bar{\zeta})\right )^* \left ( P_1 \oplus Q_1 \right )\cJ \left ( P_2 \oplus Q_2 \right ) \T R_2(\zeta)\mathbf{u}, \mathbf{v} \right \rangle_{\cH^2}, \\
			&\qquad \qquad \qquad = \left \langle \left (\T R_1(\bar{\zeta})\right )^* \Big( \left (P_1 J P_2\right ) \oplus  \left (-Q_1 J Q_2\right ) \Big) \T R_2(\zeta)\mathbf{u}, \mathbf{v} \right \rangle_{\cH^2},
		\end{align*}
		from which \eqref{resolvent_diff2} follows. 
	\end{proof}

	\section{First-order asymptotic perturbation theory}\label{sec:resolvent_formulas_expansions}
	
	In this section, we introduce $t$-dependence of the various operators, working under the following assumptions.
	\begin{hypo}\label{hypo:QPGH_families}
		Let 
		\[
		\tr : [0,1] \to \cB(\cH_+,\fH\times \fH) : t\mapsto \tr_t \coloneqq[\Gamma_{0,t}, \Gamma_{1,t}]^\top
		\]
		be a one-parameter family of trace operators, such that $(\fH, \Gamma_{0,t},\Gamma_{1,t})$ is a boundary triplet for each $t\in[0,1]$ satisfying \cref{hypo:2.1_fundamental_hypo}. Let $P:[0,1] \to \cB(\fH^2), t\mapsto P_t$ and $Q:[0,1] \to \cB(\fH^2), t\mapsto Q_t$ be one-parameter families of orthogonal projections, such that $\ran P_t, \ran Q_t \in \Lambda(\fH^2)$ are Lagrangian planes for each $t\in[0,1]$. Let $A\subset \cA_t, \cB_t \subset A^*$ be families of self-adjoint extensions of $A$ satisfying 
		\begin{equation}\label{projectionsQP}
			\tr_t (\dom \cA_t) = \ran P_t, \qquad  \tr_t (\dom \cB_t)=\ran Q_t.
		\end{equation}
		Define the minimal canonical symplectic operator $N$ as in \eqref{N}, and the following one-parameter family of extensions of $N$,
		\begin{equation}
			\cN_t : \dom(\cN_t) = \dom(\cA_t) \oplus \dom(\cB_t)\subseteq \cH^2\to\cH^2, \qquad \cN_t = \begin{pmatrix}
				0 & -\cB_t \\ \cA_t & 0
			\end{pmatrix}.
		\end{equation}
		For 
		$
		\T\coloneqq \tr \oplus \tr : [0,1] \to \cB(\cH_+^2, \fH^4) : t\mapsto \T_t = \tr_t \oplus \tr_t = [\Gamma_{0,t}, \Gamma_{1,t}]^\top \oplus [\Gamma_{0,t}, \Gamma_{1,t}]^\top, 
		$ 
		we note
		\begin{equation}\label{TtdomNt}
			\T_t(\dom \cN_t  ) = \T_t(\dom \cA_t \oplus\dom \cB_t ) = \ran P_t \times \ran Q_t. 
		\end{equation}
		Next, let $V_-:[0,1] \to \cB(\cH), t\mapsto V_{-,t}$ and $V_+:[0,1] \to \cB(\cH), t\mapsto V_{+,t}$ be families of bounded, self-adjoint operators acting in $\cH$. On the space $\cH^2$, the mapping
		\begin{equation}
			V:[0,1] \to \cB(\cH^2), \qquad V_t = \begin{pmatrix}
				0 & - V_{-,t} \\ V_{+,t} & 0
			\end{pmatrix},
		\end{equation}
		is then a one-parameter family of bounded operators, which, similar to \eqref{N*tau} satisfies 
		\begin{equation}\label{V*tau}
			V_t^* = \tau V_t \tau \iff V_t = \tau V_t^* \tau.
		\end{equation}
		For each $t\in[0,1]$, we define the resolvent $R_t(\zeta) \coloneqq (\cN_t + V_t - \zeta)^{-1}\in\cB(\cH^2)$ for all $\zeta\notin \spec(\cN_t + V_t)$.
	\end{hypo}

	Our first task is to write down a resolvent difference formula for the resolvent
	\begin{equation}
		R_t(\zeta) = (\cN_t + V_t - \zeta)^{-1}, \qquad  \zeta \notin \spec(\cN_t + V_t),
	\end{equation}
	now incorporating the bounded potential $V_t$. 
	
		\begin{lemma}\label{lemma:difference2}
		Let $t,s,r\in[0,1]$ and $\zeta \notin \spec(\cN_t + V_t) \cup \spec(\cN_s + V_s)$. Then for $R_t(\zeta) \coloneqq (\cN_t + V_t -\zeta)^{-1}$ we have
		\begin{align}
			R_t(\zeta) &- R_s(\zeta) =  R_t({\zeta})(V_s -V_t ) R_s(\zeta)+ \tau \left ( \T_r R_t(\bar{\zeta})\right )^*\cJ \T_r R_s(\zeta),  \label{resolvent_diff1new}\\
			\begin{split}\label{resolvent_diff2new}
			&= R_t({\zeta})  (V_s -V_t ) R_s(\zeta) + \tau\left ( \T_t R_t(\bar{\zeta})\right )^* \Big ((P_t-P_s)J\oplus \left (-(Q_t-Q_s) J \right )\Big )\T_s R_s(\zeta) \\
			& \qquad \qquad \qquad \qquad \qquad \qquad \qquad +\tau \left ( \T_t R_t(\bar{\zeta})\right )^* \cJ(\T_t-\T_s) R_s(\zeta).
			\end{split}
		\end{align}
	\end{lemma}

	\begin{proof}
		Let $\mathbf{u},\mathbf{v}\in\cH^2$ be arbitrary. We have, using that $\tau(\cN_t+V_t)$ is self-adjoint,
		\begin{align*}
			&\left \langle \tau R_t(\zeta)\mathbf{u} - \tau R_s(\zeta)\mathbf{u}, \mathbf{v} \right \rangle_{\cH^2} = \left \langle R_t(\zeta)\mathbf{u} - R_s(\zeta)\mathbf{u}, \tau (\cN_t + V_t - \bar{\zeta}) R_t(\bar{\zeta})\mathbf{v} \right \rangle_{\cH^2}, \\
			&\qquad = \left \langle \tau (\cN_t + V_t - {\zeta}) R_t(\zeta)\mathbf{u}, R_t(\bar{\zeta})\mathbf{v} \right \rangle_{\cH^2}
			- \left \langle R_s(\zeta)\mathbf{u}, \tau (\cN_t + V_t - \bar{\zeta}) R_t(\bar{\zeta})\mathbf{v} \right \rangle_{\cH^2}, \\
			&\qquad= \left \langle \tau\mathbf{u}, R_t(\bar{\zeta})\mathbf{v} \right \rangle_{\cH^2} 
			+ \left \langle R_s(\zeta)\mathbf{u},  \tau(V_s -V_t ) R_t(\bar{\zeta})\mathbf{v} \right \rangle_{\cH^2} 
			- \left \langle R_s(\zeta)\mathbf{u}, ((\tau N + \tau V_s)^* - \tau\bar{\zeta}) R_t(\bar{\zeta})\mathbf{v} \right \rangle_{\cH^2}, \\
			&\qquad= \left \langle \tau\mathbf{u}, R_t(\bar{\zeta})\mathbf{v} \right \rangle_{\cH^2} 
			+ \left \langle R_s(\zeta)\mathbf{u},  \tau (V_s -V_t ) R_t(\bar{\zeta})\mathbf{v} \right \rangle_{\cH^2} \\ 
			&\qquad \qquad - \left \langle ((\tau N + \tau V_s)^* - \tau{\zeta}) R_s(\zeta)\mathbf{u}, R_t(\bar{\zeta})\mathbf{v} \right \rangle_{\cH^2} 
			+ \Om(\T_r R_s(\zeta)\mathbf{u}, \T_r R_t(\bar{\zeta})\mathbf{v}),  \\
			&\qquad= \left \langle \tau\mathbf{u}, R_t(\bar{\zeta})\mathbf{v} \right \rangle_{\cH^2} 
			+ \left \langle R_s(\zeta)\mathbf{u},  \tau (V_s -V_t ) \tau^2 R_t(\bar{\zeta})\mathbf{v} \right \rangle_{\cH^2}, \\ 
			&\qquad \qquad - \left \langle \tau ( \cN_s + V_s - {\zeta}) R_s(\zeta)\mathbf{u}, R_t(\bar{\zeta})\mathbf{v} \right \rangle_{\cH^2} 
			+ \left \langle \cJ \T_r R_s(\zeta)\mathbf{u}, \T_r R_t(\bar{\zeta})\mathbf{v} \right \rangle_{\fH^4},  \\
			&\qquad =  \left \langle \tau R_t(\zeta) \tau (V_s -V_t )^* \tau R_s(\zeta)\mathbf{u}, \mathbf{v} \right \rangle_{\cH^2}
			+ \left \langle \left (\T_r R_t(\bar{\zeta})\right )^*\cJ \T_r R_s(\zeta)\mathbf{u}, \mathbf{v} \right \rangle_{\cH^2},
		\end{align*}
		which, using \eqref{V*tau}, proves \eqref{resolvent_diff1new}. (Note that we used \eqref{modified_greensOm} with trace operator $\T_r$.) 
		
		Next, using \eqref{fact3} and that $P_s J P_s = Q_s J Q_s = 0$ since $\ran P_s$ and $\ran Q_s$ are Lagrangian, we have, 
		\begin{align*}
			&\left (\T_t R_t(\bar{\zeta})\right )^*\cJ \T_t R_s(\zeta)= \left (\T_t R_t(\bar{\zeta})\right )^*\cJ \T_s R_s(\zeta) + \left (\T_t R_t(\bar{\zeta})\right )^*\cJ (\T_t-\T_s) R_s(\zeta), \\
			&\qquad = \left (\T_t R_t(\bar{\zeta})\right )^*(P_t\oplus Q_t)\cJ (P_s\oplus Q_s) \T_s R_s(\zeta) + \left (\T_t R_t(\bar{\zeta})\right )^*\cJ (\T_t-\T_s) R_s(\zeta), \\
			&\qquad = \left (\T_t R_t(\bar{\zeta})\right )^*(P_tJP_s \oplus (-Q_tJ Q_s) \T_s R_s(\zeta) + \left (\T_t R_t(\bar{\zeta})\right )^*\cJ (\T_t-\T_s) R_s(\zeta), \\
			&\qquad= \left (\T_t R_t(\bar{\zeta})\right )^*\left ( (P_t-P_s)J \oplus (-( Q_t - Q_s)J)\right ) \T_s R_s(\zeta) + \left (\T_t R_t(\bar{\zeta})\right )^*\cJ (\T_t-\T_s) R_s(\zeta),
		\end{align*}
		which yields \eqref{resolvent_diff2new} upon substituting $r=t$ into \eqref{resolvent_diff1new}.
	\end{proof}

	\subsection{Asymptotic expansions and a Hadamard formula}\label{subsec_asymptotic_hadamard}
	
	In this section, we give first-order asymptotic expansions for the resolvent, and for an operator to which $\sP(t)(\cN_t + V_t)\sP(t)$, where $\sP(t)$ is the spectral projection onto the finite dimensional kernel of $\cN_t+V_t-\la$, is similar. Using the latter, we give a Hadamard-type variational formula for the eigenvalue curve $\la(t)$. We work under the following assumption.

	\begin{hypo}\label{hyp:non_empty_res_set}
		We fix $t_0\in(0,1)$ and suppose that there exists $z\in\bbC$ such that $z\not\in\spec(\cN_t+V_t)$ for $t$ near $t_0$.
		\end{hypo}
	
	We begin by giving two auxiliary statements regarding the resolvent operator $R_t(z) = (\cN_t+V_t-z)^{-1}$ for the $z$ given in \cref{hyp:non_empty_res_set}. The first of these establishes the boundedness of $R_t(z)$ when viewed as a bounded operator from $\cH^2$ into $\cH_+^2$, and will be required, e.g., when applying the trace $\T_t$ to the resolvent operator, since $\dom(\T_t) = \cH_+^2$ (recall \cref{prop:T_trace}). The second statement establishes the continuity of the mapping $t\mapsto R_t(z)$. 
	
	\begin{prop}\label{prop:resolvent_bounded_continuous}
		Assume \cref{hypo:QPGH_families} and \cref{hyp:non_empty_res_set}. Then the following assertions hold. 
		\begin{enumerate}
			\item For  $t$ near $t_0$, the resolvent $R_t(z) = (\cN_t +V_t- z)^{-1}$ can be viewed as a bounded operator from $\cH^2$ into $\cH_+^2$. Consequently, $\T_t R_t(z)\in \cB(\cH^2, \fH^4)$. 
			\item If the mappings $t\mapsto P_t, t\mapsto Q_t, t\mapsto V_t, t\mapsto \T_t$ are continuous at $t_0$, then so is the mapping $t\mapsto R_t(z)\in\cB(\cH^2, \cH^2_+)$, that is, 
			\begin{equation}\label{resolvent_continuous_H_H+}
				\| (\cN_t+V_t - z)^{-1} - (\cN_{t_0}+V_{t_0} - z)^{-1} \|_{\cB(\cH^2,\cH_+^2)} = o(1), \,\,\, t \to t_0.
			\end{equation}
		\end{enumerate} 
	\end{prop}
	\begin{proof}
		Recall that \(\mathcal H_+^2=\dom(\mathcal N^*)\) is equipped with the norm induced by the graph scalar product \eqref{inn_prod_H+}. Towards statement (1), using the Cauchy-Schwarz inequality and that $\tau \cN_t \subset (\tau N)^*$, we have for all $\mathbf{u}\in\cH^2$,
		\begin{align*}
			 &\| R_t(z)\mathbf{u} \|^2_{\cH_+^2} =  \| \tau R_t(z)\mathbf{u}\|^2_{\cH_+^2} = \| \tau R_t(z)\mathbf{u}\|^2_{\cH^2} + \|N^* \tau  R_t(z)\mathbf{u}\|^2_{\cH^2}, \\
			 &\quad =\| \tau R_t(z)\mathbf{u}\|^2_{\cH^2} + \|(\tau N)^*  R_t(z)\mathbf{u}\|^2_{\cH^2}, \\
			 &\quad \leq \| \tau R_t(z)\mathbf{u}\|^2_{\cH^2} + \left (\| \tau (\cN_t +V_t- z)  R_t(z)\mathbf{u}\|_{\cH^2} + (\|\tau V_t\|_{\cB(\cH^2)}+|z|) \| \tau  R_t(z)\mathbf{u}\|_{\cH^2}\right )^2, \\
			 &\quad=  \|  R_t(z)\mathbf{u}\|^2_{\cH^2} + \left (\|  \mathbf{u}\|_{\cH^2} + (\|\tau V_t\|_{\cB(\cH^2)}+|z|) \|   R_t(z)\mathbf{u}\|_{\cH^2}\right )^2.
		\end{align*}
		It follows that
		\begin{equation}\label{eq:3.11}
			\|  R_t(z) \|^2_{\cB(\cH^2,\cH_+^2)} \leq \| R_t(z)\|^2_{\cB(\cH^2)} + \left (1 +  (\|\tau V_t\|_{\cB(\cH^2)}+|z|) \| R_t(z)\|_{\cB(\cH^2)} \right )^2,
		\end{equation}
		hence
		\begin{equation}\label{eq:3.11.2}
			\|  R_t(z) \|_{\cB(\cH^2,\cH_+^2)} \leq 1 +\left  (\|\tau V_t\|_{\cB(\cH^2)}+|z|+1\right ) \|R_t(z)\|_{\cB(\cH^2)}, 
		\end{equation}
		as required. Regarding statement (2), again using that $\tau \cN_t \subset (\tau N)^*$, for $\mathbf{u}\in\cH^2$ we have
		\begin{align*}
			&\left \| R_t(z)\mathbf{u}-R_{t_0}(z)\mathbf{u} \right \|^2_{\cH_+^2} = \left \| \tau R_{t}(z)\mathbf{u}  - \tau R_{t_0}(z) \mathbf{u} \right \|^2_{\cH_+^2}, \\
			&\qquad = \left \| \tau R_{t}(z)\mathbf{u}  - \tau R_{t_0}(z) \mathbf{u} \right \|^2_{\cH^2} + \left \| (\tau N)^* R_{t}(z)\mathbf{u} - (\tau N)^* R_{t_0}(z) \mathbf{u} \right \|^2_{\cH^2}, \\
			&\qquad = \left \| \tau R_{t}(z)\mathbf{u}  - \tau R_{t_0}(z) \mathbf{u} \right \|^2_{\cH^2} \\
			& \qquad \quad +\left \| \tau (\cN_t+V_t -z) R_{t}(z)\mathbf{u} - (\tau (\cN_{t_0}+V_{t_0} -z) R_{t_0}(z) \mathbf{u} -\tau ( V_t - z) R_t(z)\mathbf{u} + \tau (V_{t_0}-z) R_{t_0}(z) \mathbf{u} \right \|^2_{\cH^2}, \\
			&\qquad =\left \| R_t(z)\mathbf{u}  - R_{t_0}(z) \mathbf{u} \right \|^2_{\cH^2} +\left \| \tau ( V_t - z) R_t(z)\mathbf{u} - \tau (V_{t_0}-z) R_{t_0}(z) \mathbf{u} \right \|^2_{\cH^2}, \\
			&\qquad \leq \left \| R_t(z)\mathbf{u}  - R_{t_0}(z) \mathbf{u} \right \|^2_{\cH^2} +2\left \| \tau ( V_t - z) \right\|^2_{\cB(\cH^2)}\left\|(R_t(z)-R_{t_0}(z))\mathbf{u} \|^2_{\cH^2} +2\|\tau (V_t-V_{t_0})\|^2_{\cB(\cH^2)}\| R_{t_0}(z) \mathbf{u} \right \|^2_{\cH^2}, \\
			&\qquad\leq \left ( \| R_t(z)\mathbf{u}  - R_{t_0}(z) \mathbf{u}   \|_{\cH^2}\left (1+\sqrt{2}\left \| \tau ( V_t - z) \right\|_{\cB(\cH^2)}\right )+\sqrt{2}\|\tau (V_t-V_{t_0})\|_{\cB(\cH^2)}\left\| R_{t_0}(z) \mathbf{u} \right \|_{\cH^2}\right )^2.
		\end{align*}
		Taking the supremum over all $\mathbf{u}\in\cH^2$ such that $\|\mathbf{u}\|_{\cH^2}=1$, we obtain
			\begin{align}
				\begin{split}\label{norm_bound_prop33}
					&\left \| R_t(z)-R_{t_0}(z) \right \|_{\cB(\cH^2,\cH_+^2)} \\
					&\,\,\,\,\leq  \| R_t(z)  - R_{t_0}(z) \|_{\cB(\cH^2)}\left (1+\sqrt{2}\left \| \tau ( V_t - z) \right\|_{\cB(\cH^2)}\right )+\sqrt{2} \left \|\tau (V_t-V_{t_0}) \right \|_{\cB(\cH^2)}\left\| R_{t_0}(z)  \right \|_{\cB(\cH^2)}.			
				\end{split}
		\end{align}
		Thus, it is enough to prove that the right-hand side of \eqref{norm_bound_prop33} is $o(1)$ as $t\to t_0$. To this end, we first establish auxiliary assertions
		\begin{equation}\label{eq:res_bound}
			\|R_t(z)\|_{\cB (\cH^2)}=\cO(1), \quad 	\|R_t(z)\|_{\cB (\cH^2, \cH_+^2)}=\cO(1),\quad  t\rightarrow t_0,
		\end{equation} 
		noting that the second equality follows immediately from the first and \eqref{eq:3.11.2}. Using the resolvent difference formula \eqref{resolvent_diff2new}, we find
		\begin{align}
		&\left \|R_t(z) - R_{t_0}(z) \right \|_{\cB(\cH^2)}  \no \\
		& \quad \leq  \|R_t({z})  (V_t -V_{t_0} ) R_{t_0}(z)\|_{\cB(\cH^2)} + \| \tau	\left ( \T_t R_t(\bar{z})\right )^* \Big ((P_t-P_{t_0})\oplus (Q_t-Q_{t_0})\Big )\cJ T_{t_0} R_{t_0}(z) \|_{\cB(\cH^2)}  \no \\ 
		& \quad \quad + \| \tau \left ( \T_t R_t(\bar{z})\right )^* \cJ(\T_t-\T_{t_0}) R_{t_0}(z) \|_{\cB(\cH^2)},  \no \\
		&\quad \leq \ \|R_t({z})\|_{\cB(\cH^2)} \|  (V_t -V_{t_0} )\|_{\cB(\cH^2)} \| R_{t_0}(z)\|_{\cB(\cH^2)} \no \\
		&\quad\quad + \|\T_t   \|_{\cB(\cH_+^2,\fH^4)}
		   \|R_t(z)  \|_{\cB(\cH^2,\cH_+^2)}
		   \|(P_t-P_{t_0})\oplus (Q_t-Q_{t_0}) \|_{\cB(\fH^4)}  \left \|\T_{t_0} \right \|_{\cB(\cH_+^2,\fH^4)} \left \|  R_{t_0}(z)\right \|_{\cB(\cH^2,\cH_+^2)}  \no \\
		 &\quad\quad + \|\T_t   \|_{\cB(\cH_+^2,\fH^4)}
		   \|R_t(z)  \|_{\cB(\cH^2,\cH_+^2)}
		   \left \|\T_t - \T_{t_0} \right \|_{\cB(\cH_+^2,\fH^4)} \left \|  R_{t_0}(z)\right \|_{\cB(\cH^2,\cH_+^2)}, \no \\
		   \begin{split}\label{eq:res_dif_norm_bound}
		&\quad \leq C \big\{ \|R_t({z})\|_{\cB(\cH^2)} \|  (V_t -V_{t_0} )\|_{\cB(\cH^2)} +	\left  \|R_t(z)\right \|_{\cB(\cH^2,\cH_+^2)} \left  \|(P_t-P_{t_0})\oplus (Q_t-Q_{t_0})\right \|_{\cB(\fH^4)} \\
		& \quad \qquad \qquad\qquad\qquad \qquad\qquad\qquad \qquad+	\left  \|R_t(z)\right \|_{\cB(\cH^2,\cH_+^2)}\left \|\T_t - \T_{t_0} \right \|_{\cB(\cH_+^2,\fH^4)}\big \},
			\end{split}
		\end{align}
		for some $C=C(z)>0$. Further employing \eqref{eq:3.11.2} and the triangle inequality  we arrive at
			\begin{align}
				\begin{split}
			&\left \|R_t(z) - R_{t_0}(z) \right \|_{\cB(\cH^2)} \\
			&\qquad \leq C \big\{ \|R_t({\zeta})\|_{\cB(\cH^2)} \|  (V_t -V_{t_0} )\|_{\cB(\cH^2)} \,+ \\
			& \qquad \qquad \qquad + \left (1 +  \|R_t(z)\|_{\cB(\cH^2)}\right ) \big ( \left  \|(P_t-P_{t_0})\oplus (Q_t-Q_{t_0})\right \|_{\cB(\fH^4)}+\left \|\T_t - \T_{t_0} \right \|_{\cB(\cH_+^2,\fH^4)}\big )\big\},
				\end{split}
		\end{align}
	for some $C=C(z)>0$. Next, we use the triangle inequality in the right-hand side above and rearrange terms to get
			\begin{align}
				\begin{split}
			&\left \|R_t(z)\right \|_{\cB(\cH^2)}\left(1 -C\big( \|  (V_t -V_{t_0} )\|_{\cB(\cH^2)}+\left  \|(P_t-P_{t_0})\oplus (Q_t-Q_{t_0})\right \|_{\cB(\fH^4)} + \left \|\T_t - \T_{t_0} \right \|_{\cB(\cH_+^2,\fH^4)}  \big  )	\right)\\
			&\qquad \leq  \left\|R_{t_0}(z) \right \|_{\cB(\cH^2)}+C \left (\left  \|(P_t-P_{t_0})\oplus (Q_t-Q_{t_0})\right \|_{\cB(\fH^4)}+\left \|\T_t - \T_{t_0} \right \|_{\cB(\cH_+^2,\fH^4)}\right ),
				\end{split}
		\end{align}
		which together with 
		\begin{equation}\label{eq:lag_o1}
		\|  (V_t -V_{t_0} )\|_{\cB(\cH^2)}+\left\|(P_t-P_{t_0})\oplus (Q_t-Q_{t_0})\right \|_{\cB(\fH^4)} + \left \|\T_t - \T_{t_0} \right \|_{\cB(\cH_+^2,\fH^4)}=o(1), \ t\rightarrow t_0,
	\end{equation}
		 yields the first equality in \eqref{eq:res_bound}.  Finally, \eqref{resolvent_continuous_H_H+} follows from \eqref{norm_bound_prop33}, \eqref{eq:res_bound}, \eqref{eq:res_dif_norm_bound}, and \eqref{eq:lag_o1}.
	\end{proof}

	In the following theorem we show that the continuity, Lipschitz continuity and differentiability of the mappings $t\mapsto V_t$, $t \mapsto \T_t,t\mapsto Q_t, t\mapsto P_t$ imply the continuity, Lipschitz continuity and differentiability, respectively, of the mapping $t\mapsto R_t(\zeta)$. 
	
	\begin{theorem}\label{thm:resolvent_regularity}
Assume \cref{hyp:non_empty_res_set}. Suppose  that $\zeta_0 \notin \spec(\cN_{t_0}+V_{t_0})$  and define
\begin{equation}
	\cU_\e \coloneqq \{ (t,\zeta) \in (0,1)\times \C : |t-t_0|\leq \e, |\zeta-\zeta_0|\leq \e\}\,\,\,\text{for} \,\,\,\e>0.
\end{equation} 
		\begin{enumerate}
			\item Suppose that the mappings $t\mapsto V_t$, $t \mapsto \T_t,t\mapsto Q_t, t\mapsto P_t$ are continuous at $t_0$. Then there exists an $\e>0$ such that if $(t,\zeta)\in \cU_\e$ then $\zeta \notin \spec(\cN_t + V_t)$, and the operator valued function $t\mapsto R_t(\zeta) = (\cN_t+V_t -\zeta)^{-1}\in\cB(\cH^2)$ is continuous at $t_0$ uniformly for $|\zeta-\zeta_0| < \e$.
			\item Suppose  that the mappings $t\mapsto V_t$, $t \mapsto \T_t,t\mapsto Q_t, t\mapsto P_t$ are Lipschitz continuous at $t_0$. Then there exists a constant $c>0$ such that for all $(t,\zeta)\in\cU_\e$, we have
			\begin{equation}\label{resolvent_difference_norm}
				\|R_t(\zeta) - R_{t_0}(\zeta)\|_{\cB(\cH^2)} \leq c\left |t-t_0 \right |.
			\end{equation} 
			\item Suppose that the mappings $t\mapsto V_t$, $t \mapsto \T_t,t\mapsto Q_t, t\mapsto P_t$ are differentiable at $t_0$. Then for some $\e>0$ the following asymptotic expansion holds uniformly for $|\zeta - \zeta_0|<\e$,
			\begin{align}
				\begin{split}\label{Rt_asymp_exp}
					R_t(\zeta) &\underset{t\to t_0}{=} R_{t_0}(\zeta) + \Big ( - R_{t_0}(\zeta) \dot V_{t_0}  R_{t_0}(\zeta) + \tau\left ( \T_{t_0} R_{t_0}(\bar\zeta)   \right )^*  \big (\dot P_{t_0} J \oplus  -\dot Q_{t_0} J\big ) \T_{t_0} R_{t_0}(\zeta)  \\
					& \qquad \qquad \qquad \qquad +\tau\left ( \T_{t_0} R_{t_0}(\bar\zeta)   \right )^*  \cJ \dot\T_{t_0} R_{t_0}(\zeta)\Big ) (t-t_0) +o(t-t_0) \quad in\,\, \cB(\cH^2).
				\end{split}
			\end{align}
		\end{enumerate}
		\end{theorem}
	\begin{proof}
		Assertion (1) follows immediately from assertion (2) in \cref{prop:resolvent_bounded_continuous}. For (2), by assertion (1) we have 
		\begin{equation}\label{sup1}
			\sup\left \{ \|R_t(\zeta)\|_{\cB(\cH^2)} :(t,\zeta) \in \cU_\e \right \} < \infty.
		\end{equation}
		We claim that yet a smaller choice of $\e>0$ gives
		\begin{equation}\label{sup2}
			\sup\left \{ \|R_t(\zeta)\|_{\cB(\cH^2,\cH_+^2)} :(t,\zeta) \in \cU_\e \right \} < \infty.
		\end{equation}
		Indeed, by the resolvent identity we have 
		\begin{equation}
			R_t(\zeta) = R_t(z) - (z-\zeta)R_t(z)R_t(\zeta).
		\end{equation}
		Using this and the second equality in \eqref{eq:res_bound},  \eqref{sup1} then implies \eqref{sup2}. Next, by the resolvent difference formula \eqref{resolvent_diff2new} and \eqref{sup2}, we infer
		\begin{align*}
			R_t(\zeta) - R_{t_0}(\zeta) &= R_t(\zeta)(V_{t_0}-V_t)R_{t_0}(\zeta) \\
			& \qquad + \tau (\T_t R_t(\bar{\zeta}))^*( (P_t - P_{t_0})\oplus (Q_{t} - Q_{t_0}) ) \cJ (P_{t_0} \oplus Q_{t_0}) \T_{t_0} R_{t_0}(\zeta) \\
			& \qquad \quad + \tau (\T_t R_t(\bar{\zeta}))^* \cJ (\T_t - \T_{t_0}) R_{t_0}(\zeta).
		\end{align*}		
		Hence
		\begin{align*}
			\| 	R_t(\zeta) - R_{t_0}(\zeta)  \|_{\cB(\cH^2)}& \leq c \max\Big\{ \|( (P_t - P_{t_0})\oplus (Q_{t} - Q_{t_0}) )\|_{\cB(\fH^4)}, \\
			& \qquad \qquad \qquad \qquad \qquad \qquad\|\T_t - \T_{t_0}\|_{\cB(\cH_+^2,\fH^4)}, \| V_t - V_{t_0} \|_{\cB(\cH^2)} \Big\},
		\end{align*}
		for some $c>0$ and all $(t,\zeta)\in\cU_\e$, where we used that 
		\begin{equation}
			\| \T_t R_t(\zeta)\|_{\cB(\cH^2,\fH^4)} \leq \|\T_t\|_{\cB(\cH_+^2, \fH^4)} \| R_t(\bar{\zeta}) \|_{\cB(\cH^2,\cH_+^2)}.
		\end{equation}
		For assertion (3), we first note that the mapping $t\mapsto R_t(\zeta)\in\cB(\cH^2,\cH_+^2)$ is continuous at $t_0$, i.e. 
		\begin{equation}
			\|R_t(\zeta) - R_{t_0}(\zeta)\|_{\cB(\cH^2, \cH_+^2)} = o(1), \quad t\to t_0,
		\end{equation}
		uniformly for $|\zeta - \zeta_0|<\e$ for $\e>0$ as in \eqref{sup2}.		Next, by assumption we have
		\begin{align}
			\begin{split}\label{expansions}
				Q_t &\underset{t\to t_0}{=} Q_{t_0} + \dot Q_{t_0} (t-t_0) + o (t-t_0); \qquad P_t \underset{t\to t_0}{=}P_{t_0} + \dot P_{t_0} (t-t_0) + o (t-t_0), \\
				V_t &\underset{t\to t_0}{=} V_{t_0} + \dot V_{t_0} (t-t_0) + o (t-t_0); \qquad T_t\underset{t\to t_0}{=} \T_{t_0} + \,\dot \T_{t_0} (t-t_0) + o(t-t_0).
			\end{split}
		\end{align}
		Substituting $s=t_0$ into the resolvent difference formula \eqref{resolvent_diff2new}, we have, using \eqref{resolvent_difference_norm} and \eqref{expansions},
		\begin{align*}
			R_t(\zeta) &- R_{t_0}(\zeta) \, = R_t({\zeta})\left ( V_{t_0} -V_t \right ) R_{t_0}(\zeta) + \tau\left ( \T_t R_t(\bar{\zeta})\right )^* \Big ((P_t-P_{t_0}) J \oplus \left (-(Q_t-Q_{t_0}) J \right )\Big )\T_{t_0} R_{t_0}(\zeta) \\
			& \qquad \qquad \qquad \qquad \qquad \qquad \qquad +\tau \left ( \T_t R_t(\bar{\zeta})\right )^* \cJ(\T_t-\T_{t_0}) R_{t_0}(\zeta), \\
			&\hspace{-1.5mm}\underset{t\to t_0}{=} \left (R_{t_0}(\zeta) + \mathcal{O}(t-t_0) \right ) \left ( -\dot V_{t_0} (t-t_0) + o (t-t_0) \right ) R_{t_0}(\zeta)  \\
			& \qquad + \tau\left ( \left ( \T_{t_0} + \cO(t-t_0)\right )\left (R_{t_0}(\bar\zeta) + \cO_{\|\cdot \|_{\cB\left (\cH^2,\cH_+^2\right )}}(1) \right )  \right )^*\times \\
			& \qquad\Big ((\dot P_{t_0}(t-t_0) + o(t-t_0)) J\oplus - (\dot Q_{t_0}(t-t_0) +o(t-t_0)) J \Big) \T_{t_0} R_{t_0}(\zeta) \\
			& \qquad + \tau \left ( \left ( \T_{t_0} +  \cO(t-t_0)\right )\left (R_{t_0}(\bar\zeta) + \cO_{\|\cdot \|_{\cB\left (\cH^2,\cH_+^2\right )}}(1) \right )  \right )^* \cJ(\dot \T_{t_0} (t-t_0)+o(t-t_0)) R_{t_0}(\zeta), \\
			&\hspace{-2mm}\underset{t\to t_0}{=} \Big(  - R_{t_0}(\zeta) \dot V_{t_0} R_{t_0}(\zeta)  +  \tau \left ( \T_{t_0} R_{t_0}(\bar\zeta)\right )^*  \big (\dot P_{t_0} J \oplus  -\dot Q_{t_0} J\big ) \T_{t_0} R_{t_0}(\zeta)  \\
			& \qquad \qquad \qquad \qquad +\tau\left ( \T_{t_0} R_{t_0}(\bar\zeta)   \right )^*  \cJ \dot\T_{t_0} R_{t_0}(\zeta)  \Big) (t-t_0) + o(t-t_0),
		\end{align*}
		in $\cB(\cH^2)$ uniformly for $|\zeta - \zeta_0 | < \e$. This proves \eqref{Rt_asymp_exp}. 
	\end{proof}
	
	Having established the necessary properties of the resolvent operator in $t$, we now focus our attention on deriving the Hadamard formula. In the remainder of this section, we work under the following assumption. 
	\begin{hypo}\label{hypo:simplicity}
		At a given $t_0\in[0,1]$, we suppose that $\la_0$ is an isolated, (algebraically) simple eigenvalue of $\cN_{t_0}+V_{t_0}$. We denote by $\gamma$ the following contour surrounding $\la_0$,
		\begin{equation}
			\gamma = \{ z\in \C : 2|z-\la_0| = \dist(\la_0,\spec(\cN_{t_0}+V_{t_0})\backslash\{\la_0\})\}.
		\end{equation}
	\end{hypo}
	By \cref{thm:resolvent_regularity}, for $t$ near $t_0$ the contour $\gamma$ encloses a simple eigenvalue $\lambda(t)\in\spec(\cN_t+V_t)$. Corresponding to this eigenvalue is the Riesz projection 	\begin{equation}\label{Pt_defn}
		\mathsf{P}(t) \coloneqq -\f{1}{2\pi i}\int_{\gamma} R_t(\zeta) d\zeta, \qquad R_t(\zeta) = (\cN_t +V_t -\zeta)^{-1},
	\end{equation}
	where $\gamma$ has the anti-clockwise orientation, as well as the reduced resolvent
	\begin{equation}\label{S_red_res}
		S \coloneqq \f{1}{2\pi i} \int_{\gamma} (\zeta - \la_0)^{-1} R_{t_0}(\zeta) d\zeta,
	\end{equation}
	and the identity 
	\begin{equation}\label{identity1}
		\sP(t_0)R_{t_0}(\zeta) = (\la_0-\zeta)^{-1}\sP(t_0).
	\end{equation}
\begin{rem}
	We note that \eqref{identity1} holds due to our assumption that $\lambda_0$ is a simple eigenvalue. It would still continue to hold for semi-simple eigenvalues (i.e., those with equal algebraic and geometric multiplicities). However, in the more general case of non-semi-simple eigenvalues, the right-hand side of \eqref{identity1} contains a nilpotent part, cf., e.g., \cite[eq. (5.23) in Chapter I.5]{Kato}, which would enter a Hadamard-type formula for the derivative of eigenvalue curves. 
\end{rem}
	\begin{rem}\label{rem:bounded_in_H_to_H+}
		Due to continuity of the mapping $\zeta\mapsto R_t(\zeta)\in \cB(\cH^2, \cH^2_+)$ we can view $\sP(t), S$ as operators in $\cB(\cH^2, \cH_+^2)$ and we have
		\begin{align}
			\begin{split} \label{complex_integral_trace}
				\f{1}{2\pi i}\int_{\gamma}^{} \T_t\left ( (\zeta-\la_0)^{-1}R_t(\zeta) \right ) d\zeta &= \T_t \f{1}{2\pi i}\int_{\gamma}^{}\left ( (\zeta-\la_0)^{-1}R_t(\zeta) \right ) d\zeta = \T_t S, \\
				(\T_t \sP(t))&\in\cB(\cH^2, \fH^4).
			\end{split} 
		\end{align}
	\end{rem}
	We remark that, since $\cN_t + V_t$ is not self-adjoint, $\sP(t)$ and $S$ are not self-adjoint, see for example \cite[Proposition I.2.5]{Gohberg_Goldberg_Kaashoek}. However, using \eqref{N*tau} and \eqref{V*tau} we have
	\begin{align}
		(\tau \sP(t))^* = \sP(t)^*\tau = \f{1}{2\pi i}\int_{\bar\gamma} (\cN_t^* +V_t^* -\zeta)^{-1}\tau \,d\zeta = \f{-1}{2\pi i}\int_{\gamma}\tau  ( \cN_t + V_t - \zeta)^{-1} d\zeta = \tau \sP(t),
	\end{align}
	and, similarly, $(\tau S)^* = S^*\tau = \tau S$.

	Towards the Hadamard formula, we need to write down an asymptotic expansion for $\sP(t) (\cN_t+V_t) \sP(t)$ for $t$ near $t_0$. Since the domain of $\sP(t) (\cN_t+V_t) \sP(t)$ is $t$-dependent, we introduce the following transformation operators which allow us to work with a similar operator for which the domain is fixed, cf. \cite[Section I.4.6]{Kato}. To that end, we define $\sD(t) \coloneqq \sP(t) - \sP(t_0)$ satisfying $\|\sD(t)\|_{\cB(\cH^2)} = o(1)$ as $t\to t_0$, which follows from \eqref{resolvent_difference_norm} and the definition of $\sP(t)$ in  \eqref{Pt_defn}. The following operators are then well defined for $t$ near $t_0$:
	\begin{align}\label{UandU-1}
		\begin{split}
			\sU(t) &\coloneqq (I - \sD^2(t))^{-1/2}\Big(  (I - \sP(t))  (I - \sP(t_0))  + \sP(t)\sP(t_0) \Big), \\
			\sU(t)^{-1} &= \Big (   (I - \sP(t_0))  (I - \sP(t)) + \sP(t_0) \sP(t) \Big ) (I - \sD^2(t))^{-1/2}.
		\end{split}
	\end{align}
	Moreover, as in  \cite[Section I.4.6]{Kato} and \cite[Proposition 2.18]{F04}, we have that
	\begin{equation}
		\sU(t) \sP(t_0) = \sP(t) \sU(t), 
	\end{equation}
	and that $\sU(t)$ maps $\ran \sP(t_0)\subset \cH_+^2$ onto $\ran \sP(t)\subset \cH_+^2$ for $t$ near $t_0$. In contrast with the analysis in \cite{LS20first}, $\cN_t+V_t$ is not self-adjoint, and thus the Riesz projection $\sP(t)$ is not self-adjoint, and therefore not an orthogonal projection. Hence, $\sU(t)$ is not unitary. (In fact, one can show that $\sU(t)^{-1} = \tau \sU(t)^*\tau$.) Nonetheless, only minor tweaks to the proof of \cite[Lemma 3.24]{LS20first} are required for the following lemma.
	
	\begin{lemma}
		For a given $t_0$, we assume that the mappings $t\mapsto V_t$, $t \mapsto \T_t,t\mapsto Q_t, t\mapsto P_t$ are differentiable at $t_0$, and that \cref{hypo:simplicity} holds. We then have
		\begin{equation}
			\begin{split}\label{operator_expansion}
				& \sP(t_0) \sU(t)^{-1} (\cN_t +V_t) \sP(t) \sU(t) \sP(t_0) \underset{t\to t_0}{=} \la_0 \sP(t_0) + \Big(\sP(t_0) \dot V_{t_0} \sP(t_0)  \\
				& \,- \tau \left ( \T_{t_0} \sP(t_0)\right  )^* \left ( \dot P_{t_0} \oplus \dot Q_{t_0} \right )\cJ \T_{t_0} \sP(t_0)  - \tau\left ( \T_{t_0} \sP(t_0)\right  )^* \cJ \dot\T_{t_0} \sP(t_0) \Big)(t-t_0) + o(t-t_0).
			\end{split}
		\end{equation}
	\end{lemma}
	\begin{proof}
		We expand the left hand side by making use of the resolvent expansion \eqref{Rt_asymp_exp} for $t$ near $t_0$. Multiplying \eqref{Rt_asymp_exp} on the right by $\sP(t_0)$ and using the identity 
		\begin{equation}\label{RPPR}
			R_{t_0}(\zeta)\sP(t_0) = \sP(t_0)R_{t_0}(\zeta) = (\la_0-\zeta)^{-1} \sP(t_0),
		\end{equation} 
		we find that 
		\begin{align}
			\begin{split}\label{Rt_asymp_exp_Pt0}
				R_t(\zeta)\sP(t_0) \underset{t\to t_0}{=}& (\la_0-\zeta)^{-1} \sP(t_0) + (\la_0-\zeta)^{-1} \Big ( - R_{t_0}(\zeta) \dot V_{t_0}  \sP(t_0)  \\
				&\qquad + \tau\left ( \T_{t_0} R_{t_0}(\bar\zeta)   \right )^*  \big (\dot P_{t_0} J \oplus  -\dot Q_{t_0} J\big ) \T_{t_0} \sP(t_0) \\
				& \qquad \qquad +\tau\left ( \T_{t_0} R_{t_0}(\bar\zeta)   \right )^*  \cJ \dot\T_{t_0} \sP(t_0)\Big ) (t-t_0) +o(t-t_0).
			\end{split}
		\end{align}
		Similarly, multiplying \eqref{Rt_asymp_exp} on the left by $\sP(t_0)$, using \eqref{RPPR} and that $\sP(t_0)\tau = \tau \sP(t_0)^*$, we have
		\begin{align}
			\begin{split}\label{Pt0_Rt_asymp_exp}
				\sP(t_0)R_t(\zeta) \underset{t\to t_0}{=}& (\la_0-\zeta)^{-1} \sP(t_0) + (\la_0 -\zeta)^{-1} \Big ( - \sP(t_0) \dot V_{t_0}  R_{t_0}(\zeta) \\
				&\qquad + \tau\left ( \T_{t_0} \sP(t_0)   \right )^*  \big (\dot P_{t_0} J \oplus  -\dot Q_{t_0} J\big ) \T_{t_0} R_{t_0}(\zeta) \\
				& \qquad \qquad +\tau\left ( \T_{t_0} \sP(t_0)  \right )^*  \cJ \dot\T_{t_0} R_{t_0}(\zeta) \Big ) (t-t_0) +o(t-t_0).
			\end{split}
		\end{align}
		The proof is split into several steps.
		
		\textbf{Step 1.} We have
			\begin{equation}\label{step1}
				\sP(t_0) \sP(t)\sP(t_0) \underset{t\to t_0}{=} \sP(t_0) + o(t-t_0).
			\end{equation}
		\begin{proof}
			For any continuous $F:\gamma \to \cB(\cH^2,\fH^4)$ we have
			\begin{equation*}
				\left ( \int_\gamma F(\zeta) d\zeta\right )^* = -\int_\gamma (F(\bar\zeta))^* d\zeta.
			\end{equation*}
			Applying this to $F(\zeta)=\f{1}{2\pi i}(\la_0 -\zeta)^{-1}\T_{t_0}R_{t_0}(\zeta)$, and using \eqref{S_red_res}, \eqref{complex_integral_trace}, we have
			\begin{align*}
				\int_\gamma \left (\f{1}{2\pi i}(\la_0 -\bar\zeta)^{-1}\T_{t_0}R_{t_0}(\bar\zeta)\right )^* d\zeta = \left (-\int_\gamma \f{1}{2\pi i}(\la_0 -\zeta)^{-1}\T_{t_0} R_{t_0} (\zeta) d\zeta\right )^* = (\T_{t_0} S)^*.
			\end{align*}
			Using this, multiplying both sides of \eqref{Rt_asymp_exp_Pt0} by $-\f{1}{2\pi i}$ and integrating over $\gamma$, we obtain
			\begin{align}
				\begin{split}\label{PtPt0}
				\sP(t) \sP(t_0) &\underset{t\to t_0}{=} \sP(t_0) + \Big ( -S \dot{V}_{t_0} \sP(t_0)  + \tau \left ( \T_{t_0} S \right )^*\big (\dot P_{t_0} J \oplus  -\dot Q_{t_0} J\big ) \T_{t_0} \sP(t_0) \\
				& \qquad \qquad \qquad \qquad \qquad + \tau \left ( \T_{t_0} S \right )^*\cJ \dot\T_{t_0} \sP(t_0)\Big )(t-t_0) + o(t-t_0).			
				\end{split}
			\end{align}
			Similarly, multiplying both sides of \eqref{Pt0_Rt_asymp_exp} by $-\f{1}{2\pi i}$ and integrating over $\gamma$, we obtain
			\begin{align}
				\begin{split}\label{Pt0Pt}
					\sP(t_0) \sP(t) &\underset{t\to t_0}{=}  \sP(t_0) + \Big ( - \sP(t_0) \dot{V}_{t_0} S  + \tau \left ( \T_{t_0} \sP(t_0)\right )^*\big ( \dot P_{t_0} J \oplus  -\dot Q_{t_0} J\big ) \T_{t_0} S  \\
					&\qquad \qquad \qquad \qquad \qquad + \tau \left ( \T_{t_0} \sP(t_0)\right )^*\cJ\T_{t_0} S \Big )(t-t_0) + o(t-t_0).		
				\end{split}
			\end{align}
			Multiplying on the right by $\sP(t_0)$ and using $S\sP(t_0)=0$, we arrive at \eqref{step1}.
		\end{proof}
		\textbf{Step 2.} We have
			\begin{align}
				\sP(t_0) \sU(t) \sP(t_0) &\underset{t\to t_0}{=} \sP(t_0) + o(t-t_0), \label{step2_PUP}\\
				\sP(t_0) \sU(t)^{-1} \sP(t_0) & \underset{t\to t_0}{=} \sP(t_0) + o(t-t_0), \label{step2_PU-1P}\\
				\begin{split}\label{I-PUP}
				\left (I - \sP(t_0) \right ) \sU(t) \sP(t_0) &\underset{t\to t_0}{=} (I-\sP(t_0)) \Big ( -S \dot{V}_{t_0} \sP(t_0)  + \tau \left ( \T_{t_0} S \right )^*\big (\dot P_{t_0} J \oplus  -\dot Q_{t_0} J\big ) \T_{t_0} \sP(t_0)  \\
				&\hspace{-4em} \qquad \qquad\qquad + \tau \left ( \T_{t_0} S \right )^*\cJ \dot\T_{t_0} \sP(t_0)\Big )(t-t_0) + o(t-t_0),				
				\end{split} \\
				\begin{split}\label{I-PU-1P}
					\sP(t_0) \sU(t)^{-1} (I- \sP(t_0)) &\underset{t\to t_0}{=} \Big ( -\sP(t_0) \dot{V}_{t_0} S + \tau \left ( \T_{t_0} \sP(t_0) \right )^*\big (\dot P_{t_0} J \oplus  -\dot Q_{t_0} J\big ) \T_{t_0} S\\
					&\hspace{-4em} \qquad \qquad\qquad + \tau \left ( \T_{t_0} \sP(t_0) \right )^*\cJ \dot\T_{t_0} S\Big )(I-\sP(t_0))(t-t_0) + o(t-t_0).
				\end{split} 
			\end{align}
			\begin{proof}
				Using \eqref{Pt_defn} and \eqref{resolvent_difference_norm}, we have 
				\begin{equation*}
					\sD(t) = \sP(t) - \sP(t_0) = -\f{1}{2\pi i}\int_{\gamma} R_t(\zeta) - R_{t_0}(\zeta) d\zeta \underset{t\to t_0}{=} \mathcal{O} (t-t_0),
				\end{equation*}
				hence
				\begin{align*}
					(I-\sD(t)^2)^{-1/2}  \underset{t\to t_0}{=}  I + \cO(|t-t_0|^2).
				\end{align*}
				Thus
				\begin{align}
					\sU(t) &= (I - \sD(t)^2)^{-1/2}\left ((I-\sP(t) )(I-\sP(t_0)) +\sP(t)\sP(t_0)\right ),  \no \\
					&\hspace{-2mm}\underset{t\to t_0}{=} \left ((I-\sP(t) )(I-\sP(t_0)) +\sP(t)\sP(t_0)\right ) + o(t-t_0),		 \label{UDP} \\			
					\sU(t)^{-1} &= \left ((I-\sP(t_0) )(I-\sP(t)) +\sP(t_0)\sP(t)\right )(I - \sD(t)^2)^{-1/2},  \no \\
					&\hspace{-2mm}\underset{t\to t_0}{=} \left ((I-\sP(t_0) )(I-\sP(t)) +\sP(t_0)\sP(t)\right ) + o(t-t_0).\label{U-1DP}
				\end{align}
				Using \eqref{UDP}, \eqref{U-1DP} and \eqref{step1} we find that
				\begin{align*}
					\sP(t_0) \sU(t)\sP(t_0) \hspace{-0.5mm}\underset{t\to t_0}{=} \sP(t_0)\sP(t)\sP(t_0) + o(t-t_0) \hspace{-0.5mm}\underset{t\to t_0}{=} \sP(t_0) +o(t-t_0), \\
					\sP(t_0) \sU(t)^{-1}\sP(t_0) \hspace{-0.5mm}\underset{t\to t_0}{=} \sP(t_0)\sP(t)\sP(t_0) + o(t-t_0) \hspace{-0.5mm}\underset{t\to t_0}{=} \sP(t_0) +o(t-t_0),
				\end{align*}
				as required. On the other hand, from \eqref{UDP} and \eqref{U-1DP} we have 
				\begin{align*}
					\left (I-\sP(t_0)\right ) \sU(t)\sP(t_0) \hspace{-0.5mm}\underset{t\to t_0}{=} \left (I-\sP(t_0)\right )\sP(t)\sP(t_0) + o(t-t_0), \\
					\sP(t_0) \sU(t)^{-1}\left (I-\sP(t_0)\right ) \hspace{-0.5mm}\underset{t\to t_0}{=} \sP(t_0)  \sP(t)\left (I-\sP(t_0)\right ) + o(t-t_0).
				\end{align*}
				Thus, multiplying \eqref{PtPt0} on the left by $(I-\sP(t_0))$ yields \eqref{I-PUP}, while multiplying  \eqref{Pt0Pt} on the right by $(I-\sP(t_0))$ yields \eqref{I-PU-1P}.
			\end{proof}
		\textbf{Step 3.} We have
			\begin{align}
				\begin{split}\label{res_op_exp}
				\sP(t_0)& \sU^{-1}(t) R_t(\zeta) \sU(t) \sP(t_0) \underset{t\to t_0}{=} (\la_0 -\zeta)^{-1} \sP(t_0) \\
				&+ (\la_0 -\zeta)^{-2} \Big( - \sP(t_0) \dot V_{t_0} \sP(t_0) + \tau \left (\T_{t_0}\sP(t_0)\right )^*\left (\dot P_{t_0} J \oplus -\dot Q_{t_0} J \right ) \T_{t_0} \sP(t_0) \\
				& \qquad \qquad \qquad \qquad + \tau \left ( \T_{t_0} \sP(t_0)\right )^*J\dot \T_{t_0} \sP(t_0)  \Big )(t-t_0) + o(t-t_0).
				\end{split}
			\end{align}
		\begin{proof}
			Sandwiching the factor $R_t(\zeta)$ in the left hand side by $ \sP(t_0) + (I- \sP(t_0))$, we have
			\begin{align*}
				\begin{split}
					\sP(t_0) \sU^{-1}(t) R_t(\zeta) \sU(t) \sP(t_0) &= \sP(t_0) \sU^{-1}(t) \Big( \sP(t_0) + (I- \sP(t_0)) \Big) R_t(\zeta) \\ 
					& \qquad \qquad \qquad \qquad \times \Big( \sP(t_0) + (I- \sP(t_0)) \Big) \sU(t) \sP(t_0), 			
				\end{split} \\
				& \eqqcolon I + II + III + IV.
			\end{align*}
			We compute each term individually. We have
			\begin{align*}
				I \coloneqq \sP(t_0)\sU^{-1}(t) (I- \sP(t_0)) \times (I- \sP(t_0)) R_t(\zeta) \sP(t_0) \times \sP(t_0) \sU(t) \sP(t_0)\underset{t\to t_0}{=} o(t-t_0),
			\end{align*}
			which follows from \eqref{Rt_asymp_exp_Pt0}, \eqref{step2_PUP} and \eqref{I-PU-1P}, since the terms $ (I- \sP(t_0)) R_t(\zeta) \sP(t_0) $ and  $\sP(t_0)\sU^{-1}(t) (I- \sP(t_0)) $ are both $\cO(t-t_0)$. We similarly infer that
			\begin{align*}
				II \coloneqq \sP(t_0) \sU^{-1}(t)  \sP(t_0) \times \sP(t_0) R_t(\zeta)(I- \sP(t_0)) \times (I- \sP(t_0)) \sU(t) \sP(t_0) \underset{t\to t_0}{=} o(t-t_0),
			\end{align*}
			using \eqref{Pt0_Rt_asymp_exp}, \eqref{step2_PU-1P} and \eqref{I-PUP}, and 
			\begin{align*}
				III \coloneqq \sP(t_0) \sU^{-1}(t)  (I- \sP(t_0)) \times  R_t(\zeta) \times (I- \sP(t_0)) \sU(t) \sP(t_0) \underset{t\to t_0}{=} o(t-t_0)
			\end{align*}
			by \eqref{I-PUP} and \eqref{I-PU-1P}. Finally, the term
			\begin{align}
				IV \coloneqq \sP(t_0) \sU^{-1}(t) \sP(t_0) \times \sP(t_0) R_t(\zeta)\sP(t_0) \times \sP(t_0) \sU(t) \sP(t_0)
			\end{align}
			admits the sought after expansion \eqref{res_op_exp}, upon using \eqref{Rt_asymp_exp_Pt0}, \eqref{RPPR}, \eqref{step2_PUP} and \eqref{step2_PU-1P}.
		\end{proof}
		\textbf{Step 4.} Recalling the identities
		\begin{equation}
			(\cN_t + V_t) \sP(t) \coloneqq \f{-1}{2\pi i}\int_{\gamma}^{}\zeta R_t(\zeta) d\zeta, \qquad \f{1}{2\pi i} \int_{\gamma}^{} \zeta(\la_0 -\zeta)^{-1}d\zeta =1, 
		\end{equation}
		multiplying \eqref{res_op_exp} by $-\zeta/2\pi i$ and integrating over $\gamma$, we arrive at \eqref{operator_expansion}.
	\end{proof}
	\begin{theorem}\label{thm:Hadamard_via_resolvents}
		Assume \cref{hypo:simplicity} and that the mappings $t \mapsto \T_t, t\mapsto V_t,t\mapsto P_t, t \mapsto Q_t,$ are differentiable at $t_0$. We define
		\begin{equation*}
			T^{(1)} \coloneqq \sP(t_0) \dot V_{t_0} \sP(t_0) - \tau \left (\T_{t_0} \sP(t_0)\right )^* \left ( \dot P_{t_0} J \oplus  -\dot Q_{t_0} J \right )\T_{t_0} \sP(t_0) - \tau \left (\T_{t_0} \sP(t_0)\right )^*\cJ \dot \T_{t_0} \sP(t_0),
		\end{equation*}
		and denote the eigenvalue and eigenvector of this one-dimensional operator by $\la_{t_0}^{(1)}$ and $\{ \mathbf{u}_{t_0} \} \subset \ran \sP(t_0) = \ker (\cN_{t_0} + V_{t_0} - \la_0 )$, respectively. Suppose further that $\lambda_0 \in\bbR$. Then, for $t$ near $t_0$, the eigenvalue $\la(t)$ of $\cN_t + V_t$ satisfies the asymptotic formula
		\begin{equation}\label{la_asymp_expansion}
			\la(t) = \la_0 + \la_{t_0}^{(1)}(t-t_0) +o(t-t_0).
		\end{equation}
		Moreover, one has
		\begin{equation}\label{Had_fmla_via_res}
			\la'(t_0) = \f{\langle \tau \dot{V}_{t_0} \mathbf{u}_{t_0}, \mathbf{u}_{t_0}\rangle_{\cH^2} + \Om \left (( \dot P_{t_0} \oplus \dot Q_{t_0} )\T_{t_0}\mathbf{u}_{t_0}, \T_{t_0}\mathbf{u}_{t_0}\right ) + \Om \left (\T_{t_0}\mathbf{u}_{t_0}, \dot{\T}_{t_0}\mathbf{u}_{t_0}\right )}{ \langle \tau \mathbf{u}_{t_0}, \mathbf{u}_{t_0}\rangle_{\cH^2}},
		\end{equation}
		where $\langle \tau \mathbf{u}_{t_0}, \mathbf{u}_{t_0}\rangle_{\cH^2}\neq 0$.
	\end{theorem}
	\begin{proof}
		Recalling that $\sU(t)$ is an isomorphism between $\ran \sP(t_0)$ and $\ran \sP(t)$, see \cite[Section I.4.6]{Kato}, \cite[Proposition 2.18]{F04}, we note that $(\cN_t + V_t)|_{\ran \sP(t)}$ is similar to 
		\[
			\sP(t_0) \sU(t)^{-1}(\cN_t + V_t)\sP(t) \sU(t) \sP(t_0) |_{\ran \sP(t_0)}
		\]
		for $t$ near $t_0$. Thus the eigenvalues of these operators coincide for all $t$ near $t_0$, see \cite[Section I.5.7]{Kato}. Asymptotically expanding the eigenvalues of the latter using the finite-dimensional first-order perturbation theory, specifically \cite[Theorem II.5.11]{Kato}, we deduce \eqref{la_asymp_expansion}.
		 
		 Next, applying $\langle \cdot, \tau \mathbf{u}_{t_0}\rangle$ to the eigenvalue equation $T^{(1)}\mathbf{u}_{t_0} = \la_{t_0}^{(1)}\mathbf{u}_{t_0}$, we find
		 \begin{align*}
		 	&\la_{t_0}^{(1)} \langle \mathbf{u}_{t_0}, \tau \mathbf{u}_{t_0}\rangle_{\cH^2} =  \left \langle  \sP(t_0) \dot V_{t_0} \sP(t_0) \mathbf{u}_{t_0}, \tau \mathbf{u}_{t_0}\right \rangle_{\cH^2} \\
		 	& \quad - \left \langle \tau \left (\T_{t_0} \sP(t_0)\right )^* \left (  \dot P_{t_0} \oplus \dot Q_{t_0} \right ) \cJ\T_{t_0} \sP(t_0)\mathbf{u}_{t_0}, \tau \mathbf{u}_{t_0}\right \rangle_{\cH^2}  -\left \langle \tau \left (\T_{t_0} \sP(t_0)\right )^* \cJ \dot\T_{t_0} \sP(t_0)\mathbf{u}_{t_0}, \tau \mathbf{u}_{t_0}\right \rangle_{\cH^2},  \\
		 	&\,\,=  \left \langle  \dot V_{t_0}\mathbf{u}_{t_0}, \sP(t_0) ^*\tau \mathbf{u}_{t_0}\right \rangle_{\cH^2}  - \left \langle \left (\dot P_{t_0}\oplus \dot Q_{t_0} \right )\cJ \T_{t_0} \mathbf{u}_{t_0}, \T_{t_0} \mathbf{u}_{t_0}\right \rangle_{\fH^4}  -\left \langle  \cJ \dot\T_{t_0} \mathbf{u}_{t_0},  \T_{t_0} \mathbf{u}_{t_0}\right \rangle_{\fH^4}  \\
		 	&\,\,=  \left \langle  \tau\dot V_{t_0}\mathbf{u}_{t_0},  \mathbf{u}_{t_0}\right \rangle_{\cH^2}  + \left \langle \cJ \left (\dot P_{t_0}\oplus \dot Q_{t_0} \right ) \T_{t_0} \mathbf{u}_{t_0}, \T_{t_0} \mathbf{u}_{t_0}\right \rangle_{\fH^4}  -\Om\left( \dot\T_{t_0} \mathbf{u}_{t_0},  \T_{t_0} \mathbf{u}_{t_0}\right), \\
		 	&\,\,=  \left \langle  \tau\dot V_{t_0}\mathbf{u}_{t_0},  \mathbf{u}_{t_0}\right \rangle_{\cH^2}  + \Om\left(\left (\dot P_{t_0}\oplus \dot Q_{t_0} \right )\T_{t_0} \mathbf{u}_{t_0},  \T_{t_0} \mathbf{u}_{t_0}\right)  +\Om\left( \T_{t_0} \mathbf{u}_{t_0},  \dot\T_{t_0} \mathbf{u}_{t_0}\right),
		 \end{align*}
		 where in the second last line, we used that $\sP(t_0)^*\tau = \tau \sP(t_0)$ , and that $\ran(P_t\oplus Q_t)$ is Lagrangian, hence $\cJ (P_t\oplus Q_t) + (P_t\oplus Q_t) \cJ = \cJ$ and $\cJ (\dot P_{t_0}\oplus \dot Q_{t_0}) = - (\dot P_{t_0}\oplus \dot Q_{t_0}) \cJ $. In the last line, we used that $ \Om\left( \dot\T_{t_0} \mathbf{u}_{t_0},  \T_{t_0} \mathbf{u}_{t_0}\right)\in \R$, which follows from differentiating $t\mapsto\Om\left( \T_{t} \mathbf{u}_{t_0},  \T_{t} \mathbf{u}_{t_0}\right)$ at 
		 $t_0$.
		 
		 Since $\la'(t_0) = \la_{t_0}^{(1)}$, \eqref{Had_fmla_via_res} follows. That $\langle \tau \mathbf{u}_{t_0}, \mathbf{u}_{t_0}\rangle_{\cH^2}\neq 0$ follows from the assumption that $\la_0 $ is simple. Indeed, in this case, there are no generalised eigenvectors, i.e. no solution $\hat{\mathbf{u}}\in\dom(\cN_{t_0} + V_{t_0} -\la_0 )$ to the inhomogeneous equation
		 	\[
		 	(\cN_{t_0} + V_{t_0} -\la_0 ) \hat{\mathbf{u}} = \mathbf{u}_{t_0}.
		 	\]
		 	Noting that $\ker(\cN_{t_0}^* + V_{t_0}^*-\la_0 ) = \{\tau \mathbf{u}_{t_0}\}$ (since $\lambda_0 \in\bbR$), the conclusion now follows from the Fredholm alternative (which is applicable due to $\lambda_0$ being a point in the discrete spectrum of $\cN_{t_0}+V_{t_0}$).		 	
		 \end{proof}
	
	\section{Hadamard formulas via Lyapunov-Schmidt}\label{sec:Hadamard_crossingforms}

	In this section we prove local existence results for eigenvalue curves $\la(t)\in \spec\left (\cN_{t}+V_{t}\right )$ or $\la\in \spec\left (\cN_{t(\la)}+V_{t(\la)}\right )$ and derive Hadamard-type formulas for their derivatives using Lyapunov–Schmidt reduction. The argument follows the strategy of \cite{CCLM23}, but the present setting is more delicate due to the $t$-dependence of the Lagrangian planes describing the boundary conditions. We handle this issue by carefully constructing  a family of isomorphisms which fixes the domain, see \eqref{435}.
	
	Throughout this section we assume \cref{hypo:QPGH_families}, and impose the following additional assumptions. The first ensures that $\cN_{t_0}+V_{t_0} - \la_0$ is Fredholm; the second gives the regularity needed for the Lyapunov Schmidt reduction; and the third is an abstract version of the unique continuation principle \cite{LS18}.
	\begin{hypo}\label{hypo_basic_assumptions}
		Given $\la_0\in\R$ and $t_0\in(0,1)$, we assume:
		\begin{enumerate}
			\item [\rm{(i)}] $\la_0\notin \esspec(\cN_{t_0}+V_{t_0})$. \vspace{1mm}
		\end{enumerate} 
		In addition, there exists an interval $\cI\subset[0,1]$ containing $t_0$ such that:
		\begin{itemize}
			\item [\rm{(ii)}] the mappings $t\mapsto V_t, t\mapsto P_t, t\mapsto Q_t, t\mapsto \T_t$ are $C^2$ on $\cI$, and
			\item [\rm{(iii)}] $\ker \big( (N^* + V_t^* -\la) \tau\big) \cap \dom(\tau N) =\{0\}$ for all $t$ in $\cI$.
		\end{itemize}
	\end{hypo}

	The objective of this section is to prove the following Hadamard formulas for the derivatives of geometrically simple eigenvalues. The crossing forms $\mathfrak{m}_{t_0}, \mathfrak{m}_{\la_0}$ appearing therein are defined in the next section. 
	
		\begin{theorem}\label{thm:Hadamard}
		Assume \cref{hypo_basic_assumptions}. Suppose further that $\dim\ker (\cN_{t_0}+V_{t_0}-\la_0)=1$ with eigenfunction $\mathbf{u}_{t_0}$, and let $\mathbf{q} = \T_{t_0} \mathbf{u}_{t_0}$. Then:
		\begin{enumerate}
			\item If $\mathfrak{m}_{\la_0} \neq 0$, there exists a $C^1$ curve $\la=\la(t)$, defined for $t$ near $t_0$, such that 
			\begin{equation}\label{hadamard1}
				\la(t) \in \spec\left (\cN_{t} + V_{t}\right), \qquad \la'(t_0) = - \f{\mathfrak{m}_{t_0}(\mathbf{q},\mathbf{q})}{\mathfrak{m}_{\la_0}(\mathbf{q},\mathbf{q})}.
			\end{equation}
			\item If $\mathfrak{m}_{t_0} \neq 0$, there exists a $C^1$ curve $t=t(\la)$, defined for $\la$ near $\la_0$, such that 
			\begin{equation}\label{hadamard2}
				\la\in\spec\left (\cN_{t(\la)} + V_{t(\la)} \right ), \qquad t'(\la_0) = -\f{\mathfrak{m}_{\la_0}(\mathbf{q},\mathbf{q})}{\mathfrak{m}_{t_0}(\mathbf{q},\mathbf{q})}.
			\end{equation}
		\end{enumerate}
	\end{theorem}

	\subsection{Maslov crossing forms}

	Towards the proof of \cref{thm:Hadamard}, we begin by computing crossing forms.
	The following discussion mimics that in \cite[\S 4.5]{LS20first}. Recall from \eqref{TtdomNt} that
	\begin{equation}\label{Ttilde2}
		\T_{t} (\dom \cA_t \oplus \dom \cB_t) = \ran P_t \times \ran Q_t,
	\end{equation}
	and also note that it follows from \eqref{N*} that $\dom(N^*\tau) = \dom A^*\oplus \dom A^*$. For $\la\in \R$ and $t\in[0,1]$, we consider the following families of subspaces:
	\begin{align}\label{cKcFcD}
		\begin{split}
			\K_{\la,t} &\coloneqq \T_{t}\big(\ker\big((N^* + V_t^*-\la)\tau \big)\big) \subset \fH^4, \\
			\mathcal{F}_t &\coloneqq \ran P_t \times \ran Q_t  \subset \fH^4, \\
			\Upsilon_{\la,t} &\coloneqq \K_{\la,t} \oplus \cF_t \subset \fH^4 \oplus \fH^4, \\
			\mathfrak{D}&\coloneqq \{ \mathbf{q} = (q,q)^\top : q\in \fH^4\} \subset \fH^4 \oplus \fH^4.
		\end{split}
	\end{align}
	Let us note that
	\begin{equation}\label{Lag_intersect_spectral_interp_iff}
		\ker(\cN_t +V_t - \la) \neq \{0\} \,\,\,\iff \,\,\, \K_{\la,t} \cap \cF_t \neq \{0\} \,\,\,\iff \,\,\, \Upsilon_{\la,t}\cap \mathfrak{D}\neq \{0\}.
	\end{equation}
In addition, as in the standard self-adjoint case the Cauchy data plane $\K_{\la,t}$ is Lagrangian for $\lambda\in\bbR$.
\begin{lemma}\label{lemma:_Cauchy_data_plane_is_Lagrangian}
	For all $\la\in\R$, the Cauchy data space $\K_{\la,t}$ is a Lagrangian subspace of $\fH^4$ with respect to the symplectic form $\Omega$ defined by $\cJ$ in \eqref{Om}.
\end{lemma}
\begin{proof}
The result follows from \cite[Proposition 3.5]{BBvnkF98}, but note that the vanishing of $\Om$ on $\K_{\la,t}$ follows immediately upon observing that $(N^* + V_t^*-\la)\tau = (\tau N)^* + V_t^*\tau - \la \tau$ is a bounded, symmetric perturbation of the adjoint operator $(\tau N)^*$ which satisfies Green's identity \eqref{modified_greens_real_la} with trace operator $\T_t$.
\end{proof}

To compute crossing forms, it will be convenient to view the Lagrangian plane $\Upsilon_{\la,t}$ as the graph of an infinitesimally symplectic operator, the regularity of which follows from previous results. Let $\Pi_{\la,t}$ be the orthogonal projection onto $\Upsilon_{\la,t}$ so that the mapping $t \mapsto  \Pi_{\la,t}$ is continuously differentiable on $[0,1]$ for each $\la\in\R$, see e.g. \cite[p. 480-481]{LS18}. For $\la\in\R$ and $t_0\in[0,1]$ satisfying \cref{hypo_basic_assumptions}, there is an interval $\sI\subseteq \cI$ centred around $t_0$ and a $C^1$ family of operators $t\mapsto \sM_{\la,t}, t\in \sI$ which is in $C^1(\sI, \cB\left (\Upsilon_{\la,t_0},  (\Upsilon_{\la,t_0})^{\top} \right ))$ with $\sM_{\la,t_0}=0$ such that 
\begin{equation}\label{sM_defn}
	\Upsilon_{\la,t} = \left \{\mathbf{q} + \sM_{\la,t}\mathbf{q} : \mathbf{q} \in \Upsilon_{\la,t_0}  \right \}, \quad t\in \sI,
\end{equation}
see, e.g. \cite[Lemma 3.8]{CJLS14}. Then $(\la_0,t_0)$ is called a conjugate point if $\ker(\cN_{t_0} + V_{t_0} -\la_0)\neq \{0\}$, or, equivalently,  either of the remaining assertions in \eqref{Lag_intersect_spectral_interp_iff} hold. The $t$ crossing form for the Lagrangian path $t\mapsto \Upsilon_{\la_0,t}$ with respect to $\mathfrak{D}$ on the finite-dimensional intersection $\Upsilon_{\la_0,t_0} \cap \fD$ is defined to be
\begin{align}\label{define_t_crossing_form}
	\mathfrak{m}_{t_0}(\mathbf{q},\mathbf{p}) \coloneqq \de{}{t}\big|_{t=t_0} \tilde \Om ( \mathbf{q}, \sM_{\la_0,t} \mathbf{p}), \qquad \mathbf{q}, \mathbf{p}\in \Upsilon_{\la_0,t_0},
\end{align}
where $\tilde \Om = \Om \oplus(- \Om)$ is a symplectic form on $\fH^4\oplus \fH^4$. Similarly,  the $\la$ crossing form for $\la\mapsto \Upsilon_{\la,t_0}$ with respect to $\mathfrak{D}$ is defined as
\begin{align}\label{define_la_crossing_form}
	\mathfrak{m}_{\la_0}(\mathbf{q},\mathbf{p}) \coloneqq \de{}{\la}\big|_{\la=\la_0} \tilde \Om (  \mathbf{q}, \sM_{\la,t_0}\mathbf{p}), \qquad \mathbf{q},\mathbf{p}\in \Upsilon_{\la_0,t_0} \cap \mathfrak{D}.
\end{align}
The following lemma is an analogue of \cite[Lemma 4.19]{LS20first}. The families $t\mapsto \mathbf{w}_t, \la\mapsto \mathbf{w}_\la$ described therein will be used to derive tractable expressions for the crossing forms defined above. 
\begin{lemma}\label{lemma:t_la_family}
	Suppose $(\la_0,t_0)$ is a conjugate point, i.e. $\ker(\cN_{t_0} + V_{t_0} -\la_0)\neq \{0\}$, which satisfies \cref{hypo_basic_assumptions}, and let $\mathbf{u}_{t_0}\in\ker(\cN_{t_0} + V_{t_0} -\la_0)$. There exists an open interval $\sI\subseteq \cI$ around $t_0$, a family $t\mapsto \mathbf{w}_t$ in $C^1(\sI, \cH_+^2)$ and a family $t\mapsto g_t \in \cF_t$ in $C^1(\sI, \fH^4)$ such that 
	\begin{subequations}\label{w_t_g_t_families}
		\begin{align}
			\mathbf{w}_{t_0} &= \mathbf{u}_{t_0}, \quad g_{t_0} = \T_{t_0} \mathbf{u}_{t_0}, \\
			\mathbf{w}_t &\in \ker(N^*+ V_t^* - \la_0)\tau, \\
			(\T_{t} \mathbf{w}_t, g_t)^{\top} &= (\T_{t_0} \mathbf{u}_{t_0}, \T_{t_0} \mathbf{u}_{t_0})^{\top} + \sM_{\la_0,t}(\T_{t_0} \mathbf{u}_{t_0}, \T_{t_0} \mathbf{u}_{t_0})^{\top}. \label{412c}
		\end{align}
	\end{subequations}
	Similarly, under the same assumptions, there exists an open interval $\sI$ around $\la_0$ and a family $\la\mapsto \mathbf{w}_\la$ in $C^\infty(\sI, \cH_+^2)$ such that 
	\begin{subequations}\label{w_la_family}
		\begin{align}
			\mathbf{w}_{\la_0} &= \mathbf{u}_{t_0},  \\
			\mathbf{w}_\la &\in \ker(N^*+ V_{t_0}^* - \la)\tau, \\
			(\T_{t_0} \mathbf{w}_\la , \T_{t_0} \mathbf u_{t_0})^{\top}&=( \T_{t_0} \mathbf{u}_{t_0}, \T_{t_0} \mathbf u_{t_0})^{\top} + \sM_{\la, t_0} (\T_{t_0} \mathbf{u}_{t_0}, \T_{t_0} \mathbf u_{t_0})^{\top}
		\end{align}
	\end{subequations}
\end{lemma}
\begin{proof}
	We begin with \eqref{w_t_g_t_families}. The proof is almost identical to that of \cite[Lemma 4.19]{LS20first}, but we give it here for completeness. Let us denote $K_t \coloneqq \ker\left ( (N^*+V_t^* - \la_0)\tau\right )$, so that $\K_{\la_0,t} = \T_t K_t$. We also denote $q \coloneqq \T_{t_0} \mathbf{u}_{t_0}$, $\mathbf{q} \coloneqq (q,q)$, and let $\bbP_t$ be the orthogonal projection onto $\K_{\la_0,t}$. Then $\bbP_t \in C^1(\sI,\cB(\fH^4))$ for some open interval $\sI \subseteq \cI $ centred at $t_0$ (see, for example, \cite[Theorem 3.9]{BBvnkF98}). We define projections in $\fH^4 \oplus \fH^4$,
	\begin{equation}
		\widehat{\bbP}_t \coloneqq \begin{bmatrix}
			\bbP_t & 0 \\ 0 & 0 
		\end{bmatrix}, \qquad \widehat{\Q}_t \coloneqq \begin{bmatrix}
			0 & 0 \\ 0 & P_t \oplus Q_t
		\end{bmatrix},
	\end{equation}
	so that $\widehat{\bbP}_t + \widehat{\Q}_t = \Pi_{\la_0,t}$ where $\ran(\Pi_{\la_0,t}) = \Upsilon_{\la_0,t} = \K_{\la_0,t} \oplus \cF_t$. Using the definition of $\Upsilon_{\la,t}$ and $\sM_{\la,t}$, see \eqref{cKcFcD} and \eqref{sM_defn}, we consider families of vectors $t\mapsto h_t, g_t$ such that
	\begin{equation}
		h_t \in \ran \bbP_t \subset \fH^4, \quad g_t \in \ran(P_t\oplus Q_t) \subset \fH^4,
	\end{equation}
	and 
	\begin{equation}
		(h_t,0)^\top = \widehat{\bbP}_t (\mathbf{q} + \sM_{\la_0,t} \mathbf{q}), \qquad (0,g_t)^\top  = \widehat{\Q}_t(\mathbf{q}+ \sM_{\la_0,t} \mathbf{q}),
	\end{equation}
	so that $h_{t_0} = g_{t_0} = q$. Since $t\mapsto \sM_{\la,t}, t\mapsto \bbP_t$ and $t\mapsto P_t\oplus Q_t$ are $C^1$, it follows that $t\mapsto h_t$ and $t\mapsto g_t$ are $C^1$. Using \cref{hypo_basic_assumptions} and that $\ker(\T_t) = \dom(N)$ (see \eqref{kerT_domN}), it follows that the restriction
	\begin{equation}\label{T_t_bijection}
		\T_t |_{K_t} : K_t \to \ran \bbP_t \subset \fH^4
	\end{equation}
	is an isomorphism. Hence, for each $t\in \sI$ there is a unique vector $\mathbf{w}_t\in K_t$ such that $\T_t\mathbf{w}_t = h_t$. Thus equations \eqref{w_t_g_t_families} hold for this choice of $\mathbf{w}_t$ and $g_t$.
	
	It remains to show that the mapping $t\mapsto \mathbf{w}_t$ is in $C^1(\sI,\cH_+^2)$. We will exploit the bijectivity of the mapping \eqref{T_t_bijection}, although some care is needed to handle the $t$-dependent domain of that mapping. Let $\mathsf{W}_t$ denote the $C^1$ family of boundedly invertible transformation operators in $\cH_+^2$ associated with the pair of projections $\scP_{K_t}$,  $\scP_{K_{t_0}}$ with $\ran \scP_{K_t} = K_t$ and $ \ran \scP_{K_{t_0}} = K_{t_0}$. Thus, $\mathsf{W}_t\scP_{K_{t_0}} = \scP_{K_t}\mathsf{W}_t$, and $\mathsf{W}_t : K_{t_0} \to K_{t}$ are bijections for $t$ near $t_0$, see \cite[Section I.4.6]{Kato}, \cite[Remark 2.4]{LSS18}, \cite[Section IV.1]{Dalecki_Krein}. Introducing $\mathbf{v}_t \in K_{t_0}$ by $\mathbf{v}_t = \mathsf{W}_t^{-1} \mathbf{w}_t$ so that $\T_t\mathbf{w}_t = h_t$ yields $(\T_t \circ \mathsf{W}_t) \mathbf{v}_t = h_t$. The map $\T_t \circ \mathsf{W}_t|_{K_{t_0}} : K_{t_0} \to \ran \bbP_t$ is a bijection and $t\mapsto \T_t \circ \mathsf{W}_t |_{K_{t_0}}$ is in $C^1(\sI, \cB(K_{t_0}, \fH^4))$ by the assumptions in the lemma. Since $\mathbf{w}_t = \mathsf{W}_t \circ(\T_t \circ \mathsf{W}_t)^{-1} h_t$, the function $t\mapsto \mathbf{w}_t$ is $C^1$ because each of the terms in the composition is $C^1$.
	
	The proof above can now be applied to prove the existence and regularity of the family $\la \mapsto \mathbf{w}_\la $ satisfying \eqref{w_la_family}. Indeed, setting $\tilde V^*_\la \coloneqq V^*_{t_0} - \la$, in this situation only the bounded perturbation $\tilde V^*_\la$ depends on $\la$, while the trace operator $\T_{t_0}$ and Lagrangian plane $\ran P_{t_0}\times \ran Q_{t_0}$ are $\la$-independent. The mapping $\la\mapsto \tilde V_\la^*$ is $C^\infty$, and the proof now follows from the arguments used to prove \eqref{w_t_g_t_families}. 
\end{proof}

We are ready to compute crossing forms. 
\begin{lemma}\label{lemma:cross_forms}
	Suppose $(\la_0,t_0)$ is a conjugate point satisfying \cref{hypo_basic_assumptions}, let $\mathbf{u}_{t_0}\in\ker(\cN_{t_0} + V_{t_0} - \la_0)$ and define $\mathbf{q} \coloneqq (\T_{t_0} \mathbf{u}_{t_0}, \T_{t_0} \mathbf{u}_{t_0}) \in \Upsilon_{\la_0,t_0} \cap \mathfrak{D}$. The $t$ crossing form \eqref{define_t_crossing_form} for the Lagrangian path $t\mapsto \Upsilon_{\la_0,t}$ with respect to the reference plane $\mathfrak{D}$ is given by
	\begin{align}\label{t_crossing_form}
		\mathfrak{m}_{t_0}(\mathbf{q},\mathbf{q}) &= \langle \tau \dot{V}_{t_0} \mathbf{u}_{t_0}, \mathbf{u}_{t_0}\rangle_{\cH^2} + \Om \left (( \dot P_{t_0} \oplus \dot Q_{t_0} )\T_{t_0}\mathbf{u}_{t_0}, \T_{t_0}\mathbf{u}_{t_0}\right ) + \Om \left (\T_{t_0}\mathbf{u}_{t_0}, \dot{\T}_{t_0}\mathbf{u}_{t_0}\right ).
	\end{align}
	The $\la$ crossing form in \eqref{define_la_crossing_form} for the path $\la\mapsto \Upsilon_{\la,t_0}$ with respect to $\mathfrak{D}$ is given by 
	\begin{align}\label{la_crossing_form}
		\mathfrak{m}_{\la_0}(\mathbf{q},\mathbf{q}) &= -\langle \tau \mathbf{u}_{t_0}, \mathbf{u}_{t_0}\rangle_{\cH^2}.
	\end{align}
\end{lemma}
\begin{proof}
	We have
	\begin{align}
		(N^*+V_t^* -\la_0 )\tau \mathbf{w}_t = 0,
	\end{align}
	and, differentiating this equation with respect to $t$ at $t_0$ and applying $\langle \cdot, \mathbf{w}_{t_0}\rangle_{\cH^2}$, we find that
	\begin{align}
		\langle (N^*+V^*_{t_0} -\la_0 )\tau \dot{\mathbf{w}}_{t_0},\mathbf{w}_{t_0} \rangle_{\cH^2} + \langle \dot{V}^*_{t_0}\tau \mathbf{w}_{t_0}, \mathbf{w}_{t_0} \rangle_{\cH^2} = 0.
	\end{align}
	Using Green's identity \eqref{modified_greens_real_la} with $\mathbf{u} = \dot{\mathbf{w}}_{t_0}$ and $\mathbf{v} = {\mathbf{w}}_{t_0}$, we find that
	\begin{equation}
		\left \langle (N^*+V^*_{t_0} - \la_0) \tau \dot{\mathbf{w}}_{t_0}, {\mathbf{w}}_{t_0}\right \rangle_{\cH^2} - \left \langle \dot{\mathbf{w}}_{t_0}, (N^* +V^*_{t_0} - \la_0) \tau{\mathbf{w}}_{t_0}\right \rangle_{\cH^2} = \Om (\T_{t_0}\dot{\mathbf{w}}_{t_0}, \T_{t_0} {\mathbf{w}}_{t_0}).
	\end{equation}
	Combining the previous two equations and noting that $\mathbf{w}_{t_0} \in \ker(N^*+ V_{t_0}^* - \la_0)\tau$, we obtain
	\begin{equation}\label{V_and_omega}
		\langle \dot{V}^*_{t_0}\tau \mathbf{w}_{t_0}, \mathbf{w}_{t_0} \rangle_{\cH^2} - \Om (\T_{t_0} {\mathbf{w}}_{t_0},\T_{t_0}\dot{\mathbf{w}}_{t_0}) =0.
	\end{equation} 
	Now by \eqref{define_t_crossing_form} and \eqref{412c} we have
	\begin{align}
		\mathfrak{m}_{t_0}(\mathbf{q},\mathbf{q}) &= \de{}{t}\Big|_{t=t_0}  \, \tilde \Om ((\T_{t_0} \mathbf{w}_{t_0}, g_{t_0})^{\top},(\T_{t} \mathbf{w}_t, g_t)^{\top}), \nonumber\\
		&=  \Om \left (\T \mathbf{w}_{t_0}, \T \dot{\mathbf{w}}_{t_0}\right ) + \Om \left (\T_{t_0} \mathbf{w}_{t_0}, \dot{\T}_{t_0} \mathbf{w}_{t_0}\right ) - \Om \left ( g_{t_0}, \dot g_{t_0}\right ). \label{cross_form_t}
	\end{align}
	Since $g_t \in \cF_t$ and hence $ g_t = (P_t\oplus Q_t ) g_t $, we have 
	\begin{equation}
		\dot g_{t_0} =( \dot P_{t_0} \oplus \dot Q_{t_0} )g_{t_0} + ( P_{t_0} \oplus Q_{t_0})\dot g_{t_0}.
	\end{equation}
	Using this, along with the fact that $\cF_t = \ran(P_t) \times \ran(Q_t)$ is a Lagrangian subspace of $\fH^4\oplus \fH^4$, and $\T_{t_0} \mathbf{u}_{t_0} = \T_{t_0}\mathbf{w}_{t_0} \in \ran(P_{t_0}) \times \ran(Q_{t_0})$, we find that 
	\begin{align}
		\Om \left (\T_{t_0}\mathbf{w}_{t_0}, \dot g_{t_0}\right ) &=  \Om \left ( \T_{t_0}\mathbf{w}_{t_0}, ( \dot P_{t_0} \oplus \dot Q_{t_0} )\T_{t_0}\mathbf{w}_{t_0}+ ( P_{t_0} \oplus Q_{t_0})\dot g_{t_0}\right ), \nonumber\\
		&= \Om \left (\T_{t_0}\mathbf{w}_{t_0}, ( \dot P_{t_0} \oplus \dot Q_{t_0} )\T_{t_0}\mathbf{w}_{t_0}\right ). \label{omega_g_Tw}
	\end{align}
	Using \eqref{omega_g_Tw} and \eqref{V_and_omega} in \eqref{cross_form_t}, and the fact that $\dot{V}^*_{t_0} \tau = \tau \dot{V}_{t_0}$, yields \eqref{t_crossing_form}.
	
	For the $\la$-crossing form calculation, we now consider the family of vectors $\la \mapsto \mathbf{w}_\la$ given by \eqref{w_la_family}. By \eqref{define_la_crossing_form}\footnote{\, Here we note that the second component of the mapping $\lambda\mapsto (\T_{t_0} \mathbf{w}_\la , \T_{t_0} u_{t_0})^{\top}$ is $\lambda$-independent, hence, the computation reduces to the first component $\Omega$ of the symplectic form $\tilde\Omega=\Omega\oplus(-\Omega)$.} we have
	\begin{align}
		\mathfrak{m}_{\la_0}(\mathbf{q},\mathbf{q}) &= \de{}{\la}\Big|_{\la=\la_0}  \, \Om (\T_{t_0} \mathbf{w}_{\la_0}, \T_{t_0} \mathbf{w}_\la) =  \Om \left (\T_{t_0} \mathbf{w}_{\la_0}, \T_{t_0} \dot{\mathbf{w}}_{\la_0}\right ).
	\end{align}
	As prior, for the equation 
	\begin{equation}\label{diff_eqn_la_w_family}
		(N^*+V^*_{t_0} -\la )\tau \mathbf{w}_\la = 0,
	\end{equation}
	differentiating at $\la_0$ and applying $\langle \cdot, \mathbf{w}_{\la_0}\rangle_{\cH^2}$ yields
	\begin{align}
		\langle (N^*+V^*_{t_0} -\la_0 )\tau \dot{\mathbf{w}}_{\la_0},\mathbf{w}_{\la_0} \rangle_{\cH^2} - \langle \tau \mathbf{w}_{\la_0}, \mathbf{w}_{\la_0} \rangle_{\cH^2} = 0.
	\end{align}
	Using Green's identity \eqref{modified_greens_real_la} with $\mathbf{u} = \dot{\mathbf{w}}_{\la_0}$ and $\mathbf{v} = {\mathbf{w}}_{\la_0}$, we find that
	\begin{equation}
		\left \langle (N^*+V^*_{t_0} - \la_0) \tau \dot{\mathbf{w}}_{\la_0}, {\mathbf{w}}_{\la_0}\right \rangle_{\cH^2} - \left \langle \dot{\mathbf{w}}_{\la_0}, (N^* +V^*_{t_0} - \la_0) \tau{\mathbf{w}}_{\la_0}\right \rangle_{\cH^2} = \Om (\T_{t_0}\dot{\mathbf{w}}_{\la_0}, \T_{t_0} {\mathbf{w}}_{\la_0}).
	\end{equation}
	Combining the previous two equations and noting that $\mathbf{w}_{\la_0} \in \ker(N^*+ V_{t_0}^* - \la_0)\tau$, we obtain 
	\begin{equation}
		\Om (\T_{t_0} {\mathbf{w}}_{\la_0}, \T_{t_0}\dot{\mathbf{w}}_{\la_0}) = - \langle \tau \mathbf{w}_{t_0}, \mathbf{w}_{t_0} \rangle_{\cH^2}.
	\end{equation} 
	Substituting the previous equation into the $\la$-crossing form, we obtain \eqref{la_crossing_form}.
\end{proof}

	\subsection{Lyapunov-Schmidt reduction}\label{subsec_Lyapunov_Schmidt}
	
	In this subsection we prove the following proposition, which states that eigenvalues $\la\in\spec(\cN_{t}+V_{t})$, for $(\la,t)$ near $(\la_0,t_0)$, are determined by the zero set of the determinant of an $n\times n$ matrix $M(\la,t)$. This matrix is derived via Lyapunov-Schmidt reduction. We then prove \cref{thm:Hadamard}.
	
	\begin{prop}\label{prop:M}
		Assume \cref{hypo_basic_assumptions}. Suppose that $(\la_0,t_0)$ is a conjugate point, where $\dim \ker(\cN_{t_0} + V_{t_0} - \la_0)=n$ with basis $\{\bu_{t_0}^{(1)}, \dots, \bu_{t_0}^{(n)} \}$. There exists an $n\times n$ matrix $M(\lambda,t)$, defined near $(\lambda_0,t_0)$, such that $\lambda \in \spec(\cN_t+V_t)$ if and only if $\det M(\lambda,t) = 0$. This matrix satisfies $M(\lambda_0,t_0) = 0$ and
		\begin{align}
			\frac{\p M_{ji}}{\p t}(\lambda_0,t_0) &= \left<  \tau\dot{V}_{t_0} \bu_{t_0}^{(i)}, \bu_{t_0}^{(j)}\right> + \Om\left ( (\dot{P}_{t_0} \oplus \dot{Q}_{t_0} ) \T_{t_0} \bu^{(i)}_{t_0}, \T_{t_0} \bu^{(j)}_{t_0} \right ) + \Om\left ( \T_{t_0} \bu^{(i)}_{t_0}, \dot{\T}_{t_0} \bu^{(j)}_{t_0} \right ), \\
			\frac{\p M_{ji}}{\p \lambda}(\lambda_0,t_0) &= - \left< \tau\bu_{t_0}^{(i)}, \bu_{t_0}^{(j)}\right>.
		\end{align}
		\end{prop}
	\begin{proof}
		{
		Let us denote $\cP_t \coloneqq P_t \oplus Q_t$, so that $\cF_t = \ran \cP_t$, $\cF_t^\perp = \ran (I - \cP_t)$, and $\cD_t \coloneqq \dom \cN_t = \{\mathbf{u} \in \cH_+^2 : \T_t \mathbf{u} \in \cF_t\}$. In what follows, we build a $C^2$ family of isomorphisms $U_t : \cD_{t_0} \lra \cD_{t}$, which allow us to write the eigenvalue problem for $\cN_t + V_t -\la$ on $\cD_t$ as an equivalent eigenvalue problem on $\cD_{t_0}$. We then proceed with Lyapunov-Schmidt reduction for the problem on the fixed domain $\cD_{t_0}$. 
		
		The construction of $U_t$ requires some preparation. For $t$ near $t_0$, define $\Delta_t \coloneqq \cP_t - \cP_{t_0}$ and
		\begin{equation}
			\Theta_t \coloneqq (I - \Delta_t^2) \left ( (I-\cP_t) (I-\cP_{t_0}) + \cP_t \cP_{t_0} \right ). 
		\end{equation}
		Then, as in \cite[Section I.4.6]{Kato}, $t\mapsto \Theta_t$ is a $C^2$ family of boundedly invertible transformation operators on $\fH^4$ associated with the pair $\cP_t, \cP_{t_0}$. In particular, we have $\Theta_{t_0}=I$, $\Theta_t \cP_{t_0} = \cP_t\Theta_{t}$, and that $\Theta_t$ maps $\cF_{t_0}\subset \fH^4$ onto  $\cF_t\subset \fH^4$ for $t$ near $t_0$. Next, we define the bounded operator
		\begin{equation}
			B_t : \cH_+^2 \to \cF_{t_0}^\perp, \qquad B_t \coloneqq (I- \cP_{t_0}) \Theta_{t}^{-1} \T_t,
		\end{equation}
		with fixed codomain and whose vanishing imposes the boundary condition. Indeed, 
		\begin{equation}
			\mathbf{u} \in \cD_t \iff \T_t \mathbf{u} \in \cF_t \iff \Theta_t^{-1} \T_t \mathbf{u} \in \cF_{t_0} \iff B_t \mathbf{u} =0, 
		\end{equation}
		hence $\cD_t = \ker B_t$. In addition, notice that $B_{t_0}$ is surjective on account of $\T_{t_0}$ being surjective. Choosing a bounded right inverse $G_0 : \cF_{t_0}^\perp \to \cH_+^2$ such that $B_{t_0} G_0 = I_{\cF_{t_0}^\perp}$, it follows that the bounded operator $B_t G_0 : \cF_{t_0}^\perp \to \cF_{t_0}^\perp$ is invertible for nearby $t$. Then defining
		\begin{equation}
			G_t: \cF_{t_0}^\perp \to \cH_+^2, \quad G_t \coloneqq G_0 (B_t G_0)^{-1},
		\end{equation} 
		we have $B_t G_t = I_{ \cF_{t_0}^\perp}$. Moreover, $t\mapsto G_t$ is $C^2$ since $t\mapsto B_t$ is $C^2$. With the above operators we define
		\begin{equation}\label{435}
			U_t : \cD_{t_0} \to \cD_t, \qquad U_t \coloneqq I - G_t B_t.
		\end{equation}
		 Note that for any $\bz \in \cD_{t_0}$ we have
		 \begin{equation*}
		 	B_t U_t \bz = B_t \bz - B_t G_t B_t \bz = B_t \bz - B_t \bz = 0,
		 \end{equation*}
		 so indeed $U_t \bz \in \ker B_t = \cD_t$. Since $B_{t_0} \bz = 0$ for $\bz\in\cD_{t_0}$, we also note that $U_{t_0} \bz = \bz$, hence $U_{t_0}= I_{\cD_{t_0}}$. Moreover, \(U_t:\mathcal D_{t_0}\to \mathcal D_t\) is an isomorphism for \(t\) near \(t_0\). Indeed, since \(B_tG_0\) is invertible, every \(\mathbf{u}\in\mathcal H_+^2\) admits the unique decomposition
		 \[
		 \mathbf{u}=\bigl(\mathbf{u}-G_0(B_tG_0)^{-1}B_t \mathbf{u}\bigr)+G_0(B_tG_0)^{-1}B_t \mathbf{u}
		 \in \mathcal D_t \oplus \operatorname{ran}G_0.
		 \]
		 Thus \(\mathcal H_+^2=\mathcal D_t\oplus\operatorname{ran}G_0\) for $t$ near $t_0$. Since \(\operatorname{ran}G_t=\operatorname{ran}G_0\) and \(B_tG_t=I\), the operator $U_t=I-G_tB_t$ is the projection in $\cH^2_+$ onto \(\mathcal D_t\) along \(\operatorname{ran}G_0\). Its restriction to \(\mathcal D_{t_0}\) -- which is complementary to $\ran G_0$ -- is therefore bijective.
		 
		 Importantly, the map $U_t$ allows us to pull the eigenvalue problem for $\cN_t +V_t - \la$ on $\cD_t$ back to the fixed domain $\cD_{t_0}$. Namely, $\mathbf{u} \in \ker(\cN_t + V_t - \la)$ if and only if there exists $\mathbf{z}\in\cD_{t_0}$ such that 
		 \[
		 \mathbf{u} = U_t \mathbf{z} \quad \text{and} \quad  (\cN_t + V_t - \la) U_t \mathbf{z}=0.
		 \]
		 By the $C^2$ regularity of the mappings $t \mapsto U_t, t \mapsto V_t, t \mapsto \T_t$, the family $(\la,t) \mapsto (\cN_t + V_t -\la)U_t$ is also $C^2$ as a family of bounded operators from $\cD_{t_0}$ into $\cH^2$. 
		 
		 We are now in a position to use Lyapunov-Schmidt reduction on the operator $(\cN_t + V_t -\la)U_t$ acting between the fixed spaces $\cD_{t_0}$ and $\cH^2$. The first step is to split the eigenvalue equation -- \emph{pulled back to the fixed domain at} $t_0$ -- into two parts, one of which can always be solved uniquely. Let $\Pi$ denote the finite-rank $\cH^2$-orthogonal projection onto the adjoint kernel $\ker (\cN^*_{t_0} + V^*_{t_0} - \lambda_0)$. It follows that $\lambda$ is an eigenvalue of $\cN_t+V_t$ if and only if there exists a nonzero $\mathbf{z}\in \cD_{t_0}$ such that both
		\begin{equation}
			\label{Gsplit1}
			\Pi(\cN_t + V_t - \lambda)U_t \mathbf{z}  = 0
		\end{equation}
		and
		\begin{equation}
			\label{Gsplit2}
			(I-\Pi)(\cN_t  + V_t - \lambda)U_t \mathbf{z} = 0
		\end{equation}
		hold.
	
		We first consider the range equation \eqref{Gsplit2}. Denoting $X = \ker (\cN_{t_0} + V_{t_0} - \lambda_0) ^\perp\cap \cD_{t_0}$, we have that any $\mathbf{z} \in \cD_{t_0}$ can be written uniquely as 
		\begin{equation}\label{decomp1}
			\bz = \sum_{i=1}^n \al_i \bu_{t_0}^{(i)} +  \tilde{\bz},
		\end{equation}
		where $\al_i \in \C$ and $\tilde{\bz} \in X$. This means \eqref{Gsplit2} holds if and only if there exists a vector $\pmb{\alpha} = (\al_1, \dots ,\al_n) \in \C^n$ and a function $\tilde{\bz} \in X$ such that
		\begin{equation}
			\label{Gsplit3}
			(I-\Pi)(\cN_t + V_t - \lambda)U_t \left (\sum_{i=1}^n \al_i \bu_{t_0}^{(i)} + \tilde{\bz}\right ) = 0.
		\end{equation}
		We claim that for each $(\pmb{\al},\lambda,t)$ there exists a unique $\tilde{\bz} = \tilde{\bz}(\pmb{\al},\lambda,t) \in X$ satisfying \eqref{Gsplit3}. Expanding and rearranging yields
		\[
		(I-\Pi)(\cN_t + V_t - \lambda)U_t\tilde{\bz}(\pmb{\al},\lambda,t) = - (I-\Pi)(\cN_t + V_t - \lambda)U_t\sum_{i=1}^n \al_i \bu_{t_0}^{(i)}.
		\]
		Defining
		\begin{equation}\label{eq:Tdefn}
			T(\lambda,t) \colon X \to \ran (\cN_{t_0} + V_{t_0}- \lambda_0) , \qquad T(\lambda,t) = (I-\Pi)\big(\cN_t + V_{t} - \lambda\big)U_t\Big|_{X},
		\end{equation}
		we observe that $T(\lambda_0,t_0)$ is invertible, hence, due to continuity of resolvents as mappings from $\cH^2$ to $\cH^2_+$, see \cref{thm:resolvent_regularity} (2), $T(\lambda,t)$ is also invertible for nearby $(\lambda,t)$. Now defining
		\begin{equation}\label{eq:matrixW_lam_t}
			W(\lambda,t) : E  \to X, \qquad W(\lambda,t)  = -T^{-1}(\lambda,t) (I-\Pi)\big(\cN_t + V_t - \lambda\big) U_t\Big|_{E},
		\end{equation}
		where $E = \ker (\cN_{t_0} + V_{t_0} - \lambda_0)$, the unique solution to \eqref{Gsplit3} is thus
		\begin{equation}
			\label{zhat}
			\tilde{\bz}(\pmb{\al},\lambda,t) = W(\lambda,t) \sum_{i=1}^n \al_i \bu_{t_0}^{(i)}.
		\end{equation}
		Therefore, the range equation \eqref{Gsplit2} is satisfied if and only if $\bz$ is given by \eqref{decomp1}, \eqref{zhat}, i.e.
		\begin{equation}
			\label{uform}
			\bz 
			 = \big(I + W(\lambda,t)\big)\sum_{i=1}^n \al_i \bu_{t_0}^{(i)}
		\end{equation}
		for some $\pmb{\al}\in \C^n$. Substituting into the reduced equation \eqref{Gsplit1}, we conclude that there exists $\bz\in\cD_{t_0}$ for which $(\cN_t + V_t - \lambda)U_t\bz  = 0$ holds if and only if
		\begin{equation}
			\label{PG}
			\Pi(\cN_t + V_t - \lambda) U_t\big(I + W(\lambda,t)\big)\left( \sum_{i=1}^n \al_i \bu_{t_0}^{(i)}\right) = 0
		\end{equation}
		for some $\pmb{\al} \in \C^n$. Since $\cD_{t_0} = X \oplus E$, notice that $\bz$ is nonzero if and only if $\pmb{\al}$ is nonzero. Finally, we observe that $\ker (\cN^*_{t_0} + V^*_{t_0} - \lambda_0)$ is spanned by $\{\tau\bu_{t_0}^{(1)}, \tau\bu_{t_0}^{(2)}, \dots, \tau\bu_{t_0}^{(n)}\}$, so \eqref{PG} is equivalent to
		\begin{equation}
			\label{systemeqs}
			\left<(\cN_t + V_t - \lambda)U_t\big(I + W(\lambda,t)\big) \left(\sum_{i=1}^n \al_i \bu_{t_0}^{(i)}\right), \tau\bu_{t_0}^{(j)}\right>_{\cH^2} = 0, \qquad j=1,\dots, n. 
		\end{equation}
		Defining the $n\times n$ matrix $M(\la,t)$ by 
		\begin{equation}\label{eq:matrixM}
			M_{ji}(\la,t) = \left<(\cN_t + V_t - \lambda) U_t\big(I + W(\lambda,t)\big)\bu^{(i)}_{t_0}, \tau\bu_{t_0}^{(j)}\right>_{\cH^2}, \quad i,j=1, \dots, n,
		\end{equation}
		the system of $n$ equations \eqref{systemeqs} may be written as $M(\la,t) \pmb{\al} = 0$, which is satisfied for nonzero $\pmb{\al}$ if and only if $\det M(\la,t)=0$. This completes the first part of the proof. 
		
		It follows that $M(\lambda_0,t_0) = 0$, because $W(\lambda_0,t_0)=0$. For the $t$ derivative, we first use the modified Green's identity \eqref{modified_greensOm} and the fact that $\tau(\cN_t + V_t - \lambda)$ is a symmetric operator to write 
		\begin{align}
			M_{ji}&(\la_0,t) = \left<\tau(\cN_t + V_t - \lambda_0) U_t\big(I + W(\lambda_0,t)\big)\bu^{(i)}_{t_0}, \bu_{t_0}^{(j)}\right>_{\cH^2}, \nonumber\\
			&= \left<(\tau(N + V_{t_0} - \lambda_0))^* U_t\big(I + W(\lambda_0,t)\big)\bu^{(i)}_{t_0}, \bu_{t_0}^{(j)}\right>_{\cH^2} + \left<(\tau V_t - \tau V_{t_0})^* U_t\big(I + W(\lambda_0,t)\big)\bu^{(i)}_{t_0}, \bu_{t_0}^{(j)}\right>_{\cH^2} , \nonumber\\
			&=  \left< U_t\big(I + W(\lambda_0,t)\big)\bu^{(i)}_{t_0}, (\tau(N + V_{t_0} - \lambda_0))^* \bu_{t_0}^{(j)}\right>_{\cH^2} + \Om\left  (\T_{t}\Big( U_t \big( I + W(\lambda_0,t) \big)\Big)\bu^{(i)}_{t_0}, \T_{t} \bu^{(j)}_{t_0} \right  ) \nonumber\\
			& \qquad \qquad \qquad + \left<(\tau V_t - \tau V_{t_0})^* U_t\big(I + W(\lambda_0,t)\big)\bu^{(i)}_{t_0}, \bu_{t_0}^{(j)}\right>_{\cH^2},\nonumber \\
			&=   \left<(\tau V_t - \tau V_{t_0})^* U_t\big(I + W(\lambda_0,t)\big)\bu^{(i)}_{t_0}, \bu_{t_0}^{(j)}\right>_{\cH^2} +  \Om\left  (\T_{t}\Big( U_t \big( I + W(\lambda_0,t) \big)\Big)\bu^{(i)}_{t_0}, \T_{t} \bu^{(j)}_{t_0} \right  ), \label{M_la_t}
		\end{align}
		because $(\tau(N + V_{t_0} - \lambda_0))^* \bu_{t_0}^{(j)} = \tau(\cN_{t_0} + V_{t_0} - \lambda_0) \bu_{t_0}^{(j)} = 0$.  Now defining 
		\[
		g_t \coloneqq \T_{t}\Big( U_t \big( I + W(\lambda_0,t) \big)\Big)\bu^{(i)}_{t_0},
		\]
		since $U_t\big(I + W(\lambda_0,t)\big)\bu^{(i)}_{t_0} \in \cD_t$ for all $t$ near $t_0$, we have $g_t \in \ran \cP_t$. Hence $g_t = \cP_t g_t$, and 
		\[
		\dot{g}_{t_0} = \dot{\cP}_{t_0}  g_{t_0} + {\cP}_{t_0} \dot{g}_{t_0}.
		\] 
		Moreover, since $W(\la_0,t_0)=0$ and $U_{t_0}=I$, we have $g_{t_0} = \T_{t_0} \bu^{(i)}_{t_0}$. Now differentiating \eqref{M_la_t} with respect to $t$ at $t_0$, we find that (where dot denotes $d/dt$),
		\begin{align*}
			\frac{\p M_{ji}}{\p t}(\lambda_0,t_0) &= \left<(\tau \dot{V}_t)^* U_t \big(I + W(\lambda_0,t)\big)\bu^{(i)}_{t_0}, \bu_{t_0}^{(j)}\right>_{\cH^2} 
			+  \left \langle  (\tau V_t - \tau V_{t_0})^* \dot{U}_t \big(I + W(\lambda_0,t_0)\big)\bu^{(i)}_{t_0}, \bu_{t_0}^{(j)}    \right \rangle
			 \\
			 & \qquad 
			+ \left<(\tau V_t - \tau V_{t_0})^*U_t \big(I + \p_tW(\lambda_0,t_0)\big)\bu^{(i)}_{t_0}, \bu_{t_0}^{(j)}\right>_{\cH^2} \Big|_{t=t_0} \\
			& \qquad \qquad + \Om\left (\dot{g}_{t_0}, \T_{t_0} \bu^{(j)}_{t_0} \right ) + \Om\left (\T_{t_0} \bu^{(i)}_{t_0}, \dot{\T}_{t_0} \bu^{(j)}_{t_0} \right ), \\
			&= \left<\tau \dot{V}_{t_0} \bu^{(i)}_{t_0}, \bu_{t_0}^{(j)}\right>_{\cH^2} + \Om\left ( (\dot{P}_{t_0} \oplus \dot{Q}_{t_0} ) \T_{t_0} \bu^{(i)}_{t_0}, \T_{t_0} \bu^{(j)}_{t_0} \right ) + \Om\left (\T_{t_0} \bu^{(i)}_{t_0}, \dot{\T}_{t_0} \bu^{(j)}_{t_0} \right ).
		\end{align*}
		For the $\la$ derivative, we first observe that since $W(\lambda,t_0) \bu^{(i)}_{t_0} \in X \subset \cD_{t_0} = \dom \cA_{t_0} \times \dom \cB_{t_0}$ for all $\la$ near $\la_0$, we have $\p_\la W(\lambda,t_0) \bu^{(i)}_{t_0} \in \cD_{t_0}$ and hence
		\begin{equation*}
			\T_{t_0} \p_\la W(\lambda_0,t_0) \bu^{(i)}_{t_0}\in \ran \cP_{t_0}.
		\end{equation*}
		Therefore, we may differentiate \eqref{eq:matrixM} directly because $\p_\la W(\lambda,t_0) \bu^{(i)}_{t_0} \in \cD_{t_0}$, arriving at
		\begin{align*}
			\frac{\p M_{ji}}{\p \lambda}(\lambda_0,t_0) &= \left<-\big(I + W(\lambda_0,t_0)\big)\bu^{(i)}_{t_0} + (\cN_{t_0} + V_{t_0} - \lambda_0) \p_\la W(\lambda_0,t_0) \bu^{(i)}_{t_0}, \tau \bu^{(j)}_{t_0}\right>_{\cH^2} = - \left<\bu^{(i)}_{t_0}, \tau \bu^{(j)}_{t_0}\right>_{\cH^2},
		\end{align*} 
		where we used that $\tau \bu^{(j)}_{t_0} \in \ker(\cN_{t_0}^* + V_{t_0}^* - \lambda_0)$, as required. }
	\end{proof}

	\begin{proof}[Proof of \cref{thm:Hadamard}]
		 If $\dim\ker(\cN_{t_0} + V_{t_0} - \la_0)=1$, then $M(\la,t)$ defined in \eqref{eq:matrixM} is a scalar. Comparing the expressions in \cref{lemma:cross_forms} and \cref{prop:M}, we have
		\begin{align}\label{M_derivs_simple_case}
			\frac{\p M}{\p t}(\lambda_0,t_0) = \mathfrak{m}_{t_0}(\mathbf{q},\mathbf{q}), \qquad \frac{\p M}{\p \la}(\lambda_0,t_0) = \mathfrak{m}_{\la_0}(\mathbf{q},\mathbf{q}),
		\end{align}
		where $\mathbf{q} = \T_{t_0}\mathbf{u}_{t_0}$. \Cref{thm:Hadamard} now follows from \cref{prop:M} and the implicit function theorem. Namely, in the case that $\p_\la M(\lambda_0,t_0) =\mathfrak{m}_{\la_0}(\mathbf{q},\mathbf{q})\neq 0$, applying the implicit function theorem to $M(\la,t)=0$ implies the existence of a $C^1$ curve $\la(t)$, defined for $t$ near $t_0$, whose first derivative $\la'(t_0)$ is given by \eqref{hadamard1}. Formula \eqref{hadamard2} in the case when $\p_t M(\lambda_0,t_0) =\mathfrak{m}_{t_0}(\mathbf{q},\mathbf{q})\neq 0$ follows similarly. 
	\end{proof}

	\section{Application: Linearised NLS on compact star graphs}\label{sec:applications}
	
	In this section we apply our abstract theory to the example described in the introduction, namely the linearisation about a standing wave solution \eqref{standingwave} of the NLS equation \eqref{NLS} on a compact star graph $\mathcal{G}$, whose profile satisfies the standing wave equation \eqref{SWE} and vertex conditions \eqref{vertex_conds_1}. After restricting the eigenvalue problem \eqref{Nop1}--\eqref{Nop2} to the sub-graph $\cG_t$, $t\in(0,1]$, and rescaling back to $\cG$, we obtain the $t$-dependent eigenvalue problem \eqref{Nt_EVP_NLS}--\eqref{Nt_stargraphs}.

	Our primary goal will be to prove \cref{thm:spectral_index_thm}. This follows from a homotopy argument and computations of the crossing forms \eqref{define_t_crossing_form} and \eqref{define_la_crossing_form} in $t$ and $\la$, which are used to compute local contributions to the Maslov index (defined in  \eqref{Maslov_index_defn}). As an aside, we also compute Hadamard-type formulas for the first derivatives of the eigenvalue curves. We emphasise that such formulas at arbitrary $(\la,t)$, given in \eqref{ladash}, play no role in the proof of \cref{thm:spectral_index_thm}; however, for the purposes of the present note, we write down these expressions to give examples of the abstract formulas given in \cref{thm:Hadamard,thm:Hadamard_via_resolvents}.
	
	\begin{rem}\label{rem:existence}
		We will not discuss here the issue of the existence of a solution to \eqref{SWE} --  \eqref{vertex_conds_1}. Since the focus will be on the spectrum of the associated linearised operator, instead we assume a solution exists, and use the crossing forms derived in \cref{sec:Hadamard_crossingforms} to study the existence of positive real eigenvalues. 
	\end{rem}
	
	Let us first describe the eigenvalue problem \eqref{Nt_EVP_NLS} in the notation of \cref{sec:2_setup}. The function spaces are 
	\begin{equation}
		\cH = L^2(\cG), \quad \cH_+= \hatt{H}^2(\cG), \quad \fH = L^2(\p \cG) \cong \C^{2m}.
	\end{equation}
	The Sobolev space of functions vanishing on the boundary $\partial\cG$ together with their derivatives is denoted by
	\begin{equation}\no
		\hatt H^2_0(\cG) \coloneqq \left\{f\in \hatt{H}^2(\cG): \tr f=0\right\}.
	\end{equation}
	The minimal symmetric operator is $A = -\p_{xx}$, with $\,\dom(A) = \hatt H_0^2(\cG)$ and  maximal adjoint operator $A^* = -\p_{xx},\, \dom(A^*) = \cH_+ = \hatt H^2(\cG)$. 

	The trace operators 
	\begin{align*}
		\tr_t = (\Gamma_{0,t}, \Gamma_{1,t})^\top : \cH_+ \to \fH, \qquad 
		\T_t = \tr_t\oplus \tr_t = [\Gamma_{0,t}, \Gamma_{1,t}]^\top \oplus [\Gamma_{0,t}, \Gamma_{1,t}]^\top: \cH_+^2 \to \fH^4,
	\end{align*} 
	from \cref{hypo:QPGH_families} are given by \eqref{Tt_trt_NLS}, \eqref{Gamma_01_NLS}. Recalling that $\cL\subset \fH^2$ is the Lagrangian plane describing the vertex conditions \eqref{vertex_conds_1}, we denote by $P_\cL$ the ($t$-independent) orthogonal projection in $\C^{2m}\times \C^{2m}$ onto $\cL$. In the notation of \cref{hypo:QPGH_families}, $P_t = Q_t = P_\cL$. Thus, $\mathbf{u}=(u,v)^\top \in \cH_+^2$ lies in $\dom \cN_t$ if it satisfies the vertex conditions \eqref{vertex_conds_1}, so that $\T_t \mathbf{u} \in \cL\oplus \cL$, i.e.
	\begin{equation}\label{vertex_conds_sec_5}
		\begin{cases}
			u_{1}(0) = u_{2}(0) = \dots = u_{m}(0), \\
			u_{1}(\ell_1) = u_{2}(\ell_2) = \dots = u_{m}(\ell_m)=0, \\
			\f{1}{t} \sum_{i=1}^{m} u_{i}'(0) = \al u_{1}(0), \quad \al\in\R,
		\end{cases}
	\end{equation} 
	with similar conditions holding for $v$.

	Having established the relevant notation, we proceed with the computation of crossing forms. In what follows, we denote the spectral parameter in the eigenvalue problem \eqref{Nt_EVP_NLS} by $\mu \coloneqq t^2\la$. Let $(\mu_0,t_0)=(t_0^2\la_0,t_0)$ be a conjugate point with eigenfunction $\mathbf{u}_{t_0}$. For the Lagrangian path $t\mapsto \Upsilon_{\mu_0,t}$, since $\dot{P}_t=\dot{Q}_t=0$, the crossing form \eqref{t_crossing_form} is given by
	\begin{align}\label{t_cross_form_NLS}
		\mathfrak{m}_{t_0}(\mathbf{q},\mathbf{q}) &= \big\langle \tau \dot{V}_{t_0} \mathbf{u}_{t_0}, \mathbf{u}_{t_0}\big\rangle_{\cH^2}  + \Om \left (\T_{t_0}\mathbf{u}_{t_0}, \dot{\T}_{t_0}\mathbf{u}_{t_0}\right ).
	\end{align}
	For the first term in \eqref{t_cross_form_NLS}, writing $\mathbf{u}_{t_0} = (u_{t_0}, v_{t_0}) \in\cH^2$ and recalling from \eqref{Nop2}, \eqref{Ntop2} that $V_{\pm,t}(x) = t^2V_\pm(tx)$, we have, where dot denotes $d/dt$ and prime denotes  differentiation with respect to the argument, 
	\begin{align*}
		 \big\langle \tau \dot{V}_{t_0}& \mathbf{u}_{t_0}, \mathbf{u}_{t_0}\big\rangle_{\cH^2} 
		 = -\big \langle \left (2 t_0 V_+ (t_0x) + t_0^2 V_+ '(t_0x)x \right )u_{t_0}, u_{t_0}\big \rangle_\cH 
		 + \big \langle \left (2 t_0 V_-(t_0x) + t_0^2 V_- '(t_0x) x \right ) v_{t_0}, v_{t_0}\big \rangle_\cH, \\
		 &= - \sum_{i=1}^{m} \int_{0}^{\ell_i}\left (2 t_0 V_+ (t_0x) + t_0^2 V_+'(t_0x) x \right ) u_{t_0,i}(x)^2 dx 
		 + \sum_{i=1}^{m} \int_{0}^{\ell_i} \left (2 t_0 V_-(t_0x) + t_0^2 V_-'(t_0x) x \right ) v_{t_0,i}(x)^2 dx.
	\end{align*} 
	A direct calculation using the equation
	\begin{align*}
		v_{t_0}''(x) + t_0^2 V_-(t_0 x) v_{t_0}=0
	\end{align*}
	shows that, for each $i=1, \dots, m$, we have
	\begin{align*}
		\de{}{x}\left[\f{1}{t_0^2}x (v'_{t_0,i}(x))^2 + x (v_{t_0,i}(x))^2 V_- (t_0x) - \f{1}{t_0^2}v_{t_0,i}(x)v'_{t_0,i}(x)\right] 
		&= \left[2 V_-(t_0x) + t_0 x V_- '(t_0x) \right ] v_{t_0,i}(x)^2.
	\end{align*}
	Hence
	\begin{align*}
		 \int_{0}^{\ell_i} (2 t_0 V_- (t_0x) + t_0^2 V_- '(t_0x) x ) v_{t_0,i}(x)^2 dx &= \f{\ell_i}{t_0^2} (v'_{t_0,i}(\ell_i))^2 + \f{1}{t_0^2} v_{t_0,i}(0)v_{t_0,i}'(0).
	\end{align*}
	It can be similarly deduced using $u_{t_0}''(x) + t_0^2 V_+(t_0 x) u_{t_0}=0$ that
	\begin{align*}
		\int_{0}^{\ell_i} (2 t_0 V_+ (t_0x) + t_0^2 V_+ '(t_0x) x ) u_{t_0,i}(x)^2 dx
		 = \f{\ell_i}{t_0^2} (u'_{t_0,i}(\ell_i))^2 + \f{1}{t_0^2} u_{t_0,i}(0)u_{t_0,i}'(0),
	\end{align*}
	and therefore
	\begin{align*}
		\big\langle \tau \dot{V}_{t_0} \mathbf{u}_{t_0}, \mathbf{u}_{t_0}\big\rangle_{\cH^2} &= - \sum_{i=1}^{m}\left( \f{\ell_i}{t_0^2} (u'_{t_0,i}(\ell_i))^2 + \f{1}{t_0^2} u_{t_0,i}(0)u_{t_0,i}'(0),\right) + \sum_{i=1}^{m}\left ( \f{\ell_i}{t_0^2} (v'_{t_0,i}(\ell_i))^2 + \f{1}{t_0^2} v_{t_0,i}(0)v_{t_0,i}'(0) \right ).
	\end{align*}
	Using the $\delta$-type condition at the central vertex, we conclude that 
	\begin{align}
		 \big\langle \tau \dot{V}_{t_0} \mathbf{u}_{t_0}, \mathbf{u}_{t_0}\big\rangle_{\cH^2} &= \f{1}{t_0^2}\sum_{i=1}^m \left \{ \ell_i v'_{t_0,i}(\ell_i)^2 -  \ell_i u'_{t_0,i}(\ell_i)^2 - u_{t_0,i}(0)u_{t_0,i}'(0) + v_{t_0,i}(0)v_{t_0,i}'(0)\right \},  \no \\
		 &= \f{1}{t_0^2}\sum_{i=1}^m \left \{ \ell_i v'_{t_0,i}(\ell_i)^2 -  \ell_i u'_{t_0,i}(\ell_i)^2 \right \} - \f{\al}{t_0} \left ( u_{t_0,1}(0)^2 - v_{t_0,1}(0)^2\right ). \label{term1}
	\end{align}
	For the second term in \eqref{t_cross_form_NLS}, from the definition of $\T_t$, we have
	\begin{equation*}
		\dot \T_{t_0} \mathbf{u}_{t_0} = (0, \dot \Gamma_{1,t_0} u_{t_0}, 0 , \dot \Gamma_{1,t_0} v_{t_0})^\top,
	\end{equation*}
	where, for example,
	\begin{equation*}
		\dot \Gamma_{1,t_0} u_{t_0} = -\f{1}{t_0^2}\left ( u_{t_0,1}'(0), \dots, u_{t_0,m}'(0), -u_{t_0,1}'(\ell_1),\dots, -u_{t_0,m}'(\ell_m)  \right ).
	\end{equation*}
	Hence, from the definition of $\Om$ in \eqref{Om}, and with $\T_{t_0}\mathbf{u}_{t_0} = (\Gamma_{0,t_0} u_{t_0},\Gamma_{1,t_0} u_{t_0}, \Gamma_{0,t_0} v_{t_0}, \Gamma_{1,t_0} v_{t_0})^\top$, we have
	\begin{align}\label{term2}
		 \Om \left (\T_{t_0}\mathbf{u}_{t_0}, \dot{\T}_{t_0}\mathbf{u}_{t_0}\right )
		 &= - \left \langle \Gamma_{0,t_0} u_{t_0} ,\dot \Gamma_{1,t_0} u_{t_0} \right \rangle_{\R^{2m}} + \left \langle \Gamma_{0,t_0} v_{t_0} ,\dot \Gamma_{1,t_0} v_{t_0} \right \rangle_{\R^{2m}}.
	\end{align}
	Again using the $\delta$-type condition at the central vertex, we find 
	\begin{align}\label{term3}
		 \langle \Gamma_{0,t_0} u_{t_0} ,\dot \Gamma_{1,t_0} u_{t_0} \rangle_{\R^{2m}} &= -\f{1}{t_0^2} \sum_{i=1}^m u_{t_0,i}(0) u'_{t_0,i}(0) = - \f{u_{t_0,1}(0)}{t_0^2} \sum_{i=1}^m u'_{t_0,i}(0) = -\f{\al}{t_0} u_{t_0,1}(0)^2.
	\end{align}
	Similarly, 
	\begin{equation}\label{term4}
	\left 	\langle \Gamma_{0,t_0} v_{t_0} ,\dot \Gamma_{1,t_0} v_{t_0} \right \rangle_{\R^{2m}} = -\f{\al}{t_0} v_{t_0,1}(0) ^2.
	\end{equation}
	Combining equations \eqref{term1} and \eqref{term2}--\eqref{term4}, \eqref{t_cross_form_NLS} reduces to 
	\begin{align}\label{t_crossing_form_star_NLS}
		\mathfrak{m}_{t_0}(\mathbf{q},\mathbf{q})&=  \f{1}{t_0^2}\sum_{i=1}^m \ell_i \left ( v'_{t_0,i}(\ell_i)^2 -  u'_{t_0,i}(\ell_i)^2 \right ).
	\end{align} 
	On the other hand, the crossing form \eqref{la_crossing_form} for the path $\mu\mapsto \Upsilon_{\mu,t_0}$ with respect to $\mathfrak{D}$ is simply
	\begin{align}\label{la_cross_form_NLS}
	\mathfrak{m}_{\mu_0}(\mathbf{q},\mathbf{q}) &= - \left \langle \tau \mathbf{u}_{t_0}, \mathbf{u}_{t_0} \right \rangle_{\cH^2} = -2 \left \langle u_{t_0}, v_{t_0}\right \rangle_\cH.
	\end{align}
	
	We may now use \cref{thm:Hadamard} to write down Hadamard-type formulas for the eigenvalue curves. Under the assumption that $\mu_0\in\spec(\cN_{t_0}+V_{t_0})$ is simple, we have $\mathfrak{m}_{\mu_0}\neq0$. Therefore, by \cref{thm:Hadamard} (or \cref{thm:Hadamard_via_resolvents}) and recalling that $\mu = t^2 \la$, there exists a $C^1$ curve $\la(t)$, defined for $t$ near $t_0$, which satisfies $t^2 \la(t) \in \spec(\cN_t + V_t)$ and $\la(t_0)=\la_0$. Moreover, $\de{}{t} (t^2\la(t))|_{t=t_0} = -\mathfrak{m}_{t_0}/ \mathfrak{m}_{\mu_0}$, so by the chain rule we have
	\begin{equation}\label{ladash}
		\la'(t_0) = \f{1}{t_0^2} \de{}{t} (t^2\la(t))\big|_{t=t_0} - \f{2\la_0}{t_0} = \f{ \sum_{i=1}^m \ell_i \left ( v'_{t_0,i}(\ell_i)^2 -  u'_{t_0,i}(\ell_i)^2 \right )- 4\,t_0^3\la_0\,\langle u_{t_0}, v_{t_0}\rangle_\cH }{ 2\,t_0^4 \,\langle u_{t_0}, v_{t_0}\rangle_\cH }.
	\end{equation}
	If the numerator in the right-hand side of \eqref{ladash} is nonzero, the same eigenvalue curve can be written as a graph over $\lambda$. Indeed, applying \cref{thm:Hadamard} in $(\mu,t)$ coordinates and again using the chain rule yields a locally defined $C^1$ curve $t(\la)$ such that $t(\la)^2 \la \in \spec(\cN_{t(\la)}+V_{t(\la)})$ and $t(\la_0)=t_0$, whose derivative $t'(\la_0)$ is given by the reciprocal of \eqref{ladash}.

	Let us now recall a definition of the Maslov index via crossing forms, following \cite{RS93}. Let $s$ denote a general parameter, representing either $t$ or $\lambda$, and denote by $s \mapsto \Upsilon_s$ the path obtained from the Lagrangian subspace $\Upsilon_{\la,t}$ by varying one of $t$ or $\la$ and holding the other fixed. Let $s_0$ be a crossing, i.e. $\Upsilon_{s_0}\cap\fD \neq \{0\}$, and suppose moreover that $s_0$ is the only crossing in some interval $[a,b]$. If $s_0$ is \emph{regular}, i.e. $\mathfrak{m}_{s_0}$ is nondegenerate, the Maslov index is defined locally as follows,
	\begin{align}\label{Maslov_index_defn}
		\Mas(\Upsilon_{s},\fD: s\in[a,b]) \coloneqq \begin{cases}
			-n_-(\mathfrak{m}_{s_0}) & s_0 = a, \\
			n_+(\mathfrak{m}_{s_0}) - n_-(\mathfrak{m}_{s_0}) & a<s_0<b, \\
			n_+(\mathfrak{m}_{s_0}) & s_0 =b,
		\end{cases}
	\end{align}
	where $n_+(\mathfrak{m}_{s_0})$, resp. $n_-(\mathfrak{m}_{s_0})$, is the number of positive, resp. negative, squares of the quadratic form $\mathfrak{m}_{s_0}$. The Maslov index of the path $\mathscr{J}\ni s\mapsto \Upsilon_{s}$ is then obtained by summing the Maslov indices over all crossings $s_0\in\mathscr{J}$.  Importantly, the Maslov index is invariant under fixed-endpoint homotopies of the path, and using this property one can extend the definition to \emph{all} Lagrangian paths (i.e. those with non-regular crossings).

	Next, we describe the Lagrangian path used in the proof of \cref{thm:spectral_index_thm} and record some preliminary results. Recall that $\K_{\la,t}$ is the Cauchy data plane defined in \eqref{cKcFcD}, and $\mathfrak{D}\subset \fH^4\times \fH^4$ is the diagonal plane in $\fH^4\times \fH^4$. In what follows, we will exploit \eqref{Lag_intersect_spectral_interp_iff}, i.e. that 
	\begin{equation}\label{Lag_ints_kernel}
		\ker(\cN_t +V_t - t^2\la) \neq \{0\} \,\,\,\iff \,\,\, \Upsilon_{\la,t}\cap \mathfrak{D}\neq \{0\}
	\end{equation} 
	(recalling that $t^2\la$ is the spectral parameter). Thus, crossings $(\la_0,t_0)$ can be characterised in terms of the non-triviality of the kernel of the associated differential operator.
	
	The relevant Lagrangian path will be
	\begin{equation}\label{lag_path_stargraph}
		\Gamma \ni (\la,t) \mapsto \Upsilon_{\la,t} \coloneqq \K_{\la,t} \oplus (\cL\oplus \cL)
	\end{equation} 
	over the contour $\Gamma\coloneqq \Gamma_1\cup \Gamma_2\cup \Gamma_3\cup \Gamma_4$ in the $\la t$-plane, oriented clockwise, where the segments $\Gamma_i$ are given by
	\begin{align}\label{Gammas_contour}
		\begin{split}
		&\Gamma_1: \la\in[0,\la_\infty], \quad t=\e_0; \qquad  \Gamma_2: \la=0, \quad \,\,\,\,\,t\in[\e_0,1]; \\
		&\Gamma_3: \la\in[0,\la_\infty], \quad t=1;\qquad \hspace{1.4mm} \Gamma_4: \la=\la_\infty, \quad 	t\in[\e_0,1],
		\end{split}
	\end{align}
	and where $\e_0$ and $\la_\infty$ are described below. The proof of the lower bound \eqref{lower_bound_pos_real_eigvals} is obtained by homotoping the path described by \eqref{lag_path_stargraph}--\eqref{Gammas_contour} to a fixed point -- the Maslov index of which is zero -- and analysing the Maslov index along each segment $\Gamma_i$. 
	
	In \eqref{Gammas_contour}, $0<\e_0\ll 1$ and $\la_\infty$ is taken large enough so that 
	\begin{equation}\label{N_upper_bound}
		\spec(\cN_t + V_t)\cap \{z\in\C : {\rm{Re}}\,z \geq \la_\infty\}=\emptyset \quad \text{ for all} \quad  t\in[\e_0,1].
	\end{equation}
	Indeed such a choice is possible, since for each $t\in[\e_0,1]$, $\cN_t +V_t$ is a bounded perturbation of a skew-adjoint operator; its spectrum is therefore contained in the vertical strip $\left \{z\in\C: |{\rm{Re}}(z)| \leq \|V_t\|_{\cB(\cH^2)} \right \}$ around the imaginary axis in the complex plane \cite{Kato}. Taking the supremum of the spectral bounds over $t\in[\e_0,1]$ yields the claim. 

	Choosing $\la_\infty$ large enough so that \eqref{N_upper_bound} holds ensures that there are no crossings along $\Gamma_4$. It can be similarly shown that there are no crossings along $\Gamma_1$. Indeed, if $\e_0>0$ is small enough, we claim that $\spec(\cN_{\e_0} +V_{\e_0})\cap \R = \emptyset$.
	{   To prove the claim, we first note that the operators $\cA_{\e_0} + V_{-,\e_0}$ and $\cA_{\e_0} + V_{+,\e_0}$ with domain $\dom(\cA_{\e_0}) = \{u\in\cH_+ : \tr_{\e_0} u \in \cL\}$ are strictly positive. To see this, the quadratic form associated with $\mathcal{A}_{\varepsilon_0}$ is
	\begin{equation}\label{quad_form_A}
		\mathfrak{a}_{\varepsilon_0}[u] \coloneqq \langle \cA_{\e_0} u,u \rangle_{L^2(\cG)} 
		=  \sum_{i=1}^m \|u_i'\|_{L^2(0,\ell_i)}^2 +\varepsilon_0\alpha |u(0)|^2 
		= \|u'\|_{L^2(\cG)}^2 +\varepsilon_0\alpha |u(0)|^2.
	\end{equation}
	Since $u_i(\ell_i)=0$, by the Cauchy-Schwarz inequality we have
	\begin{equation*}
		|u(0)|^2 = |u_i(0)|^2 =
		\left|\int_0^{\ell_i} u_i'(x)\,dx\right|^2
		\leq
		\ell_i \|u_i'\|_{L^2(0,\ell_i)}^2
	\end{equation*}
	for each $i$, hence
	\[
	|u(0)|^2 \leq \ell_{\min}  \|u'\|_{L^2(\cG)}^2,
	\qquad \ell_{\min}\coloneqq\min_{1\leq i\leq m}\ell_i .
	\]
	Using the previous inequality in \eqref{quad_form_A} yields
	\begin{equation}\label{quad_form_bound}
		\mathfrak{a}_{\varepsilon_0}[u]
		\geq
		\left(1-\varepsilon_0|\alpha|\ell_{\min}\right) \|u'\|_{L^2(\cG)}^2.
	\end{equation}
	By the Poincar\'e inequality we have
	\begin{equation}
		\| u \|^2_{L^2(\cG)} \leq \ell_{\max}^2 \|u'\|_{L^2(\cG)}^2,
		\qquad  \ell_{\max}\coloneqq\max_{1\leq i\leq m}\ell_i, 
	\end{equation}
	and using this in \eqref{quad_form_bound} shows that for sufficiently small $\e_0>0$ we have $\mathcal{A}_{\varepsilon_0}\geq C $ for some $C>0$. Using that $V_{+,\varepsilon_0}=-\varepsilon_0^2 V_+(\varepsilon_0 x)$, we have
	\[
	\begin{aligned}
		\langle(\mathcal{A}_{\varepsilon_0}+V_{+,\varepsilon_0})u,u\rangle_{L^2(\mathcal{G})}
		&\geq
		C\|u\|_{L^2(\mathcal{G})}^2
		-
		\varepsilon_0^2\| V_+ \|_{L^\infty(\mathcal{G})}\|u\|_{L^2(\mathcal{G})}^2 >
		c_+\|u\|_{L^2(\mathcal{G})}^2,
	\end{aligned}
	\]
	for some $c_+>0$ provided $\e_0>0$ is sufficiently small. }A similar inequality holds for $\cA_{\e_0} + V_{-,\e_0}$. 
	
	Next, suppose, by way of contradiction, that there exists $\e_0^2 \la \in \spec(\cN_{\e_0}+V_{\e_0})\cap \R$ with eigenfunction $\mathbf{u}_{t_0} = (u_{t_0}, v_{t_0}) \in \dom(\cN_{\e_0})$. Using the fact that $\cA_{\e_0} + V_{+,\e_0}$ is strictly positive and arguing as in \cite[Lemma 3.21]{CCLM23} we note that the eigenvalue problem \eqref{Nt_EVP_NLS} for $t=\e_0$ small enough is equivalent to the self-adjoint eigenvalue problem
	\begin{equation}\label{SA_EVP}
		(\cA_{\e_0} + V_{-,{\e_0}})^{1/2}(\cA_{\e_0} + V_{+,{\e_0}})(\cA_{\e_0} + V_{-,{\e_0}})^{1/2} w_{\e_0} = -\e_0^4\la^2 w_{\e_0},
	\end{equation}
	where $w_{\e_0} \in \dom (\cA_{\e_0} + V_{-,{\e_0}})^{1/2}$ is such that $(\cA_{\e_0} + V_{-,{\e_0}})^{1/2} w_{\e_0} \in \dom (\cA_{\e_0} + V_{+,{\e_0}})$ and $(\cA_{\e_0} + V_{+,{\e_0}})(\cA_{\e_0} + V_{-,\e_0})^{1/2} w_{\e_0} \in \dom(\cA_{\e_0} + V_{-,{\e_0}})$. Applying $\langle \cdot, w_{\e_0}\rangle$ to \eqref{SA_EVP}, using self-adjointness of $(\cA_{\e_0}+ V_{-,{\e_0}})^{1/2}$ and positivity of $\cA_{\e_0} + V_{-,{\e_0}}$ and $\cA_{\e_0} + V_{+,{\e_0}}$, for positive constants $C_+, C_-$ we find
	\begin{equation*}
		0 \geq -\e_0^4\la^2 \|w_{\e_0}\|^2 =\left  \langle (\cA_{\e_0} + V_{+,\e_0})(\cA_{\e_0} + V_{-,{\e_0}})^{1/2} w_{\e_0}, (\cA_{\e_0} + V_{-,{\e_0}})^{1/2}w_{\e_0}\right \rangle_{\cH} \geq C_+ C_- \|w_{\e_0}\|^2 >0,
	\end{equation*}
	a contradiction. We conclude that 
	\begin{equation}\label{no_crossings_bottom}
		\spec(\cN_{\e_0} + V_{\e_0})\cap \R = \emptyset,
	\end{equation}
	as desired. 
	
	Having eliminated the possibility of crossings along $\Gamma_1$ and $\Gamma_4$, it remains to analyse the Maslov indices along $\Gamma_2$ and $\Gamma_3$. 
	
	\begin{proof}[Proof of \cref{thm:spectral_index_thm}]
	Under \cref{hypo:NLS_simplicity}, we have $\dim\ker(\cN+V)=1$ with corresponding eigenfunction $\mathbf{u}=(0,\phi)$. Computing the $t$ and $\la$ crossing forms at $(\la_0,t_0)=(0,1)$ yields
	\begin{equation}\label{crossing_forms_NLS_star_graph}
		\mathfrak{m}_{t_0} =  \sum_{i=1}^m \ell_i  \phi'_{t_0,i}(\ell_i)^2 \neq0, \qquad  \mathfrak{m}_{\la_0}=0.
	\end{equation}
	(These are precisely the forms \eqref{t_crossing_form_star_NLS} and \eqref{la_cross_form_NLS} expressed in $(\la,t)$ coordinates via $\mu=t^2\la$ at $(\la,t)=(0,1)$.) Hence, as described in the paragraph following equation \eqref{ladash}, applying \cref{thm:Hadamard} and the chain rule shows that there exists a $C^1$ curve $t(\la)$ satisfying \eqref{spectral_set_t_graph_of_lambda}, and such that $t(0)=1$ and $t'(0)=0$. {   That $t(\la)$ is in fact $C^2$ follows from the implicit function theorem and the $C^2$ regularity of the mapping $(\la,t)\mapsto M(\la,t)$ for $(\la,t)$ near $(0,1)$\footnote{\, This regularity, in turn, follows from the $C^2$ regularity of the mappings $t\mapsto V_t$ for $V_t$ given by \eqref{Ntop2}, $t\mapsto \T_t$ for $\T_t$ given by \eqref{Trace_t_NLS}, and $t\mapsto W(\la,t)$ for $W$ given by \eqref{eq:matrixW_lam_t}.}, where $M$ is given by \eqref{eq:matrixM} (or \eqref{M_la_t}) with $n=1$.}
	
	Turning to the lower bound \eqref{lower_bound_pos_real_eigvals}, by homotopy invariance and additivity under concatenation we have
	\begin{equation}
		\Mas(\Upsilon_{\la,t}, \fD; \Gamma_1) + \Mas(\Upsilon_{\la,t}, \fD; \Gamma_2) + \Mas(\Upsilon_{\la,t}, \fD; \Gamma_3)+ \Mas(\Upsilon_{\la,t}, \fD; \Gamma_4) = 0.
	\end{equation}
	From the choice of $\la_\infty$ given by \eqref{N_upper_bound}, it follows that $\Mas(\Upsilon_{\la,t}, \fD; \Gamma_4)=0$. In addition, for our choice of $\e_0>0$ small, by \eqref{no_crossings_bottom} we have $\Mas(\Upsilon_{\la,t}, \fD; \Gamma_1)=0$. Again using the concatenation property, it follows that
	\begin{align}\label{Mas_concat_2}
		\Mas(\Upsilon_{0,t}, \fD; t\in[\e_0, 1-\e]) + \mathfrak{c} +  \Mas(\Upsilon_{\la,1}, \fD; \la\in[\e,\la_\infty]) = 0,
	\end{align}
	for some $\e>0$, where 
	\begin{align}\label{define_c}
		\mathfrak{c} \coloneqq  \Mas(\Upsilon_{0,t}, \fD; t\in[1-\e, 1]) + \Mas(\Upsilon_{\la,1}, \fD; \la\in[0,\e]).
	\end{align}
	Choosing $\e>0$ small enough guarantees that $\mathfrak{c}$ represents the contribution to the Maslov index from the corner crossing $(\la_0,t_0)=(0,1)$ only. This follows from the fact that $t_0=1$ is an isolated crossing of $\Gamma_2$ (due to nondegeneracy of $\mathfrak{m}_{t_0}$ at that point) and that $\la_0=0$ is an isolated crossing of $\Gamma_3$ (due to $\cN$  having compact resolvent). 
	
	The issue with computing $\mathfrak{c}$ directly is that $\la_0=0$ is a non-regular crossing of $\Gamma_3$ (since $\mathfrak{m}_{\la_0}=0$) and hence \eqref{Maslov_index_defn} is not applicable.
	To remedy this, we locally homotope the Lagrangian path to one with only regular crossings, exploiting the fact that we can readily compute the signature of $\mathfrak{m}_{t_0}$ when $\la=0$. To that end, note that by \eqref{Lag_ints_kernel}, the eigenvalue curve $t(\la)$ satisfying \eqref{spectral_set_t_graph_of_lambda} represents the locus of points through $(\la,t)=(0,1)$ such that $\Upsilon_{\la,t(\la)}\cap \fD \neq\{0\}$. Assuming $t''(0)\neq0$, we homotope the path according to whether $t''(0)>0$ or $t''(0)<0$, see \cref{fig:homotope}.

		\begin{figure}
		\centering
		\hspace*{\fill}
		\subcaptionbox{\label{homotope1}} 
		{\includegraphics[width=0.2\textwidth]{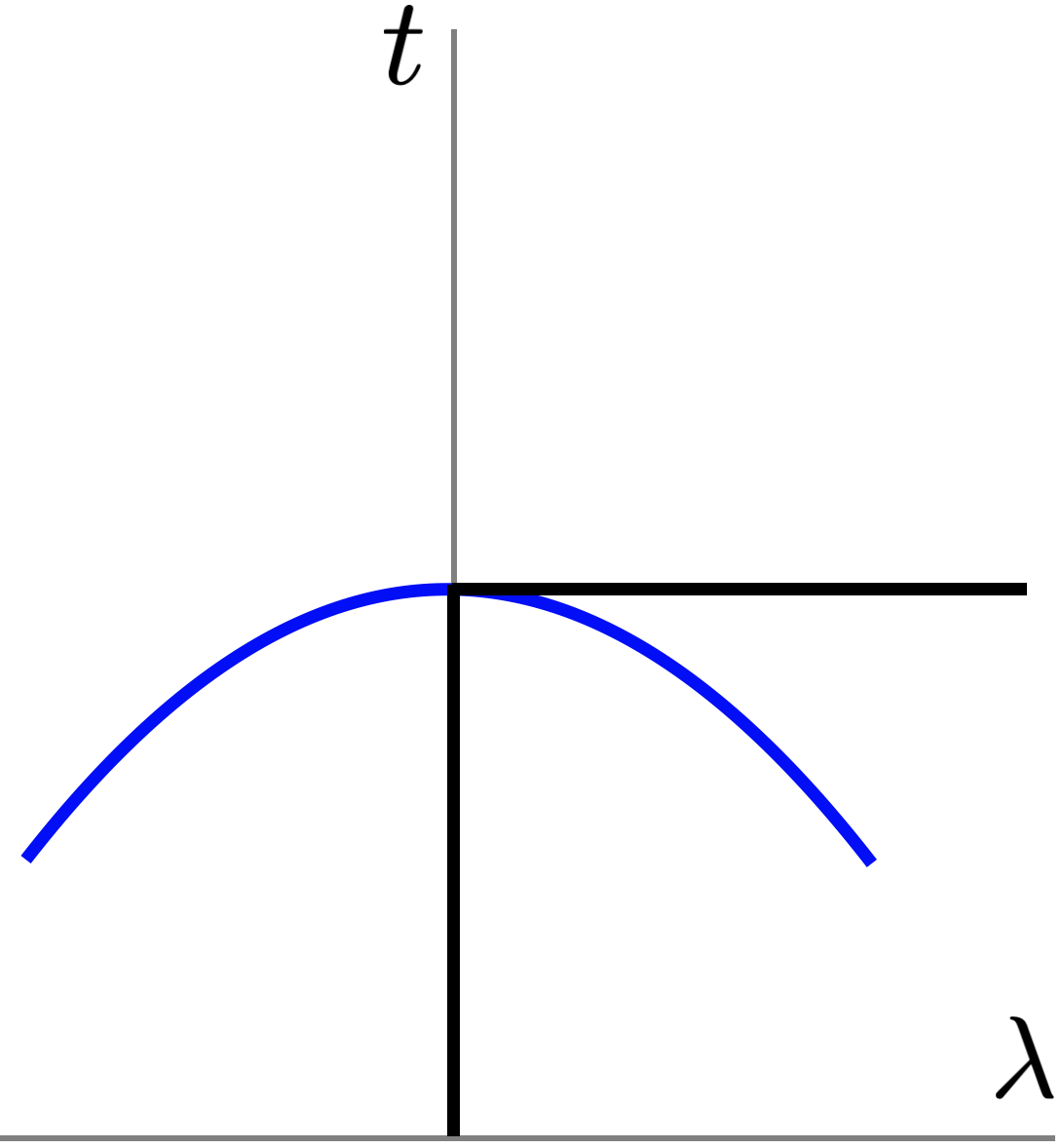}}\hfill 
		\subcaptionbox{\label{homotope2}}
		{\includegraphics[width=0.2\textwidth]{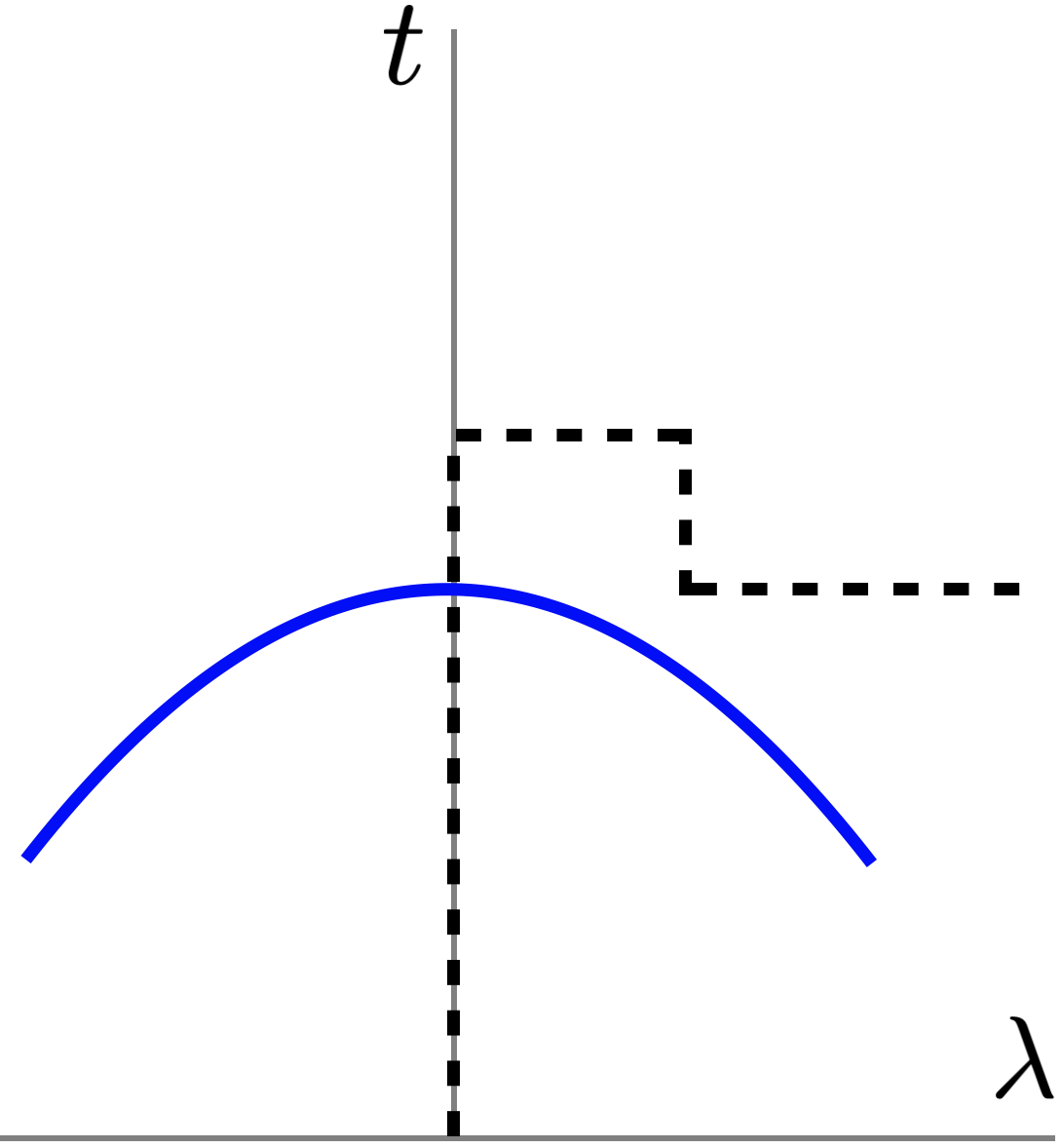}} \hspace{5mm}
		\hspace*{\fill} 
		\subcaptionbox{\label{homotope3}} 
		{\includegraphics[width=0.2\textwidth]{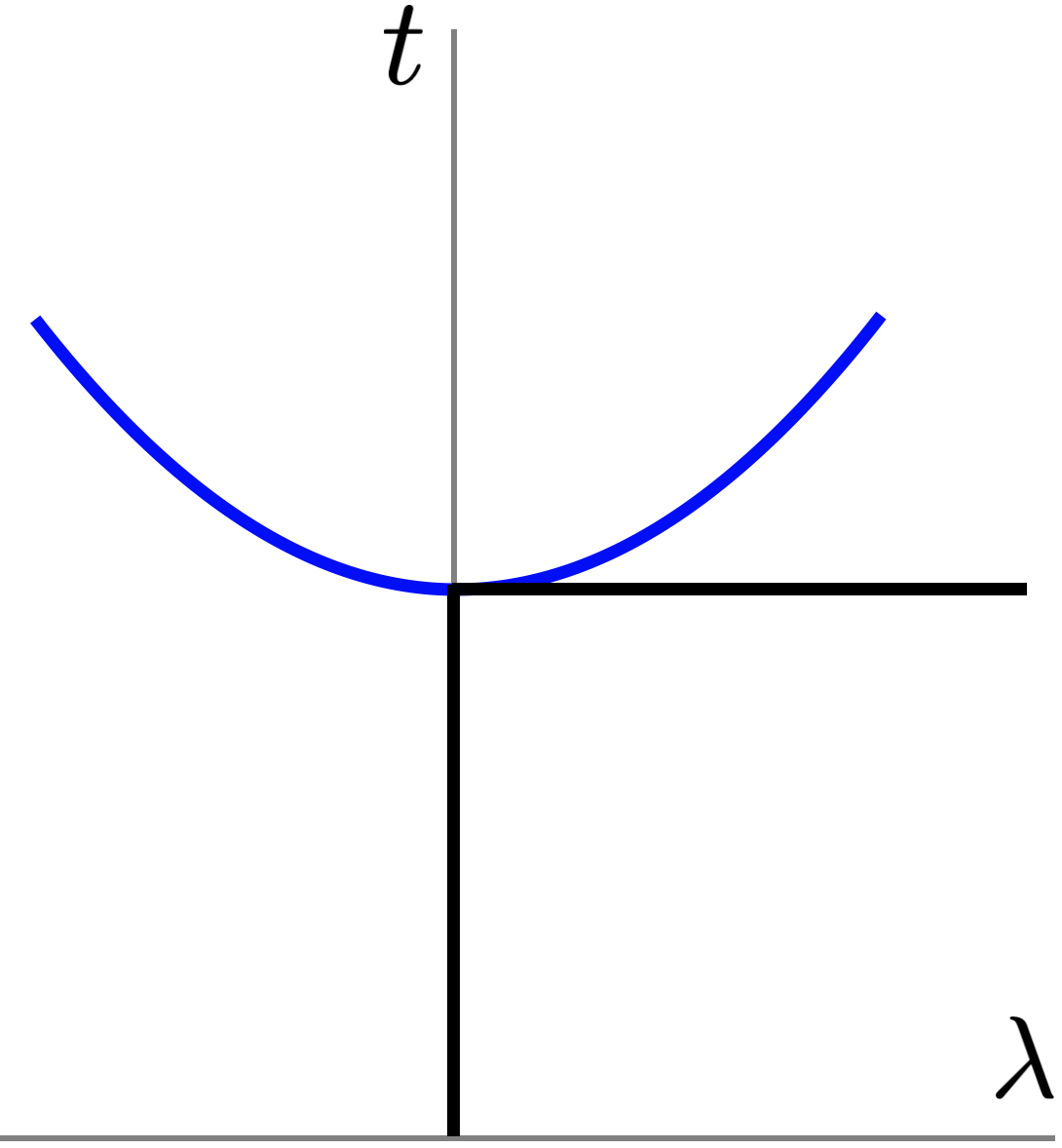}}\hfill 
		\subcaptionbox{\label{homotope4}}
		{\includegraphics[width=0.2\textwidth]{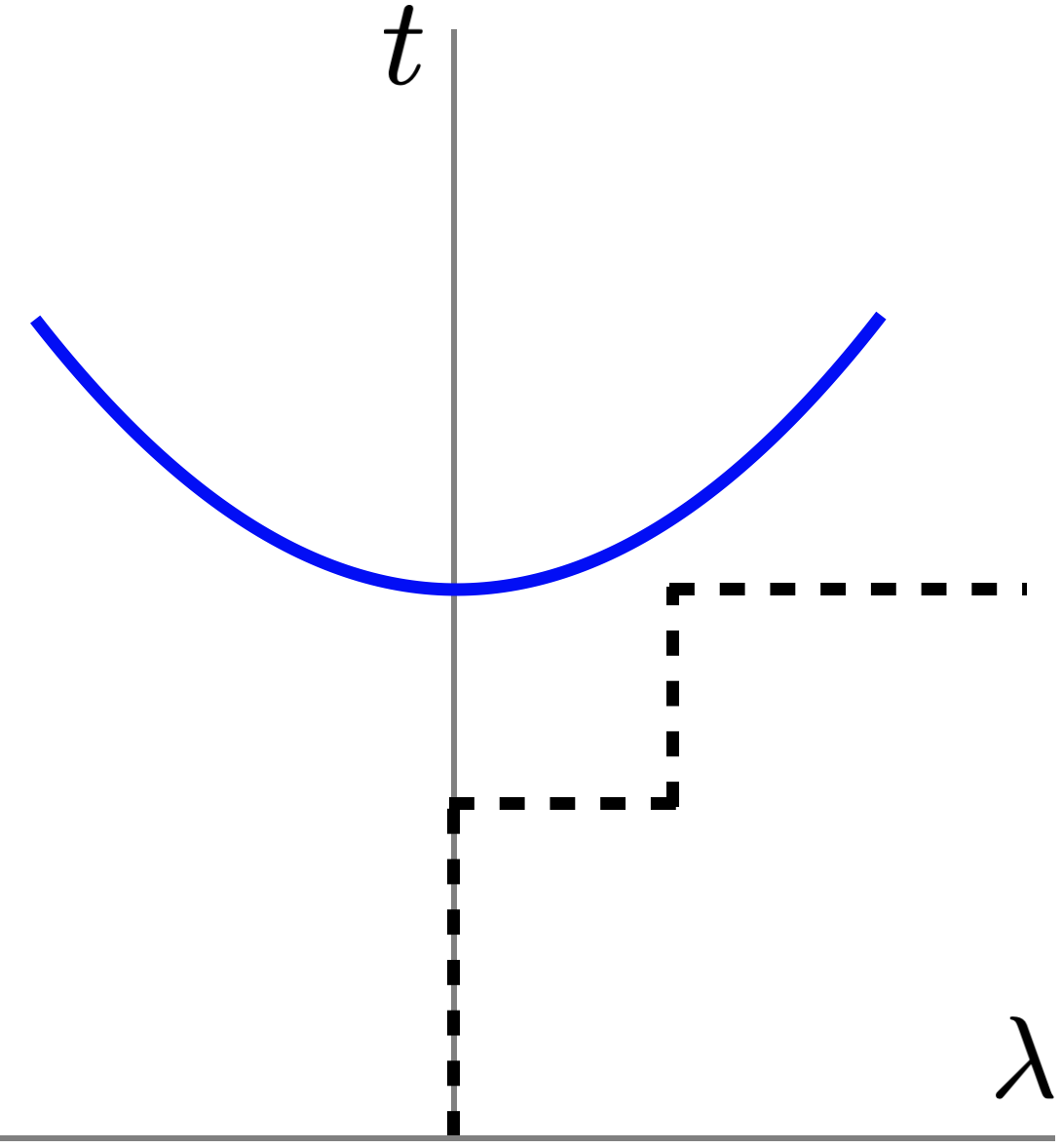}} 
		\hspace*{\fill} 
		\caption{Schematic of the Lagrangian path (solid black) through the top left corner crossing $(\la,t)=(0,1)$ when the eigenvalue curve $t(\la)$ (blue) satisfies (a) $t''(0)<0$ and (c) $t''(0)>0$, and the new path (dashed black) to which we homotope the original path when (b) $t''(0)<0$ and (d) $t''(0)>0$. The new paths have only regular crossings.}
		\label{fig:homotope}
	\end{figure}
	
	When $t''(0)<0$, we homotope the path to one with a single regular crossing in the $t$ direction when $\la=0$, see \cref{homotope1,homotope2}. (Note that the homotoped path is indeed well-defined for $t>1$, see \cref{rem:t_greater_than_one}.) In this case, from \eqref{crossing_forms_NLS_star_graph} we see that $n_+(\mathfrak{m}_{t_0})=1$. Hence by \eqref{Maslov_index_defn}, $\mathfrak{c}=1$. On the other hand, when $t''(0)>0$ we homotope the path to one having no crossings at all, see \cref{homotope3,homotope4}. In this case $\mathfrak{c}=0$. In summary,
	\begin{align}\label{compute_c}
		\mathfrak{c} = \begin{cases}
			+1  &t''(0) < 0, \\
			0  &t''(0) > 0.
		\end{cases}
	\end{align}
	Along $\Gamma_3$ ($\la>0$), the crossing form \eqref{la_cross_form_NLS} is not signed, and hence the Maslov index does not count the number of crossings. Nonetheless, the absolute value provides a lower bound for this count, and observing that crossings along $\Gamma_3$ are positive real eigenvalues of $\cN+V$, we have
	\begin{equation}\label{Gamma3_lowerbound}
		n_+(\cN+V) \geq \big| \Mas(\Upsilon_{\la,1}, \fD; \la\in[\e,\la_\infty]) \big|.
	\end{equation}
	For the crossings along $\Gamma_2$, suppose that $t_0 \in (\e_0,1-\e]$ is a crossing with $\mathbf{u}_{t_0} \in \ker(\cN_{t_0} +V_{t_0})$. Observe that when $\la_0=0$ the eigenvalue equations decouple into two independent equations:
	\begin{equation}
		(\cN_{t_0} + V_{t_0})\mathbf{u} = 0 \quad \iff \quad \begin{cases*}
			(\cA_{t_0} + V_{-,{t_0}}) v_{t_0} = 0, \\
			(\cA_{t_0}+ V_{+,{t_0}}) u_{t_0} = 0.
		\end{cases*}
	\end{equation}
	The signature of the $t$-crossing form \eqref{t_crossing_form_star_NLS} is then readily determined by the dimensions of the kernels of $\cA_{t_0} + V_{-,{t_0}}$ and $\cA_{t_0} + V_{+,{t_0}}$. In particular, we find that:
	\begin{align*}
		\sgn \mathfrak{m}_{t_0} = \begin{cases}
			+\dim\ker(\cA_{t_0} + V_{-,{t_0}}), & 0\in\spec(\cA_{t_0} + V_{-,{t_0}}) \backslash \spec(\cA_{t_0} + V_{+,{t_0}}), \\
			-\dim\ker(\cA_{t_0} + V_{+,{t_0}}), & 0\in\spec(\cA_{t_0} + V_{+,{t_0}}) \backslash \spec(\cA_{t_0} + V_{-,{t_0}}), \\
			 \dim\ker(\cA_{t_0} + V_{-,{t_0}}) -\dim\ker(\cA_{t_0} + V_{+,{t_0}}),  & 0\in \spec(\cA_{t_0} + V_{-,{t_0}})\cap \spec(\cA_{t_0} + V_{+,{t_0}}).
		\end{cases}
	\end{align*}
	Summing the signatures over all crossings $t_0 \in[\e_0, 1-\e] \subset \Gamma_2$ (excluding the corner crossing $(\la,t)=(0,1)$), and recalling the definitions of $p_c$ and $q_c$ in \eqref{pcqc}, it follows that
	\begin{align}
			\Mas(\Upsilon_{0,t}, \fD; t\in[\e_0, 1-\e]) 
			&= q_c - p_c. \label{mas_gamma2}
	\end{align}
	Now collecting \eqref{Gamma3_lowerbound} and \eqref{Mas_concat_2} together, we obtain
	\begin{equation}
		n_+(\cN+V) \geq \big |\Mas(\Upsilon_{0,t}, \fD; t\in[\e_0, 1-\e]) + \mathfrak{c} \big |,
	\end{equation}
	and using \eqref{mas_gamma2} and \eqref{compute_c}, we arrive at \eqref{lower_bound_pos_real_eigvals}.
	\end{proof}

	{\bf Conflict of interest.} The authors declare no conflict of interest.
	
	{\bf Data availability.} The data that support the findings of this study are available from the corresponding author upon reasonable request.

	\parskip=0em
	
	\bibliographystyle{alpha}
	\bibliography{mybib}

\end{document}